\documentclass[a4paper, 11pt, reqno]{amsart}
\usepackage{amsmath,amsthm,amssymb,amscd}
\usepackage[english]{babel}
\usepackage[ansinew]{inputenc}
\usepackage[T1]{fontenc}

\usepackage[all]{xy}
\usepackage{xspace}
\usepackage{mathrsfs}
\usepackage[dvips]{color}
\usepackage{epsfig}
\usepackage{float}
\usepackage{wasysym}
\usepackage{setspace} 
\usepackage{enumerate}

\usepackage{hyperref}

\newcommand{\diam}{\mathrm{diam}}
\newcommand{\Mod}{\mathrm{Mod}}
\newcommand{\rhostar}{\rho^{*}}

\newcommand{\N}{{\mathbb N}}

\newcommand{\R}{{\mathbb R}}
\newcommand{\C}{{\mathbb C}}

\newcommand{\be}{\begin{enumerate}}
\newcommand{\ee}{\end{enumerate}}

\numberwithin{figure}{section}

\newtheorem{theorem}{Theorem}[section]
\newtheorem{lemma}[theorem]{Lemma}
\newtheorem{proposition}[theorem]{Proposition}
\newtheorem{corollary}[theorem]{Corollary}
\newtheorem{definition}[theorem]{Definition}

\newenvironment{remark}[1][Remark.]{\begin{trivlist}
\item[\hskip \labelsep \textsc{#1}]}{\end{trivlist}}

\makeatletter

\@addtoreset{equation}{section}
\makeatother

\title{On the conformal gauge of a compact metric space}
\author{Matias Carrasco Piaggio}
\email{matias@cmi.uni-mrs.fr}
\address{Laboratoire d'Analyse, Topologie et Probabilit\'es, Aix-Marseille Universit\'e, France.}

\begin{document}

\maketitle
\begin{abstract}
In this article we study the Ahlfors regular conformal gauge of a compact metric space $(X,d)$, and its conformal dimension $\dim_{AR}(X,d)$. Using a sequence of finite coverings of $(X,d)$, we construct distances in its Ahlfors regular conformal gauge of controlled Hausdorff dimension. We obtain in this way a combinatorial description, up to bi-Lipschitz homeomorphisms, of all the metrics in the gauge. We show how to compute $\dim_{AR}(X,d)$ using the critical exponent $Q_N$ associated to the combinatorial modulus. 
\end{abstract}
\begin{quote}
\footnotesize{\textsc{Keywords}: Ahlfors regular, conformal gauge, conformal dimension, combinatorial modulus, Gromov-hyperbolic.}
\end{quote}
\tableofcontents

\section{Introduction}
The subject of this article is the study of quasisymmetric deformations of a compact metric space. More precisely, let $(X,d)$ be a compact metric space, we are interested in its \emph{conformal gauge}:
\begin{equation*}\label{gauge}
\mathcal{J}\left(X,d\right):=\left\{\theta\text{ distance on }X: \theta\sim_{qs}d\right\},
\end{equation*}
where two distances in $X$, $d$ and $\theta$, are quasisymmetrically equivalent $d\sim_{qs}\theta$ if the identity map $id:(X,d)\to(X,\theta)$ is a quasisymmetric homeomorphism. Recall that a homeomorphism $h:(X, d)\to (Y, \theta)$ between two metric spaces is \emph{quasisymmetric} if there is an increasing homeomorphism $\eta: \R_+ \to \R_+$ ---called a distortion function--- such that:
$$\frac{\theta\left(h(x),h(z)\right)} {\theta\left(h(y),h(z)\right)}\leq\eta\left(\frac{d\left(x, z\right)} {d\left(y, z\right)}\right),$$
for all $x, y, z\in X$ with $y\neq z$. In other words, a homeomorphism is quasisymmetric if it distorts relative distances in a uniform and scale invariant fashion. This class of maps provides a natural substitute of quasiconformal homeomorphisms in the broader context of metric spaces. Their precise definition was given by Tukia and Väisälä in \cite{TV}. See \cite{He} for a detailed exposition of these notions.

For example, if $d$ is a distance in $X$, then $d^\epsilon$ is also a distance for all $\epsilon\in (0,1]$, and the identity map $id:(X, d)\to (X, d^\epsilon)$ is $\eta$-quasisymmetric with $\eta(t)=t^\epsilon$. In particular, $\dim_H (X, d^\epsilon) = \epsilon^{-1} \dim_H (X, d)$. Therefore, quasisymmetric homeomorphisms can distort the Hausdorff dimension of the space, and one can always find distances in the gauge of arbitrarily large dimension.

The conformal gauge encodes the quasisymmetric invariant properties of the space. A fundamental quasisymmetry numerical invariant is the conformal dimension; introduced by P.\,Pansu in \cite{Pa}. There are different related versions of this invariant, in this article we are concerned with the \emph{Ahlfors regular conformal dimension}, which is a variant introduced by M.\,Bourdon and H.\,Pajot in \cite{BP}.

A distance $\theta\in\mathcal{J}(X,d)$ is Ahlfors regular of dimension $\alpha>0$ ---AR for short--- if there exists a Radon measure $\mu$ on $X$  and a constant $K\geq 1$ such that:
\begin{equation*}\label{ar}
K^{-1}\leq\frac{\mu\left(B_r\right)}{r^\alpha}\leq K,
\end{equation*}
for any ball $B_r$ of radius $0<r\leq\diam_{\theta} X$. In that case, $\mu$ is comparable to the $\alpha$-dimensional Hausdorff measure and $\alpha=\dim_H(X,\theta)$ is the Hausdorff dimension of $(X,\theta)$. The collection of all AR distances in $\mathcal{J}(X,d)$ is the \emph{Ahlfors regular conformal gauge} of $(X,d)$, and is denoted by $\mathcal{J}_{AR}(X,d)$.

The AR conformal dimension measures the simplest representative of the gauge. It is defined by
\begin{equation*}\label{dimconf}
\dim_{AR}(X,d):=\inf\left\{\dim_H\left(X,\theta\right):\theta\in\mathcal{J}_{AR}(X,d)\right\}.
\end{equation*}
We write $\dim_{AR} X$ when there is no ambiguity on the metric $d$. Note that we always have the estimate $\dim_TX \leq \dim_{AR} X$, where $\dim_TX$ denotes the topological dimension of $X$. Apart from this, the AR conformal dimension is generally difficult to estimate. However, it was computed by P.\,Pansu for the boundaries of homogeneous spaces of negative curvature \cite{Pa}. An exposition of the theory of conformal dimension, its variants and its applications can be found in \cite{LP}, \cite{B}, \cite{K}, \cite{H2} and \cite{MT}.

The interest, in studying quasisymmetric invariants, comes from the strong relationship between the geometric properties of a Gromov-hyperbolic space and the analytical properties of its boundary at infinity. Quasi-isometries between hyperbolic spaces induce quasisymmetric homeomorphisms between their boundaries, so any quasisymmetric invariant gives a quasi-isometric one.

For hyperbolic groups, the understanding of the canonical conformal gauge of the boundary at infinity ---induced by the visual metrics--- is an important step in the approach by Bonk and Kleiner to the characterization problem of uniform lattices of $\mathrm{PSL}_2(\C)$, via their boundaries ---Cannon's conjecture \cite{BK1}. They showed that Cannon's conjecture is equivalent to the following: if $G$ is a hyperbolic group, whose boundary is homeomorphic to the topological two-sphere $S^2$, then the Ahlfors regular conformal dimension of $\partial G$ is attained. Motivated by Sullivan's dictionary, Haïssinsky and Pilgrim translated these notions to the context of branched coverings \cite{HP1}. In particular, the AR conformal dimension characterizes rational maps between CXC branched coverings (see \cite{HP1}).

Discretization has proved to be a useful tool in the study of conformal analytical objects in metric spaces. Different versions of combinatorial modulus have been considered, by several authors, in connection with Cannon's conjecture (see \cite{Can,BK3,H1}). The combinatorial modulus is a discrete version of the analytical conformal modulus from complex analysis, but unlike the latter, is independent of any analytical framework. It is defined using coverings of $X$; therefore, it depends only on the combinatorics of such coverings. In \cite{BK}, the authors proved several important properties of combinatorial modulus for approximately self-similar sets. By defining a combinatorial modulus of a metric space $(X,d)$ that takes into account all the ``annuli'' of the space, with some fixed radius ratio, we extend to a more general setting some of these properties. 

The two main results of the present paper are Theorem \ref{description} and Theorem \ref{maintheo}. The first, gives a combinatorial description of the AR conformal gauge from an appropriate sequence of coverings of the space. The second, shows how to compute the AR conformal dimension using a critical exponent associated to the combinatorial modulus. To state the theorems we need some definitions.

Given an appropriate sequence of finite coverings $\{\mathcal{S}_n\}_n$ of $X$, with 
\begin{equation}\label{diamtozero}
||\mathcal{S}_n||:=\max\left\{\diam B:B\in \mathcal{S}_n\right\}\to 0,\text{ as }n\to +\infty,
\end{equation}
we adapt a construction of Elek, Bourdon and Pajot \cite{BP,Elek}, and construct a geodesic hyperbolic metric graph $Z_d$ with boundary at infinity homeomorphic to $X$ (see Section \ref{chapjauge} for precise definitions). With this identification the distance $d$ becomes a visual metric on $\partial Z_d$. The vertices of the graph $Z_d$ are the elements of $\mathcal{S}:=\bigcup_n\mathcal{S}_n$, and the edges are of two types: vertical or horizontal. The vertical edges form a connected rooted tree $T$ ---which is a spanning tree of $Z_d$--- and the horizontal ones describe the combinatorics of intersections of the elements of $\mathcal{S}$, i.e. two vertices $B$ and $B'$ in the same $\mathcal{S}_n$ are connected by an edge if $\lambda\cdot B\cap \lambda\cdot B'\neq\emptyset$, where $\lambda$ is a large enough universal constant. We remark that one of the assumptions involving the elements of $\mathcal{S}$ is that they are ``almost balls'' (see \eqref{almostball1},\eqref{almostball2}). In particular, it makes sense to write $\lambda\cdot B$, and to talk about the center of $B$, for an element $B\in\mathcal{S}$ (see Section \ref{chapjauge}). 

The vertical edges connect an element of $\mathcal{S}_n$ with an element of $\mathcal{S}_m$ for $|n-m|=1$. All the edges of $Z_d$ are isometric to the unit interval $[0,1]$. We denote by $w$ the root of $T$, and $B\sim B'$ means that $B$ and $B'$ are connected by a horizontal edge. For each $n\geq 0$, we denote by $G_n$ the subgraph of $Z_d$ consisting of all the vertices in $\mathcal{S}_n$ with all the horizontal edges of $Z_d$ connecting two of them.

Consider a function $\rho:\mathcal{S}\to (0,1)$. This function can be interpreted as an assignment of ``new relative radius'' of the elements of $\mathcal{S}$, or as an assignment of ``new lengths'' for the edges of $Z_d$. For each element $B\in\mathcal{S}$, there exists a unique geodesic segment in $Z_d$ which joins the base point $w$ and $B$; it consists of vertical edges and we denote it by $[w,B]$. The ``new radius'' of an element $B\in \mathcal{S}$ is expressed by the function $\pi:\mathcal{S}\to (0,1)$ given by
$$\pi(B):=\prod\rho(B'),$$
where the product is taken over all elements $B'\in \mathcal{S}\cap [w,B]$. Theorem \ref{description} says that from an appropriate function $\rho: \mathcal{S}\to (0,1)$ one can change the lengths of the edges of $Z_d$, and obtain a metric graph $Z_\rho$ quasi-isometric to $Z_d$. This graph admits a visual metric $\theta_\rho$, automatically in $\mathcal{J} _ {AR}(X, d)$, of controlled Hausdorff dimension. When $\rho$ goes through all the possible choices we get all the gauge $\mathcal{J}_ {AR}(X, d)$ up to bi-Lipschitz homeomorphisms.

To state the conditions on the function $\rho$ we need the following notation (see Section \ref{chapjauge}). For a path of edges in $Z_d$, $\gamma=\left\{(B_i, B_ {i +1})\right\}_{i=1}^{N-1}$ with $B_i\in \mathcal{S}$, we define the $\rho$-length by
$$L_\rho\left(\gamma\right)=\sum\limits_{i =1}^N\pi\left(B_i\right).$$
Let $\alpha>1$. For $x,y\in X$, by the assumption (\ref{diamtozero}), there exists a maximal level $m\in\N$ with the property that there exists an element $B\in \mathcal{S}_m$ with $x, y\in\alpha\cdot B$. We let
$$c_\alpha(x, y):=\left\{B\in S_m: x, y\in\alpha\cdot B\right\},$$ 
and we call it \emph{the center} of $x$ and $y$. We define $\pi\left(c_\alpha(x,y)\right)$ as the maximum of $\pi\left(B\right)$ for $B\in c_\alpha(x, y)$. We also define $\Gamma_n(x,y)$ as the family of paths in $Z_d$ that join two elements $B$ and $B'$ of $\mathcal{S}_n$, with $x\in B$ and $y\in B'$. We remark that the paths in $\Gamma_n(x,y)$ are not constrained to be contained in $G_n$. Finally, for an element $B\in \mathcal{S}_m$ and $n>m$, we denote by $D_n(B)$ the set of elements $B'$ in $\mathcal{S}_n$ such that the geodesic segment $[w,B']$ contains $B$.

The conditions to be imposed to the wight function $\rho$ are the following:
\begin{enumerate}
\item[(H1)] (Quasi-isometry) There exist $0<\eta_-\leq\eta_+<1$ so that $\eta_-\leq\rho(B)\leq\eta_+$ for all $B\in\mathcal{S}$.
\item[(H2)] (Gromov product) There exists a constant $K_0\geq 1$ such that for all $B,B'\in \mathcal{S}$ with $B\sim B'$, we have
$$\frac{\pi(B)}{\pi(B')}\leq K_0.$$
\item[(H3)] (Visual parameter) There exist $\alpha\in\left[2,\lambda/4\right]$ and a constant $K_1\geq 1$ such that for any pair of points $x,y\in X$, there exists $n_0\geq 1$ such that if $n\geq n_0$ and $\gamma$ is a path in $\Gamma_n(x,y)$, then
$$L_\rho\left(\gamma\right)\geq K_1^{-1} \cdot\pi\left(c_\alpha(x, y)\right).$$
\item[(H4)] (Ahlfors regularity) There exist $p>0$ and a constant $K_2\geq 1$ such that for all $B\in \mathcal{S}_m$ and $n>m$, we have
$$K_2^{-1}\cdot\pi(B)^p\leq\sum\limits_{B'\in D_{n}(B)}\pi\left(B'\right)^p\leq K_2\cdot\pi(B)^p.$$
\end{enumerate}
We obtain the following results.

\begin{theorem}[Combinatorial description of the gauge]\label{description} Let $(X,d)$ be a compact metric space such that $\mathcal{J}_{AR}(X,d)\neq\emptyset$. Suppose the function $\rho:\mathcal{S}\to(0,1)$ verifies the conditions (H1), (H2), (H3) and (H4). Then there exists a distance $\theta_\rho$ on $X$ quasisymmetrically equivalent to $d$ and Ahlfors regular of dimension $p$. Furthermore, the distortion function of $id: \left(X, d\right)\to\left(X,\theta_\rho\right)$ depends only on the constants $\eta_-,\eta_+, K_0$ and $K_1$, and $$\theta(x, y)\asymp\pi\left(c_\alpha(x, y)\right),$$ for all points $x,y\in X$. Conversely, any distance in the AR conformal gauge of $(X,d)$ is bi-Lipschitz equivalent to a distance built in that way.
\end{theorem}

The terminology used in naming the hypotheses of the theorem will be explained in Section \ref{chapjauge}. For instance, the condition stated in the hypothesis (H1) serves to prove that $Z_d$, with the new distance induced by $\rho$, is quasi-isometric to $G$. The other hypotheses are interpreted in the same way. We remark that this approach, construction of Ahlfors regular distances using the combinatorial modulus, had already been considered by J. Cannon in \cite{Can}. 

This combinatorial description is particularly adapted to work with the combinatorial modulus. Let $B$ be a ball in $X$, for $n$ large enough we define $\Gamma_n(B)$ to be the set of paths $\gamma$ of $G_{n}$, with vertices $\{B_i\}_{i=1}^N$, such that the center of $B_1$ belongs to $B$ and that of $B_N$ belongs to $X\setminus 2\cdot B$ (see Section \ref{dimqas} for a precise definition). We consider the set $\mathcal{R}_n(B)$ of all admissible weight functions $\rho:\mathcal{S}_{n}\to\R_+$; i.e. $\forall\ \gamma\in \Gamma_n(B)$ we have 
\begin{equation*}\label{admissible}
\ell_\rho(\gamma):=\sum_{i=1}^N\rho(B_i)\geq 1.
\end{equation*}
Let $p>0$. We define the $p$-\emph{combinatorial modulus} associated to the ball $B\subset X$ at scale $n$ as 
\begin{equation*}
\Mod_p(B,n):=\inf\limits_{\rho\in \mathcal{R}_n(B)}\left\{\sum\limits_{B'\in \mathcal{S}_{n}}\rho(B')^p\right\}.
\end{equation*}
That is, one minimizes the $p$-volume among all the admissible weight functions. From this combinatorial modulus, defined for the annuli associated to the balls of $X$, we define in Section \ref{dimqas} a combinatorial modulus $M_{p,n}$ that takes into account all these annuli. We are interested in the asymptotic behavior of $M_{p,n}$ as $n$ tends to infinity, and its dependence on $p$. We set $M_p:=\liminf_nM_{p,n}$. For fixed $p>0$, the sequence $\{M_{p,n}\}_n$ verifies a sub-multiplicative inequality (see Lemma \ref{lme:smultiplicative}). This allows us to define the \emph{critical exponent} $Q_N:=\inf\{p>0:M_p=0\}$. 
\begin{theorem}\label{maintheo}
Let $(X,d)$ be a compact metric space such that $\mathcal{J}_{AR}(X,d)\neq\emptyset$. Then the AR conformal dimension of $(X,d)$ is equal to the critical exponent $Q_N$.
\end{theorem}

Bruce Kleiner informed me in May 2009 that, inspired by the work of Keith and Laakso, he and Stephen Keith proved a similar (unpublished) result. In \cite{BK} it is stated and used in the self-similar case, see the Remark 1 after Corollary 3.13 therein. This was part of the motivation for working on these questions and they led me to a proof of Theorem \ref{maintheo} in this general setting.

We derive Theorem \ref{maintheo} from Theorem \ref{description}. The idea of the proof is the following: by definition, the combinatorial moduli $M_{p, n}$ tend to zero as $n$ tends to infinity for $p>Q_N$. Therefore, one can choose $n$ large enough, depending on the difference $p-Q_N$, so that $\Mod_p(B,n)$ is small for all the ``balls'' $B \in \mathcal{S}$. This gives some flexibility to change the optimal weight functions, so as to obtain a function $\rho: \mathcal{S} \to (0,1)$ which satisfies the conditions of the combinatorial description of the gauge given in Theorem \ref{description}. This gives an AR metric $\theta_\rho$ in $\mathcal{J}(X,d)$ of dimension $p$. The distortion of $ id: (X, d) \to (X, \theta_\rho)$ depends on $n$, and thus on the difference $p-Q_N$.

This result confirms that the combinatorics of the graph $Z_d$ contains all the information of the AR conformal gauge of $X$. It should be noted that it is true regardless of the topology of $X$, it just requires $\mathcal{J}_{AR}(X, d) $ to be non empty.

Let us discuss some important aspects of Theorem \ref{maintheo}. First, it relates the two \emph{a priori} different definitions of conformal dimension. The definition given here is due to Bourdon and Pajot \cite{BP}, and is better suited for analytical issues. Nevertheless, the original definition given in \cite{Pa} is closer to that of the critical exponent $Q_N$. For example, if $Z$ is a geodesic proper hyperbolic space, then $\beta_n:=\{\mho(x, R), x \in Z\setminus B_n\}$ ---where $B_n$ is the ball of radius $n$ centered at a base point $w\in Z$, and $\mho(x, R)$ is the shadow of the ball $B(x,R)$ projected from the point $w$--- defines a \emph{quasiconformal structure}, in the sense of Pansu, on $\partial Z$. Pansu associates a $p$-\emph{module grossier} to such a quasiconformal structure, and defines the conformal dimension as the infimum of $p> 1$ such that the $p$-module grossier of all ---non trivial--- connected subsets of $\partial Z$ is zero. From the theoretical point of view, Theorem \ref{maintheo} shows that these two approaches are actually equivalent. In relation to Pansu's definition, one advantage of the critical exponent $Q_N$ is its discrete nature, and the fact that $Q_N$ is computed only from the horizontal curves of $Z_d$.

Second, Theorem \ref{maintheo} enables to compute the AR conformal dimension when the combinatorics of the coverings $\mathcal{S}$ is not too complicated. Hopefully, this is the case in general when the space $X$ has good symmetry properties, such as the Sierpi\'nski carpet. Also, for this important fractal, the discrete nature of $Q_N$ could provide a numerical estimate of the AR conformal dimension. Theorem \ref{maintheo} also relates the AR conformal dimension to other quasisymmetric invariants, like the $\ell_p$-equivalence classes defined using the $\ell_p$-cohomology of the conformal gauge $\mathcal{J}(X,d)$ (see \cite{BouK}). In a forthcoming paper \cite{CaPi2} we give some applications of Theorem \ref{maintheo} to the boundary of hyperbolic groups.

The existence of curve families of positive analytical modulus is strongly related to the AR conformal dimension. This was already showed by J. Tyson \cite{Tys} who proved that if $(X,d)$ is AR of dimension $Q>1$ and admits a family of curves
of positive $Q$-analytical modulus, then $(X,d)$ attains its AR conformal dimension. Certainly more surprising, S. Keith and T. Laakso \cite{KL} showed that this condition is almost necessary in the following sense: if $(X,d)$ is AR of dimension $Q>1$ and if $\dim_{AR}(X,d)=Q$, then there exists a weak tangent space of $X$ that admits a family of curves of positive $Q$-analytical modulus (see Section \ref{sectangents} for a definition of tangent space of a metric space). We note the importance of this fact in the proof of the theorem of Bonk and Kleiner \cite{BK1}.

We show in Corollary \ref{kl} how Theorem \ref{maintheo} clarifies the reasons for the existence of curve families of positive analytical modulus when the AR conformal dimension is attained: this is a consequence of the sub-multiplicative inequality of the combinatorial modulus and the fact that the latter is bounded from above by the analytical modulus on weak tangent spaces.

The proofs of Theorem \ref{description} and Theorem \ref{maintheo} are rather technical and involve a careful study of the dependence of some constants on the parameters used in the constructions. For this reason, at risk of being repetitive, we include in the proofs detailed computations.

\bigskip
\textbf{Acknowledgments.} The author would like to thank Peter Haïssinsky for all his help and advice. He also thanks Marc Bourdon and Bruce Kleiner for comments and suggestions.

\subsection{Outline of the paper}
The paper is mainly divided into two parts. In Section \ref{chapjauge} we construct the graph $Z_d$ and we prove Theorem \ref{description}. In Section \ref{sec:controldim} we give sufficient conditions that will allow us to construct regular distances, of given dimension, in the conformal gauge (Proposition \ref{controldim}), and we simplify the hypothesis of Theorem \ref{description} to work with the combinatorial modulus (Proposition \ref{simplification}).

The purpose of Section \ref{secexpcrit} is to show how to compute the AR conformal dimension of a compact metric space using the combinatorial modulus. We define the combinatorial modulus associated to a sequence of graphs, the nerves of a sequence of coverings of $X$, and its critical exponent $Q_N$. In Section \ref{proofmain} we complete the proof that $Q_N$ is equal to the AR conformal dimension of $X$ (Theorem \ref{maintheo}).

In Section \ref{sectangents}, inspired by \cite{BK}, we show that the combinatorial modulus satisfies some kind of sub-multiplicative inequality giving the positiveness of the combinatorial modulus at the critical exponent (Corollary \ref{modpos}). We adapt to our situation arguments from \cite{KL} and \cite{H1} to bound from above the combinatorial modulus by the analytical moduli defined in the tangent spaces. Finally, these two facts with the equality $Q_N=\dim_{AR} X$ give a more conceptual proof of Keith and Laakso's theorem (Corollary \ref{kl}).

In Section \ref{secdifermods} we treat different definitions of combinatorial modulus. In Theorem \ref{egalite}, we give metric conditions on $X$ that allow us to compute its AR conformal dimension using another critical exponent $Q_X$, defined from ``genuine'' curves of $X$. In Corollary \ref{kk}, we give a proof of the result of Keith and Kleiner mentioned earlier (Remark 1 after Corollary 3.13 in \cite{BK}), i.e. when $X$ is approximately self-similar, it suffices to work with the modulus of curves with definite diameter. This allows us to give, in Corollary \ref{supcomposantesqas}, conditions under which the AR conformal dimension of $X$ is equal to the supremum of the AR conformal dimensions of its connected components. 

\subsection{Notations and some useful properties}

Two quantities $f(r)$ and $g(r)$ are said to be comparable, which will be denoted by $f(r)\asymp g(r)$, if there exists a constant $K$ which does not depends on $r$, such that $K^{-1} f(r)\leq g(r) \leq K f(r)$. If only the second inequality holds, we write $g(r)\lesssim f(r)$.
Similarly, we say that $f(r)$ and $g(r)$ differ by an additive constant, denoted by $g(r)=f(r) + O(1)$, if there exists a constant $K$ such that $|g(r)-f(r)|\leq K$. For a finite set $A$ we denote by $\#A$ its cardinal number.

A global distortion property of quasisymmetric homeomorphisms, which we will use repeatedly throughout this article, is the following (see \cite{He} Proposition 10.8):
let $h:X_1\to X_2$ be a $\eta$-quasisymmetric homeomorphism. If $A\subset B\subset X_1$ and $\diam_1 B <+\infty$, then $\diam_2 f(B) <+\infty$ and
\begin{equation}\label{distordiam}
\frac{1}{2} \eta\left(\frac{\diam_1 B} {\diam_1 A}\right)^ {-1} \leq\frac{\diam_2 h(A)} {\diam_2 h (B)} \leq \eta\left(\frac {2\diam_1 A}{\diam_1 B} \right).
\end {equation}

Let $(X,d)$ be a compact metric space. We say that $X$ is a \emph{doubling} space if there exists a constant $K_D\geq 1$ such that any ball of $X$ can be covered by at most $K_D$ balls of half the radius. This is equivalent to the existence of a function $K_D:(0,1/2)\to \R_+$ such that the cardinal number of any $\epsilon r$-separated subset contained in a ball of radius $r>0$, is bounded from above by $K_D(\epsilon)$. We recall that a subset $S$ of $X$ is called $\epsilon$-separated, where $\epsilon> 0$, if for any two different points $x$ and $y$ of $S$ we have $d(x, y)\geq \epsilon$.

We say that $X$ is \emph{uniformly perfect} if there exists a constant $K_P>1$ such that for any ball $B\left(x, r\right)$ of $X$, with $0<r\leq\diam X$, we have $B\left(x, r\right)\backslash B\left(x, K_P^{-1}r\right)\neq\emptyset$. This is equivalent to the fact that the diameter of any ball in $X$ is comparable to the radius of the ball. These two properties are quasisymmetric invariant and in fact we have: $(X,d)$ is doubling and uniformly perfect if and only if $\mathcal{J}_{AR}(X,d)\neq\emptyset$ (see \cite{He} Corollary 14.15).

Throughout the text, unless explicitly mentioned, $X$ denotes a compact, doubling and uniformly perfect metric space. We denote by $\dim_HX$ its Hausdorff dimension and by $\dim_TX$ its topological dimension. We reserve the letter $Z$ for a geodesic proper Gromov-hyperbolic metric space.

The distance between two subsets $A$ and $B$ of a metric space is denoted by
$$\mathrm{dist}(A,B):=\inf\{d(x,y):x\in A, y\in B\}.$$
We denote the ball centered at $x\in X$ and radius $r> 0$ by $B(x,r)=\{y\in X:d(x,y)<r\}$. The $r$-neighborhood $V_r(A)$ of $A$ is defined as the union of all balls centered at $A$ of radius $r$. The diameter of $A$ is denoted by $\diam A$. The Hausdorff distance between $A$ and $B$ is
$$\mathrm{dist}_H(A,B):=\min\{\partial (A,B),\partial (B,A)\},$$
where $\partial (A,B)=\inf\{r> 0: A\subset V_r(B)\}$.

We use in general the letters $A,B,C,\ldots$ to denote subsets of the space $X$ and the letters $x,y,z,\ldots$ to denote its points. The letters $K$, $L$ and $M$, eventually with indices, denote constants bigger or equal to $1$, and the letter $c$, eventually with an index, denotes a positive constant.

\section{Combinatorial description of the AR conformal gauge\label{chapjauge}}
\subsection{Hyperbolic structure of the snapshots of a compact metric space}\label{sec:graphe}

To construct distances in the gauge we use tools and techniques from hyperbolic geometry. This approach is based on the hyperbolicity of the \emph{snapshots} of a compact metric space. By the snapshots we mean the balls of the space, in the sense that a ball is a snapshot of the space at a certain point and at a certain scale. This terminology comes from S. Semmes \cite{Sem}.

We adapt a construction of Bourdon and Pajot, based on a nearby construction due to G. Elek (see \cite{BP} Section 2.1 and \cite{Elek}). This allows us to see the conformal gauge of a compact metric space $\left(X, d\right)$ as the canonical gauge of the boundary at infinity of a hyperbolic space (Proposition \ref{graphebp}). This hyperbolic space is a graph that reflects the combinatorics of the balls of $\left(X,d\right)$.

It is assumed in the following that $X$ is doubling and uniformly perfect. Let $\kappa\geq 1$, $a>1$ and $\lambda\geq 3$, the following constructions depend on these parameters. For $n\geq 1$, let $\mathcal{S}_n$ be a finite covering of $X$ such that for all $B\in \mathcal{S}_n$, there exists $x_B\in X$ with 
\begin{equation}\label{almostball1}
B\left(x_B,\kappa^{-1}r_n\right)\subset B\subset B\left(x_B,\kappa r_n\right),
\end{equation}
where $r_n:=a^{-n}$. We also suppose that for all $B\neq B'$ in $\mathcal{S}_n$, we have
\begin{equation}\label{almostball2}
B\left(x_B,\kappa^{-1}r_n\right)\cap B\left(x_{B'},\kappa^{-1}r_n\right)=\emptyset.
\end{equation}
We define $\mathcal{S}_0$ to be a one point subset of $X$, which we denote by $w:=\{x_0\}$, and represents the covering consisting of $X$ itself. We set $\mathcal{S}:=\bigcup_n\mathcal{S}_n$. Also denote $X_n$ the subset of $X$ consisting of the centers $x_B$, with $B\in \mathcal{S}_n$, defined in \ref{almostball1}. We write $|B|=n$ if $B\in \mathcal{S}_n$. For $\mu\in \R_+$ and $B\in\mathcal{S}_n$ we denote by $\mu\cdot B$ the ball centered at $x_B$ and of radius $\mu\kappa r_n$. 

We define a metric graph $G_d$ as follows. Its vertices are the elements of $\mathcal{S}$, and two distinct vertices $B$ and $B '$ are connected by an edge if
$$\left| |B|-|B'|\right|\leq 1\ \ \text{and}\ \ \lambda \cdot B \cap \lambda\cdot B\neq\emptyset.$$
We say that $B, B'\in\mathcal{S}_n$ are neighbors, and we write $B\sim B'$, if $B = B'$ or if they are connected by an edge of $G_d$. We equip $G_d$ with the length metric obtained by identifying each edge isometrically to the interval $[0,1]$; we denote this distance by $|B-B'|$. So $\mathcal{S}_n$ is the sphere of $G_d$ centered at $w$ and of radius $n$. See Figure \ref{fig:graphintervalle}.

\begin{figure}
\centering
\setlength{\unitlength}{1cm}
\begin{picture}(10,6.5)
\includegraphics[scale=0.6]{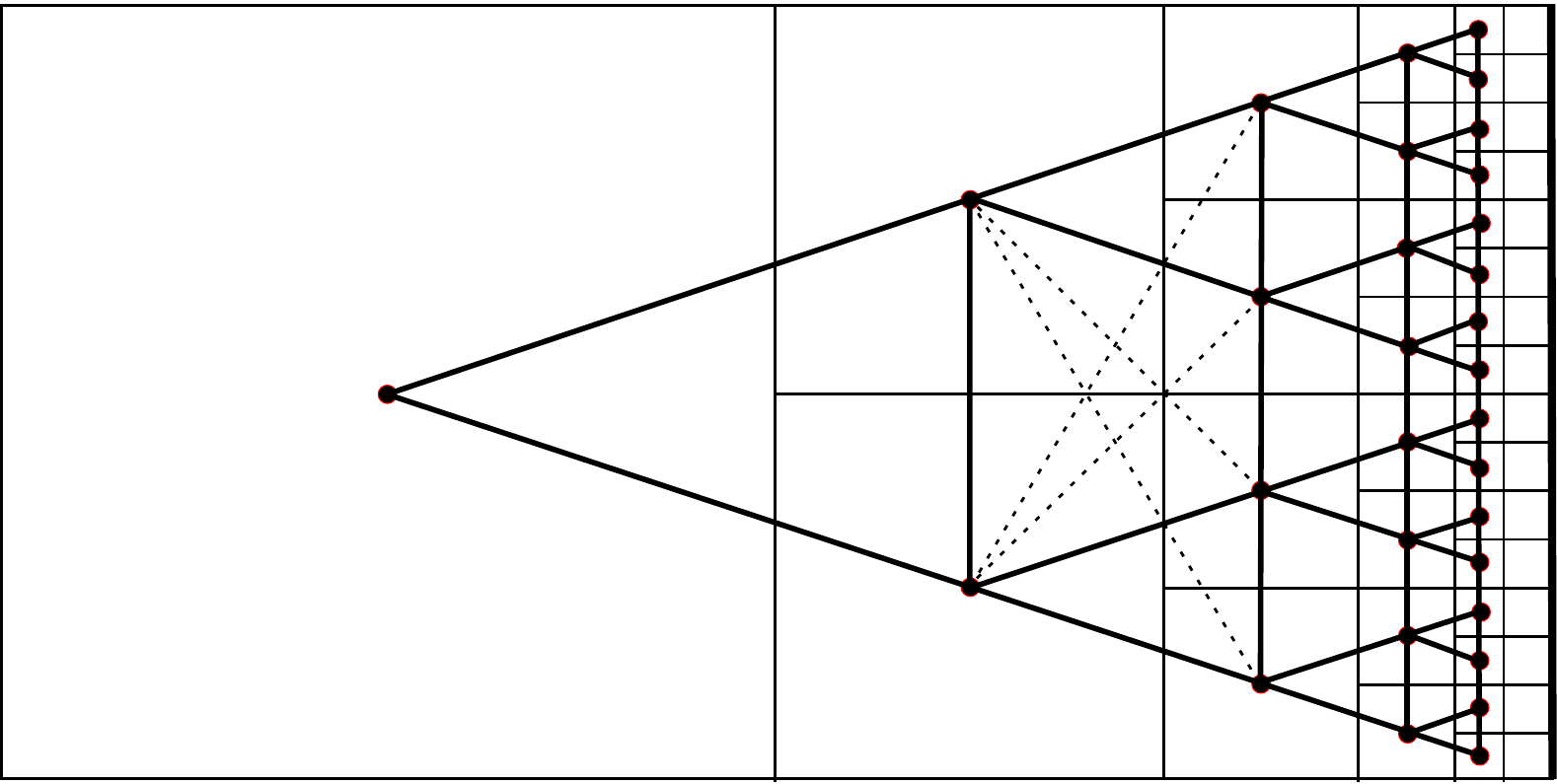}
\put(0.2,0){\footnotesize $0$}
\put(0.2,4.5){\footnotesize $1$}
\put(0.2,2.25){\footnotesize $X$}
\put(-7.6,2.25){\footnotesize $w$}
\put(-7.5,5){\footnotesize $S_0$}
\put(-3.7,5){\footnotesize $S_1$}
\put(-2,5){\footnotesize $S_2$}
\put(-1.1,5){\footnotesize $S_3\ldots$}
\end{picture}
\caption{\label{fig:graphintervalle}
Let $X$ be the interval $[0,1]$ in $\R$. We choose $a=2$ and $\lambda=3$. For each $n\geq 0$, let $X_n$ be the set of all mid points of the dyadic intervals $\left\{\left[\frac{j}{2^{n}}\frac{j+1}{2^n}\right]: j=0,\ldots,2^n\right\}$. Then $X_n $ is a maximal $2^{-n}$-separated set. The figure shows a sketch of the graph $G_d$. The edge length is equal to $1$ and the reader can see the hyperbolic nature of $G_d$.}
\end{figure}
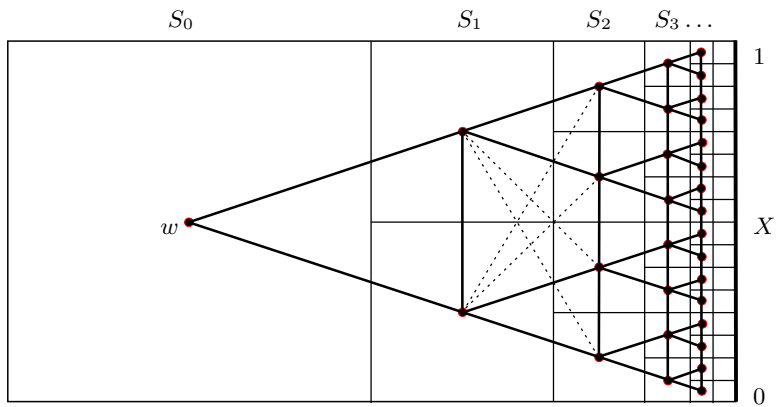

Before proceeding, we recall some notions from the theory of Gromov-hyperbolic spaces. We refer to \cite{CDP} and \cite{GH} for a detailed exposition. Let $Z$ be a metric space, we say that $Z$ is \emph{proper} if all closed balls are compact. A \emph{geodesic} is an isometric embedding of an interval of $\R$ in $Z$. We say that $Z$ is a geodesic space if for any pair of points there exists a geodesic joining them. In general, we use the notation $|x-y|$ for the distance between points, when the space $Z$ is geodesic, proper and unbounded. We fix $w\in Z$ a base point and we denote $|x|:=|x-w|$ for $ x\in Z$. The Gromov product of two points $x,y\in Z$, seen from the base point $w$, is defined by
$$(x|y):=\frac{1}{2}\left(|x|+|y|-|x-y|\right).$$
We say that $Z$ is \emph{Gromov-hyperbolic} (with hyperbolicity constant $\delta\geq 0$) if
$$(x|y)\geq\min\left\{(x|z),(z|y)\right\}-\delta,$$
for all $x,y,z\in Z$. A \emph{ray from $w$} is a geodesic $\gamma: \R_+ \to Z$ such that $\gamma(0)=w$. Let $\mathcal{R}_\infty$ be the set of rays from $w$. The Gromov-boundary of $Z$, denoted by $\partial Z$, is defined as the quotient of $\mathcal{R}_\infty$ by the following equivalence relation: two rays $\gamma_1$ and $\gamma_2$ are said to be equivalent if $\mathrm{dist}_H(\gamma_1,\gamma_2)<+\infty$. The space $Z\cup\partial Z$ has a canonical topology so that it is a compactification of $Z$. This topology is in fact metrizable.

Let $\varepsilon>0$, denote by $\phi_\varepsilon:Z\to(0,+\infty)$ the application $\phi_\varepsilon(x)=\exp(-\varepsilon|x|)$. We define a new metric on $Z$ by setting
\begin{equation}\label{distepsilon}
d_\varepsilon(x,y)=\inf\limits_{\gamma}\ell_\varepsilon(\gamma), \text{ where }\ell_\varepsilon(\gamma) = \int_\gamma\phi_\varepsilon,
\end{equation}
and where the infimum is taken over all rectifiable curves $\gamma$ of $Z$ joining $x$ and $y$. The space $(Z, d_\varepsilon)$ is bounded and not complete. Let $\overline{Z}_\varepsilon$ be the completion of $(Z, d_\varepsilon)$ and denote $\partial_\varepsilon Z =\overline{Z}_\varepsilon\backslash Z$. When $Z$ is a Gromov-hyperbolic space, there exists $\varepsilon_0=\varepsilon_0(\delta)>0$ such that for all $0<\varepsilon\leq\varepsilon_0$, the space $\partial_\varepsilon Z$ coincides with the Gromov-boundary of $Z$ and $d_\varepsilon$ is a \emph{visual metric} of parameter $\varepsilon$. That is to say, we can extend the Gromov product to the boundary $\partial Z$, and for all $x, y\in \partial Z$ we have
\begin{equation}\label{defmvisuelle}
d_\varepsilon(x, y)\asymp \exp\left(-\varepsilon(x|y)\right).
\end{equation}

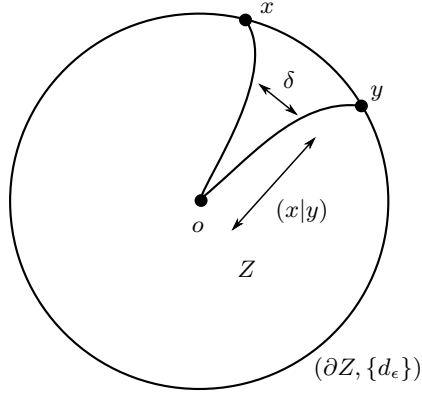
\begin{figure}
\centering
\setlength{\unitlength}{1cm}
\begin{picture}(6,6)
\includegraphics{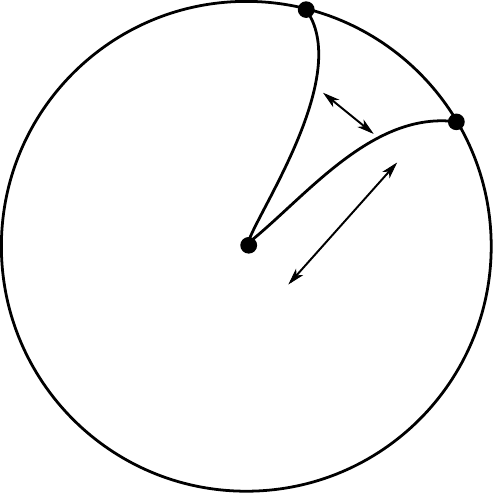}
\put(-2.6,2.1){\footnotesize $o$}
\put(-1,0.2){\footnotesize $\left(\partial Z,\{d_\epsilon\}\right)$}
\put(-2,1.5){\footnotesize $Z$}
\put(-1.5,2.3){\footnotesize $(x|y)$}
\put(-1.4,4){\footnotesize $\delta$}
\put(-1.7,5){\footnotesize $x$}
\put(-0.25,3.9){\footnotesize $y$}
\end{picture}
\caption{
\label{fig:produitgromov}
The canonical conformal gauge of the boundary of a Gromov-hyperbolic space.}
\end{figure}

We can interpret (\ref{defmvisuelle}) as follows: if $\gamma_1$ and $\gamma_2$ are two geodesic rays which represent the points $x$ and $y$ of $\partial Z$ respectively, then the Gromov product $(x|y)$ measures the length over which these two geodesics are at a distance comparable to $\delta$. So these two points of the boundary are close for the visual metric if the two geodesic rays are at a distance comparable to $\delta$, for a long period of time. Since visual metrics are always quasisymmetrically equivalent, they define a \emph{canonical conformal gauge} on the boundary $\partial Z$.

Since $X$ is doubling, the graph $G_d$ is of finite valence and hence it is a proper space. It is also geodesic, because it is a complete length space. The vertices of a ray $\gamma$ from $w$ determine a sequence of elements $B_n\in \mathcal{S}_n$ with $\lambda\cdot B_n\cap\lambda\cdot B_{n+1}\neq\emptyset$. Such a sequence has a unique limit point in $X$ denoted $\mathtt{p}(\gamma)$. If $\gamma_1$ and $\gamma_2$ are two rays at finite Hausdorff distance, then $\mathtt{p}\left(\gamma_1\right)=\mathtt{p}\left(\gamma_2\right)$. Also the map $\mathtt{p}:\mathcal{R}_\infty\to X$ is surjective, because $\left\{\mathcal{S}_n\right\}$ is a sequence of coverings. The following proposition, due to Bourdon and Pajot, allows us to use the tools of hyperbolic geometry to study the conformal gauge of $X$.

\begin{proposition}[\cite{BP} Proposition 2.1]\label{graphebp}
The metric graph $G_d$ is a Gromov-hyperbolic space. The map $\mathtt{p}$ induces a homeomorphism between $\partial G_d$ and $X$, and the metric $d$ of $X$ is a visual metric of visual parameter $\log a$. That is, for all $\xi,\eta\in\partial G_d$, we have
$$d\left(\mathtt{p}(\xi),\mathtt{p}(\eta)\right)\asymp a^{-(\xi|\eta)}.$$
In particular, with this identification, the conformal gauge of $X$ coincides with the canonical conformal gauge of $\partial G_d$.
\end{proposition}

\begin{remark}
From the proof of the proposition, we know that the comparison constants depend on $\lambda$, $K_P$, $\kappa$ and $a$: indeed, for any pair of vertices $B$ and $B'$ of $G_d$, we have
$$\frac{1}{a^2K_P\kappa}\cdot a^ {-\left(B|B'\right)}\leq\diam\left(B\cup B'\right)\leq\frac{4\lambda \kappa a}{a-1}\cdot a^{-\left(B|B'\right)}.$$
This shows that the distortion of $\mathtt{p}$ tends to infinity when $a\to\infty$. Nevertheless, the hyperbolicity constant of $G_d$ is given by
$$\delta=\log_a\left(\frac{8\lambda K_P \kappa^2 a^3}{a-1}\right),$$ which remains bounded when $a\to \infty$.
\end{remark}

We start by simplifying the space $G_d$, by taking a subgraph with less vertical edges on which it will be easier to control the length of vertical curves, while remaining within the same quasi-isometry class. To do this, we need the notion of a \emph{genealogy} on $\mathcal{S}$. Let $\mathcal{V}:= \{\mathcal{V}_n\}_ {n \geq 0}$ be a sequence of ``almost partitions'' of $X$, defined by $\mathcal{V}_n:= \{V_n (B): B\in\mathcal{S}_n\}$, where for each $n\geq 0$ and $B\in \mathcal{S}_n$, $V_n(B)$ is the subset of $X$ defined by
\begin {equation}\label{gencanonique}
V_n(B):=\left\{y\in X: d(y,x_B)=\mathrm{dist}(y,X_n)\right\}.
\end{equation}
The sets in $\mathcal{V}_n$ satisfy the following properties:
\begin{enumerate}
\item $X=\bigcup\limits_{B\in\mathcal{S}_n} V_n(B)$ for each $n\geq 0$,
\item for each $n\geq 0$ and $B\in\mathcal{S}_n$, the set $V_n(B)$ is compact and from (\ref{almostball1}) and (\ref{almostball2}), we have
\begin{equation}\label{fundaminclusions}
B\left(x_B,\kappa^{-1}r_n\right)\subset V_n(B)\subset B\left(x_B,\kappa r_n\right).
\end{equation}
\end{enumerate}
The second inclusion is a consequence of the fact that $\mathcal{S}_n$ is a covering of $X$. From $\mathcal{V}$ we can define for each $n\geq 0$, a partition $\{T_{n}(B)\}_{B\in\mathcal{S}_n}$ of $\mathcal{S}_{n +1}$ as follows: associate to $B'\in\mathcal{S}_{n+1}$ an element $B\in\mathcal{S}_n$ which verifies $x_{B'}\in V_n(B)$. Choose any one of them if there are several such elements. 

If $B\in \mathcal{S}_n$ and $r>0$, we denote by $N_r(B)$ the set of $B'\in\bigcup_{l\geq n+1}\mathcal{S}_l$ such that $x_{B'}\in B\left(x_B,r\right)$. We also set $A_r(B):=N_r(B)\cap\mathcal{S}_{n+1}$. With this notation, according to (\ref{fundaminclusions}) above, we have
\begin{equation}\label{fundaminclusions2}
A_{\kappa^{-1}r_n}(B)\subset T_n(B)\subset A_{\kappa r_n}(B).
\end{equation}
We say that the elements of $T_n(B)$ are descendants of $B\in\mathcal{S}_n$, and that $B$ is their common parent. This is reminiscent of the construction of dyadic decompositions, see for example \cite{Chr}.

Note that for all $n\geq 0$ and $B\in\mathcal{S}_n$, the cardinal number of $T_n(B)$ is less than or equal to $K_D\left(\kappa,a\right)$ a constant which depends only on $\kappa$, $a$ and the doubling constant of $X$. Also since $X$ is uniformly perfect, for any constant $N\in\N$, we can choose $a$ large enough which depends only on the constants $K_P$ and $\kappa$, such that $\#A_{\kappa^{-1}r_n}(B)\geq N$ for all $n$ and $B\in\mathcal{S}_n$.

We define the genealogy of an element $B\in S$ as
$$g(x)=\begin{cases} B & \text{ if } B\in \mathcal{S}_0 \\ 
(B_0,B_1,\ldots,B_{n+1}) & \text{ if } B\in \mathcal{S}_{n +1}, n\geq 0\end{cases},$$
where $B_{n+1}=B$ and $B_j\in\mathcal{S}_j$ is the parent of $B_{j +1}\in\mathcal{S}_{j+1}$ for $j=0,\ldots,n$. Let $B\in\mathcal{S}_n$, denote by $D(B)$ the elements of $\mathcal{S}$ which are descendants of $B$. That is,
$$D(B)=\left\{B'\in \mathcal{S}_l: l\geq n+1, g(B')_n=B\right\}.$$
We also set $D_l(B):=D(B)\cap\mathcal{S}_ {l}$, $l>n$, the descendants of $B$ in the generation $l$. The genealogy $\mathcal{V}$ determines a spanning tree $T$ of $G_d$, where $e=(B, B')$ is an edge of $T$ if and only if $B$ or $B'$ is the parent of the other.

Let $Z_d$ be the subgraph of $G_d$ such that it has the same vertices that $G_d$, and the edge $e=(B,B')$ of $G_d$ is also an edge of $Z_d$ if and only if either $e$ is a horizontal edge (i.e. $B$ and $B'$ belong to the same $\mathcal{S}_n$), or $e$ belongs to the spanning tree $T$ given by the genealogy $\{\mathcal{V}_n\}$. In this way, $Z_d$ is a connected graph, and we equip it with the length distance that makes all edges isometric to the interval $[0,1]$. Thus, we obtain a geodesic distance which we will denote by $|\cdot |_1$. The inclusion $Z_d\hookrightarrow G_d$ is co-bounded, because all vertices belong to $Z_d$, and we have $|\cdot|\leq|\cdot|_1$.

Recall that a \emph{quasi-isometry} between two metric spaces is a map $f:\left(Z_1, |\cdot|_1\right)\to \left (Z_2, |\cdot|_2\right)$ which satisfies the following properties: there exist constants $\Lambda\geq 1$ and $c\geq0$ such that
\begin{enumerate}
\item[(i)] for all $x, y \in Z_1$, we have $$\frac{1}{\Lambda}|x-y|_1-c\leq|f(x)-f(y)|_2\leq\Lambda|x-y|_1+c,$$
\item[(ii)] and for all $z\in Z_2$, $\mathrm{dist}_2\left(z,f(Z_1)\right)\leq c$, i.e. $f$ has co-bounded image.
\end {enumerate}
See for example \cite{CDP} and \cite{GH}. More important for us, is the fact that a quasi-isometry $f:\left (Z_1,|\cdot|_1\right)\to\left(Z_2,|\cdot|_2\right)$ induces a quasisymmetric homeomorphism $\hat{f}:\partial Z_1\to\partial Z_2$ between the boundaries, when they are endowed with visual metrics. Therefore, it preserves the canonical conformal gauge of the boundary. See Section 3 of \cite{H2}  for a more precise statement of this property.

Now let $e=(B,B')$ be an edge of $G_d$ which does not belongs to $Z_d$. We can assume, without loss of generality, that $B\in\mathcal{S}_{n +1}$ and $B'\in\mathcal{S}_n$. According to item (v) of Lemma \ref{propgraphe} below, if $B''\in\mathcal{S}_n$ is the parent of $B$, then $e'= (B'',B')$ is a horizontal edge of $G_d$, and hence of $Z_d$. This implies that for any edge-path $\gamma$ in $G_d$, there exists an edge-path $\gamma_1\in Z_d$ with $\ell_1\left(\gamma_1\right)\leq 2\ell\left(\gamma\right)$. This implies that $|\cdot|_1\leq 2|\cdot|+2$, so $Z_d$ is quasi-isometric to $G_d$.

This completes the construction of the geodesic hyperbolic metric graph $Z_d$ with boundary at infinity homeomorphic to $X$. With this identification, the distance $d$ is quasisymmetrically equivalent to any visual metric on $\partial Z_d$. The vertices of the graph $Z_d$ are the elements of $\mathcal{S}=\bigcup_n\mathcal{S}_n$, and the edges are of two types: vertical or horizontal. The vertical edges form a connected rooted tree $T$ and the horizontal ones describe the combinatorics of intersections of the elements in $\mathcal{S}$.

\subsection{Some properties of the graph $Z_d$}

For some technical reasons, the parameters $a$ and $\lambda$ must be large enough. We fix $\lambda\geq 32$, and it is thought to be an additional constant. Once $\lambda$ is fixed, we can choose the parameter $a$ freely, with the sole condition that
\begin{equation}\label{constantes}
a\geq K:= 6\kappa^2\max\{\lambda, K_P\}.
\end{equation}
This inequality ensures that certain conditions, which occur naturally in subsequent computations, are verified, and guarantees some geometric properties of the graph $Z_d$. In the following lemma we list some of these properties which will be useful in the sequel, and which show how the relation (\ref{constantes}) is involved in the geometry of the graph. The reader may skip this lemma and consult the required points at the time these properties are quoted.

\begin{lemma}[Properties of the graph]\label{propgraphe} Write $\tau=\frac{a}{a-1}$ and $\epsilon_n=\frac{a^{-n+1}}{a-1}=\tau r_n$.
\begin{enumerate}
\item[(i)] Let $n\geq 0$ and $B\in \mathcal{S}_n$. Then
\begin{equation}\label{descendents}
N_{\kappa^{-1}r_n}\left(B\right)\subset D(B)\subset N_{\tau\epsilon_n}(B).
\end{equation}
Recall the notation adopted in (\ref{fundaminclusions2}).
\item[(ii)] Let $z$ be a point of $X$, $r>0$ and $n\geq 1$ such that $r_n\leq r$. If $B$ is an element of $\mathcal{S}_{n+1}$ which verifies $d(z,x_B)<r_{n+1}$, then 
\begin{equation}\label{inclusion1}
X(B):=\left\{x_{B'}:B'\in D(B)\right\}\subset B\left(z,\frac{r}{2}\right).
\end{equation}
\item[(iii)]  If $B$ is an element of $\mathcal{S}_n$ and $z$ is a point of $X$ such that $d(z,x_B)\geq r+2\kappa r_n$, then 
\begin{equation}\label{inclusion2}
X(B)\cap B(z,r)=\emptyset.
\end{equation}
\item[(iv)] Let $B$ and $B'$ be two elements of $\mathcal{S}_ {n +1}$ such that $d(x_B,x_{B'})\leq 4r_n$. If $C$ and $C'$ are elements of $\mathcal{S}_n$ such that $d\left(x_C, x_B\right)\leq 2r_n$ and $d\left(x_{C'},x_{B'}\right)\leq 2r_n$, then $C$ and $C'$ are neighbors.
\item[(v)] Let $B$ be an element of $\mathcal{S}_{n+1}$ and $C,B'\in\mathcal{S}_n$ be such that there exists an edge in $G_d$ joining $B$ and $C$, and $B'$ is the parent of $B$. Then $B'$ and $C$ are neighbors.
\item[(vi)] For all $n\geq 0$ and $B\in\mathcal{S}_n$, the cardinal number of the set $A_{\kappa^{-1}r_n}(B)$, defined in (\ref{fundaminclusions2}) above, is at least two.
\item[(vii)] Let $B$ be an element of $\mathcal{S}_n$ and $B'$ an element of $\mathcal{S}_{n+1}$ such that $x_B$ belongs to the ball $B\left(x_{B'},\kappa r_{n+1}\right)$. Then, all the neighbors of $B'$ in $\mathcal{S}_{n+1}$ are descendants of $B$, i.e. they belong to $T_n(B)$.
\end{enumerate}
\end{lemma}
\begin{proof}$ $
\begin{enumerate}

\item[(i)] If $B\in\mathcal{S}_n$ and $B'\in D(B)\cap\mathcal{S}_l$ with $l\geq n+1$, then
$$d(x_B,x_{B'})\leq\sum_{i=n}^{l-1}d(x_i,x_{i+1})\leq\sum_{i=n}^{l-1}\kappa a^{-i}\leq \frac{a^{-n+1}}{a-1}=\kappa\epsilon_n,$$
where $x_i$ is the center of $g(B)_i$.
Thus $D(B)\subset N_{\kappa\epsilon_n}(B)$. Analogously, if $d(x_B,x_{B'})<\frac{r_n}{2\kappa}$, with $B'\in S_l$,  $B\in S_n$ and $l>n$, then $B'\in D(B)$. In fact, let $B''=g(B')_{n+1}$, then $$d(x_{B''},x_B)\leq d(x_{B''},x_{B'})+d(x_{B'},x_B)< \kappa\epsilon_{n+1}+\frac{a^{-n}}{2\kappa}=a^{-n}\left(\frac{\kappa\tau}{a}+\frac{1}{2\kappa}\right)<\frac{a^{-n}}{\kappa},$$
according to (\ref{constantes}). Thus $B''\in D(B)\cap S_{n+1}$, which implies $B'\in D(B)$.

\item[(ii)] We see that if $B'\in D(B)$, then 
\begin{align*}
d(z,x_{B'})&\leq d(z,x_B)+d(x_B,x_{B''})<a^{-(n+1)}+\kappa\epsilon_{n+1}\\
&=a^{-n}\left(\frac{1}{a}+\frac{\kappa\tau}{a}\right)<r\left(\frac{1+\kappa\tau}{a}\right)<\frac{r}{2},
\end{align*} 
according to (\ref{constantes}). 

\item[(iii)] Since $D(B)\subset N_{\kappa\epsilon_n}(B)$, for all $B'\in D(B)$, we have 
$$d(z,x_{B'})\geq r+2\kappa a^{-n}-\kappa\epsilon_n>r,$$
because after the choice of $a$, we have $\tau<2$ (see (\ref{constantes})).

\item[(iv)] Since $d\left(x_B,x_{B'}\right)\leq 4a^{-n}$ we have 
$$d\left(x_C,x_{C'}\right)\leq a^{-n}(4+2\kappa)\leq \lambda \kappa a^{-n},$$
where the last inequality is true because $\lambda>6$. Therefore $C\sim C'$.
\item[(v)] Let $w\in \lambda\cdot B$. Then 
$$d\left(w,x_{B'}\right)\leq d\left(w,x_B\right)+d\left(x_{B'},x_B\right)\leq \kappa\lambda a^{-(n+1)}+\kappa a^{-n}=\left(\frac{\lambda}{a}+1\right)\kappa a^{-n}<\lambda \kappa a^{-n},$$
according to (\ref{constantes}). Thus $\lambda \cdot B\subset \lambda \cdot B'$, which implies $\lambda\cdot C\cap\lambda\cdot B'\neq\emptyset$ and that $B'\sim C$.
\item[(vi)] Let $B$ be an element of $\mathcal{S}_n$ with $n\geq 0$. Since $X$ is uniformly perfect, there exists a point $y$ in the ball $B\left(x_B,(2\kappa)^{-1}r_n\right)$ such that $d(y,x_B)\geq (2\kappa K_P)^{-1}r_n$. So if $B'$ is an element of $\mathcal{S}_{n+1}$ such that $y\in B\left(x_{B'},\kappa r_{n+1}\right)$, we have 
$$\kappa r_{n+1}<\frac{r_n}{2\kappa K_P}-\kappa r_{n+1}\leq d\left(x_{B'},x_B\right)\leq \left(\frac{\kappa}{a}+\frac{1}{2\kappa}\right)a^{-n}\leq \frac{1}{\kappa}r_n.$$
The first and the last inequalities are consequence of (\ref{constantes}). Let $B''$ be an element of $\mathcal{S}_{n+1}$ such that $x_B$ belongs to the ball $B\left(x_{B''},\kappa r_{n+1}\right)$. Then $B'$ and $B''$ are two different elements of $A_{\kappa^{-1}r_n}(B)$.
\item[(vii)] Indeed, if $B''\sim B'$ in $\mathcal{S}_{n+1}$, then $d(x_{B''},x_{B})\leq 2\lambda \kappa r_{n+1}\leq \kappa^{-1}r_n$, so all neighbors of $B'$ in $\mathcal{S}_{n+1}$ belong to the set $A_{\kappa^{-1}r_n}(B)$.
\end{enumerate}
\end{proof}

\subsection{Proof of Theorem \ref{description}}\label{sec:caracterisation}

We start by recalling the notation, given in the Introduction, involved in the statement of the theorem. For each element $B\in\mathcal{S}$, there exists a unique geodesic segment in $Z_d$ which joins the base point $w$ and $B$; it consists of vertical edges that join the parents of $B$. Denote it by $[w,B]$. Given a function $\rho:\mathcal{S}\to (0,1)$, which can be interpreted as an assignment of ``new relative radius'' of the elements of $\mathcal{S}$ ---or, as we will see later, an assignment of ``new lengths'' for the edges of $Z_d$--- the ``new radius'' of an element $B\in \mathcal{S}$, is expressed by the function $\pi:\mathcal{S}\to (0,1)$ given by
$$\pi(B):=\prod\rho(B'),$$
where the product is taken over all the balls $B'\in \mathcal{S}\cap [w,B]$. 

If $\gamma=\left\{(B_i, B_ {i +1})\right\}_{i=1}^{N-1}$ is a path of edges in $Z_d$ with $B_i\in \mathcal{S}$, we define the $\rho$-length of $\gamma$ by
$$L_\rho\left(\gamma\right)=\sum\limits_{i =1}^N\pi\left(B_i\right).$$
Let $\alpha>1$. For $x,y\in X$, let $m\in\N$ be maximal such that there exists $B\in S_m$ with $x, y\in\alpha \cdot B$. We let
$$c_\alpha(x, y):=\left\{B\in S_m: x, y\in\alpha\cdot B\right\},$$ 
and call it \emph{the center} of $x$ and $y$. Note that if $m=|c_\alpha(x, y)|$ ---distance to the base point $w$--- then by maximality of $m$, and the fact that $\mathcal{S}_{m+1}$ is a covering of $X$, we have 
$$\left(\alpha-1\right)\kappa r_{m+1}\leq d(x,y)\leq 2\kappa \alpha r_m.$$
Define $\pi\left(c_\alpha(x,y)\right)$ as the maximum of $\pi\left(B\right)$ for $B\in c_\alpha(x, y)$.  We also let $\Gamma_n(x,y)$ be the family of paths in $Z_d$ that join two elements $B$ and $B'$ of $\mathcal{S}_n$, with $x\in B$ and $y\in B'$. Finally, for an element $B\in \mathcal{S}_m$ and $n>m$, we denote by $D_n(B)$ the set of its descendants $B'$ in $\mathcal{S}_n$.

Suppose the parameters $a$ and $\lambda$ verify (\ref{constantes}), and the function $\rho$ satisfies the following conditions.
\begin{enumerate}
\item[(H1)] (Quasi-isometry) There exist $0<\eta_-\leq\eta_+<1$ so that $\eta_-\leq\rho(B)\leq\eta_+$ for all $B\in\mathcal{S}$.
\item[(H2)] (Gromov product) There exists a constant $K_0\geq 1$ such that for all $B,B'\in \mathcal{S}$ with $B\sim B'$, we have
$$\frac{\pi(B)}{\pi(B')}\leq K_0.$$
\item[(H3)] (Visual parameter) There exist $\alpha\in\left[2,\lambda/4\right]$ and a constant $K_1\geq 1$ such that for any pair of points $x,y\in X$, there exists $n_0\geq 1$ such that if $n\geq n_0$ and $\gamma$ is a path in $\Gamma_n(x,y)$, then
$$L_\rho\left(\gamma\right)\geq K_1^{-1} \cdot\pi\left(c_\alpha(x, y)\right).$$
\item[(H4)] (Ahlfors regularity) There exist $p>0$ and a constant $K_2\geq 1$ such that for all $B\in \mathcal{S}_m$ and $n>m$, we have
$$K_2^{-1}\cdot\pi(B)^p\leq\sum\limits_{B'\in D_{n}(B)}\pi\left(B'\right)^p\leq K_2\cdot\pi(B)^p.$$
\end{enumerate}

We must show that there exists a distance $\theta_\rho$ on $X$, quasisymmetrically equivalent to $d$ and Ahlfors regular of dimension $p$. Moreover, from the proof we will obtain $\theta_\rho(x,y)\asymp \pi\left(c_\alpha(x, y)\right)$ for all $x,y\in X$. Conversely, any distance in the gauge is bi-Lipschitz equivalent to a distance built in that way.

The proof of the direct implication is made in several steps. The first one (Lemma \ref{geodverticales}),
is to find a distance $|\cdot|_\rho$ in $Z_d$, so that $(Z_d, |\cdot|_\rho)$ is quasi-isometric to $Z_d$ and $|B|_\rho=\log\pi(B)^{-1}$ for all $B\in\mathcal{S}$. This is where the hypotheses (H1) and (H2) are used, they give us a control of the length of vertical curves in $Z_d$ and of the Gromov product in this new metric.

The second step is to show the existence of a visual metric $\theta_\rho$ in the boundary of $(Z_d, |\cdot|_\rho)$, of large enough visual parameter (Proposition \ref{controlepsilon}). It is mainly here where we use the assumption (H3). We automatically have $\theta_\rho\in\mathcal{J}(X,d)$, because it is a visual metric. Finally, the control of the visual parameter and hypothesis (H4), will enable us to show that the $p$-dimensional Hausdorff measure is Ahlfors regular (Proposition \ref{ahregular}).

\subsubsection{Proof of the converse}

We start by proving the converse of the theorem, because it helps understanding the significance of the hypotheses.

Let $\theta\in\mathcal{J}\left(X,d\right)$ be Ahlfors regular of dimension $p>0$, and let $\eta:[0,\infty)\to [0,\infty)$ be the distortion function of $id:\left(X, d\right)\to\left(X,\theta\right)$. We write $\diam_\theta$ for the diameter in the distance $\theta$, and $\mu$ its $p$-dimensional Hausdorff measure. For $n\geq 1$ and $B\in\mathcal{S}_n$, if we denote the parent of $B$ by $B'=g_{n-1}(B)$, we define
$$\rho\left(B\right):=\left(\frac{\mu\left(\lambda \cdot B\right)}{\mu\left(\lambda\cdot B'\right)}\right)^{1/p}.$$

With this definition, $\pi(B)=\mu\left(\lambda\cdot B\right)^{1/p}$. We begin with some general remarks. Let $\beta\geq 1$, $r>0 $ and $x\in X$, then there exists $s>0$ such that if we denote $B_\theta(s)$ the ball in the distance $\theta$ centered at $x$ and of radius $s$, then $B_\theta(s)\subset B(x,r)\subset B(x,\beta r)\subset B_\theta(H_\beta s)$, where $H_\beta=\eta(\beta)$. Therefore, there exists a constant $K_\beta$, which depends only on $H_\beta$ and the constant $K_P$, such that
$$1\leq\frac{\diam_\theta B(x, \beta r)}{\diam_\theta B(x,r)}\leq K_\beta.$$
In particular, this implies ---by taking $\beta= 1$--- that there exists a constant $K$, depending only on $p$, $H_1$ and the regularity constant of $\mu$, such that
$$K^{-1}\cdot\diam_\theta B(x,r)^p\leq\mu\left(B(x,r)\right)\leq K\cdot\diam_\theta B(x,r)^p.$$

First we check (H1): let $n\geq 1$ and $B\in\mathcal{S}_n$, denote $B'$ the parent of $B$ in $\mathcal{S}_{n-1}$. We have $\lambda\cdot B\subset 2\cdot B'$. Since $X$ is uniformly perfect of constant $K_P$, we know that $A_{n-1}:=\left(2K_P\cdot B'\right)\setminus 2\cdot B'\neq\emptyset$. Hence, there exists $C\in\mathcal{S}_n$ such that $C\cap A_{n-1}\neq\emptyset$. We have $\lambda\cdot B\cap C = \emptyset$, because by the triangle inequality $d(C,\lambda\cdot B)\geq \kappa(2r_{n-1}-(\lambda+3)r_n)>0$, and by the choice of $a$ and $\lambda$, we have $C\subset\lambda\cdot B'$ (see (\ref{constantes})). Since $\mu\left(\lambda \cdot B\right)+\mu\left(C\right)\leq\mu\left(\lambda B'\right)$, we obtain
$$\frac{\mu\left(\lambda \cdot B\right)}{\mu\left(\lambda \cdot B'\right)}\leq 1-\frac{\mu\left(C\right)}{\mu\left(\lambda B'\right)}.$$
On the other hand,
$$\frac{\mu\left(C\right)}{\mu\left(\lambda\cdot B'\right)}\geq\frac{1}{K^2}\frac{(\diam_\theta C)^p}{\left(\diam_\theta \lambda\cdot B'\right)^p}\geq\frac{1}{2K^2\cdot\eta\left(2\lambda K_P a\right)^p}:=\delta.$$
So it suffies to take $\eta_+=1-\delta^{1/p}$. Similarly, we have
$$\rho(B)\geq\frac{1}{2K^{2/p}}\cdot\eta\left(\frac{\diam\lambda\cdot B'}{\diam\lambda\cdot B}\right)^{-1}\geq\frac{1}{2K^{2/p}}\cdot\eta\left(\frac{2a}{K_P}\right)^{-1}:=\eta_-.$$

For (H2), we see that if $B,B'\in\mathcal{S}_n$ are neighbors, then $\lambda \cdot B'\subset 3\lambda \cdot B$. Thus, there exists a constant $K_0$ that depends only on $\eta$, $K$ and $p$, such that
$$\frac{\pi(B)}{\pi(B')}=\left(\frac{\mu\left(\lambda\cdot B\right)}{\mu\left(\lambda\cdot B'\right)}\right)^{1/p}\leq K^{2/p}\frac{\diam_\theta\lambda\cdot B}{\diam_\theta\lambda\cdot B'}\leq K_0.$$

We take $\alpha=2$ and we now look at (H3). Let $x,y\in X$, $m=\left|c_2(x, y)\right|$ and $C\in c_2(x,y)$. We have
$$\frac{\theta(x,y)}{\diam_\theta\lambda\cdot C}\geq \frac{1}{2}\cdot \eta\left(\frac{\diam\lambda\cdot C}{d(x,y)}\right)^{-1}\geq \frac{1}{2}\cdot \eta\left(2\lambda a\right)^{-1}.$$
Therefore, there exists a constant $K'$ which depends only on $K$ and $a$, such that
$$\theta(x,y)\geq \frac{1}{K'}\cdot \pi(C).$$
On the other hand, if $B,B'\in\mathcal{S}_n$ are such that $x\in B$ and $y\in B'$, with $n\geq m$, and if $\gamma=\left\{(B_i,B_{i +1})\right\}_{i=1}^N$ is a path of the graph $Z_d$ with $B_1=B$ and $B_N=B'$, we have
$$\theta(x,y)\leq \sum\limits_{i=1}^N\diam_\theta\lambda\cdot B_{i}\leq K^{1/p}\sum\limits_{i=1}^N\pi\left(B_i\right).$$
Thus, we obtain (H3) with $K_1=\left(K'K^{1/p}\right)^{-1}$.

Finally, we look at (H4). Let $m\geq 0$, $n\geq m+1$ and $B\in\mathcal{S}_m$. Since the union of the balls $\lambda\cdot B'$, with $B'\in D_n(B)$, contains the ball centered at $x_B$ and of radius $(2\kappa)^{-1}r_m$, we have
$$\mu\left(B(x_B,(2\kappa)^{-1}r_m)\right)\leq \sum_{B'\in D_n(B)}\pi(B')^p.$$
Remember that the balls $\left\{B\left(x_{B'},\kappa^{-1}r_n\right): B'\in D_n(B)\right\}$ are pairwise disjoint. Therefore, there exists a constant $K''$ which depends only on the function $H$ and the constants $\lambda$, $\kappa$ and $K$, such that
$$\sum\limits_{B'\in D_n(B)}\pi(B')^p\leq K''\sum\limits_{B'\in D_n(B)}\mu\left(B\left(x_{B'},\kappa^{-1}r_n\right)\right)\leq \mu\left(\lambda\cdot B\right).$$
This proves (H4) with a constant $K_2$ which depends on $\lambda$. Moreover, the proof of (H3) shows that for all $x,y\in X$, we have
$$\pi\left(c_2(x,y)\right)\asymp \theta(x,y),$$
where the constants of comparison depend on $a$ and $\lambda$. The distance $\theta_\rho$ constructed using the function $\rho$ is also bi-Lipschitz to $\pi\left(c_2(x, y)\right)$, so $\theta$ is bi-Lipschitz equivalent to $\theta_\rho$. This ends the proof of the converse.

\subsubsection{Proof of the direct implication}
We start with the following lemma.

\begin{lemma}\label{geodverticales}
There exists a distance $|\cdot |_\rho$ on the graph $Z_d$ bi-Lipschitz equivalent to $ |\cdot |$, with the property that any vertical path of edges $\gamma$ in $Z_d$, joining $B\in\mathcal{S}_n$ and $B'\in\mathcal{S}_m$, is a geodesic segment for the distance $|\cdot|_\rho$ and its length is
$$\ell_\rho(\gamma)=\left|\log\frac{1}{\pi(B)} -\log\frac{1}{\pi(B')}\right|.$$
In particular, for all $B\in\mathcal{S}$, we have $|x-w|_\rho=\log\frac{1}{\pi(B)}$. We denote by $Z_\rho$ the graph $Z_d$ with the distance $|\cdot|_\rho$.
\end{lemma}
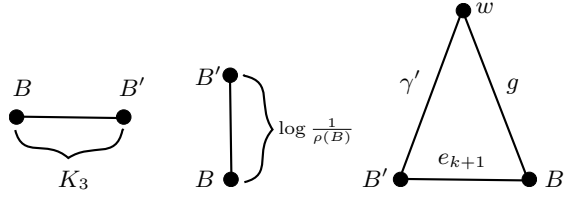
\begin{figure}
\centering
\setlength{\unitlength}{1cm}
\begin{picture}(7,3)
\includegraphics{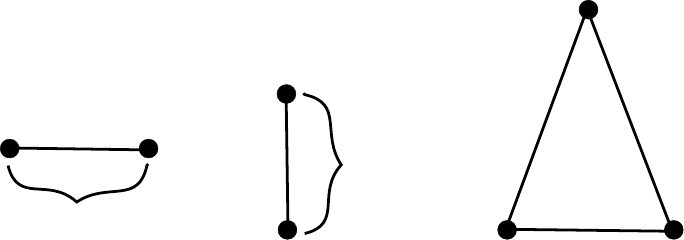}
\put(0.1,0){\footnotesize $B$}
\put(-2.3,0){\footnotesize $B'$}
\put(-1.3,0.3){\footnotesize $e_{k+1}$}
\put(-0.8,2.3){\footnotesize $w$}
\put(-0.4,1.3){\footnotesize $g$}
\put(-1.8,1.3){\footnotesize $\gamma'$}
\put(-4.5,0){\footnotesize $B$}
\put(-4.5,1.35){\footnotesize $B'$}
\put(-3.4,0.7){\tiny $\log\frac{1}{\rho(B)}$}
\put(-5.5,1.2){\footnotesize $B'$}
\put(-6.9,1.2){\footnotesize $B$}
\put(-6.3,0){\footnotesize $K_3$}
\end{picture}
\caption{\label{fig:figuredistrho}
Proof of the Lemma \ref{geodverticales}.}
\end{figure}
\begin{proof}
By (H2), there exists a constant $K_0\geq 1$ such that if $B$ and $B'$ belong to $\mathcal{S}_n$ and $B\sim B'$, then
\begin{equation}\label{czero}
\frac {1}{K_0}\leq\frac{\pi(B)}{\pi(B')}\leq K_0.
\end{equation}
Set $K_3:=2\max\left\{-\log\eta_-,-\left(\log\eta_+\right)^{-1},\log K_0\right\}>0$. Let $|\cdot|_\rho$ be the length distance in $Z_d$ such that the length of an edge $e=(B,B')$ is given by
$$\ell_\rho(e)=\begin{cases} K_3 & \text{ if } e\text{ is horizontal }\\ \log\frac{1}{\rho(B)} &\text{ if } B\in\mathcal{S}_{n+1}\text{ and } B'= g(B)_n.\end{cases}$$
Since $\frac{1}{K_3}\leq\ell_\rho(e)\leq K_3$ for any edge $e$ of $Z_d$ (by (H1)), the distance $|\cdot|_\rho$ is bi-Lipschitz equivalent to $|\cdot |$. Finally, it suffices to show that if $\gamma$ is a geodesic for $ |\cdot|_\rho$ which joins $w$ and $B\in\mathcal{S}_n$, then $\gamma$ is a path of vertical edges.

Suppose $\gamma=\{e_i\}_{i=1}^N$, and that there is a first $k\geq 1$ such that $e_{k +1}=(B',B)$ is a horizontal edge. Denote by $\gamma'= \{e_i\}_{i=1}^{k+1}$ and remark that $B',B\in\mathcal{S}_{k}$. Let $g=\{g_i\}_{i=1}^{k}$ be the path of vertical edges joining $w$ and $B$, where $g_i=(g(B)_{i-1},g(B)_i)$ for all $i=1,\ldots, k$. Then $\ell_\rho(g)=\log\frac{1}{\pi(B)}$ and $\ell_\rho(\gamma')=\log\frac{1}{\pi(B')} + K_3$. Since $B'\sim B$, we have by (\ref{czero})
$$\ell_\rho(g)=\log\frac{1}{\pi(B)}\leq\log\frac{1}{\pi(B')}+\frac {K_3}{2} <\ell_\rho(\gamma'),$$
which is a contradiction, since $\gamma'$ is also a geodesic for $|\cdot|_\rho$.
\end{proof}

Since $Z_\rho$ is a geodesic space quasi-isometric to $Z_d$, $Z_\rho$ is Gromov hyperbolic, its boundary at infinity $\partial Z_\rho $ is homeomorphic to $X$ and any visual metric on $\partial Z_\rho$ is quasisymmetrically equivalent to the original distance $d$ of $X$. We identify $\partial Z_\rho$ with $X$.

The following proposition allows us to control the visual parameter that guarantees the existence of visual metrics on $\partial Z_\rho$. For $B, B'\in Z_\rho$, we denote the Gromov product of $B$ and $B'$ in the distance $|\cdot|_\rho$ by $(B|B')_\rho$. To simplify the notation, we write $c(x,y)$ instead of $c_\alpha(x, y)$ for $x,y\in X$.

\begin{proposition}[Visual parameter control]\label{controlepsilon} There exists a visual metric $\theta_\rho$ on $\partial Z_\rho$ of visual parameter equal to $1$. Moreover, for all $x, y\in X$, we have
\begin{equation}\label{visual}
e^{-(x|y)_\rho}\asymp \pi(c(x, y)).
\end{equation}
\end{proposition}

The proof of Proposition \ref{controlepsilon} is divided into several lemmas. Recall that for $\varepsilon> 0$, we denote by $\phi_\varepsilon: Z_d\to(0, +\infty)$ the function given by $\phi_\varepsilon(x)=\exp\left(-\varepsilon|x-w|_\rho\right)$ (see (\ref{defmvisuelle})). We have the metric $d_\varepsilon$ on $Z_d$ defined in (\ref{distepsilon}). Also recall that $(Z_d,d_\varepsilon)$ is a non complete bounded metric space. Denote by $\rho_\varepsilon(B):=\rho(B)^\varepsilon$ and $\pi_\varepsilon(B):=\pi(B)^\varepsilon$.

Note that for all $\varepsilon\in (0,1]$, the function $\rho_\varepsilon$ satisfies the hypotheses (H1), (H2) and (H3) of Theorem \ref{description}, with the constants to the power $\varepsilon$, and the hypothesis (H4) holds with $p_\varepsilon:=p/\varepsilon$. In the sequel we will always assume $\varepsilon\in (0,1]$. We first need to estimate the $\rho_\varepsilon$-length of an edge $e$ of the graph $Z_d$.

\begin{lemma}\label{longarete}
There exists a constant $K_4$ such that for any edge $e=(B,B')$ of the graph $Z_d$, we have
\begin{equation}\label{ec:longaretes}
\frac{1}{K_4} \frac{\pi_\varepsilon(B)+\pi_\varepsilon(B')}{2}\leq \ell_\varepsilon(e)\leq K_4 \frac{\pi_\varepsilon(B)+\pi_\varepsilon(B')}{2}.
\end{equation}
\end{lemma}
\begin{proof}
Let $e=(B,B')$ be an edge of $Z_d$. Since $1/K_3\leq\ell_\rho(e)\leq K_3$ and $\left||z|_\rho-|B|_\rho\right|\leq K_3$ for all $z\in e$, we have
\begin{equation*}
\pi_\varepsilon(B)\frac{1}{K_3}\exp(-\varepsilon K_3)\leq \int_e\exp\left(-\varepsilon |z|_\rho\right)ds\leq  \pi_\varepsilon(B)K_3\exp(\varepsilon K_3).
\end{equation*}
Thus $K_4^{-1}\pi_\varepsilon(B)\leq \ell_\varepsilon(e)\leq K_4\pi_\varepsilon(B)$, where $K_4$ is a constant which depends only on $K_3$ and $\varepsilon$. In the same way we obtain $\ell_\varepsilon(e)\asymp \pi_\varepsilon(B')$. This show (\ref{ec:longaretes}).
\end{proof}

\begin{lemma}\label{bornesup} Let $x$ and $y$ be two points in $X$ and let $m=|c(x,y)|$. Then for all $n\geq m$, we have 
\begin{equation}\label{eqbornesup} d_\varepsilon(B,B')\lesssim \pi_\varepsilon\left(c(x,y)\right),\end{equation} where $B$ and $B'$ are elements in $\mathcal{S}_n$ that contain $x$ and $y$ respectively.
\end{lemma}
\begin{proof}
Let $B$ and $B'$ be as in the statement of the lemma and let $C\in c(x, y)$. Consider the geodesic segments $g_1=[B_m, B]$ and $g_2=[B'_m,B']$, where we write $B_m=g(B)_m\in\mathcal{S}_m$ and $B'_m=g(B')_m\in\mathcal{S}_m$. To simplify the notation, we write $x_m$ and $y_m$ for the centers of $B_m$ and $B_m'$ respectively. Then
\begin{align*}d(x_m,x_C) &\leq d(x_m,x_B)+d(x_B,x)+d(x,x_C)\\
&\leq \kappa\left(\tau+1+\alpha\right)a^{-m}\leq \kappa\left(3+\lambda/4\right)a^{-m}\leq \lambda\kappa a^{-m},\end{align*}
where the last inequality follows from the fact that $\lambda>4$. So $e=(B_m, C)$ is a horizontal edge of $Z_d$. Similarly $(B'_m,C)$ is an edge of $Z_d$.

Set $\gamma$ to be the curve of $Z_d$, which joins $B$ and $B'$, given by
\begin{equation}\label{courbemin}
\gamma:=\left[B,B_m\right]\ast \left(B_m,C\right)\ast \left(C,B_m\right)\ast \left[B'_m,B'\right].
\end{equation}
We also write $B_i=g(B)_i$ and $B'_i=g(B')_i$ for $i=m+1,\ldots, n-1$. Then, by Lemma \ref{longarete}, we can bound from above the $\rho_\varepsilon$-length of $\gamma$ by
\begin{align*}
\ell_\varepsilon\left(\gamma\right)\leq &\ \ell_\varepsilon\left(g_1\right)+\ell_\varepsilon\left((B_m,C)\right)+\ell_\varepsilon\left((C,B'_m)\right)+\ell_\varepsilon\left(g_2\right)\\
\leq &\ K_4\cdot\left(\sum\limits_{i=m}^{n-1}\pi(B_i)^\varepsilon+2\pi(C)^\varepsilon+\sum\limits_{i=m}^{n-1}\pi(B'_i)^\varepsilon\right)\\
\leq &\ K_4\cdot\left(\pi(B_m)^\varepsilon\cdot\sum\limits_{i=m}^{n-1}(\eta_+^\varepsilon)^i+2\pi(C)^\varepsilon+\pi(B'_m)^\varepsilon\sum\limits_{i=m}^{n-1}(\eta_+^\varepsilon)^i\right)\lesssim \pi(C)^\varepsilon,
\end{align*}
where the last inequality follows from (\ref{czero}). This implies (\ref{eqbornesup}).
\end{proof}

\begin{figure}
\centering
\setlength{\unitlength}{1cm}
\begin{picture}(11,9.5)
\includegraphics[scale=0.9]{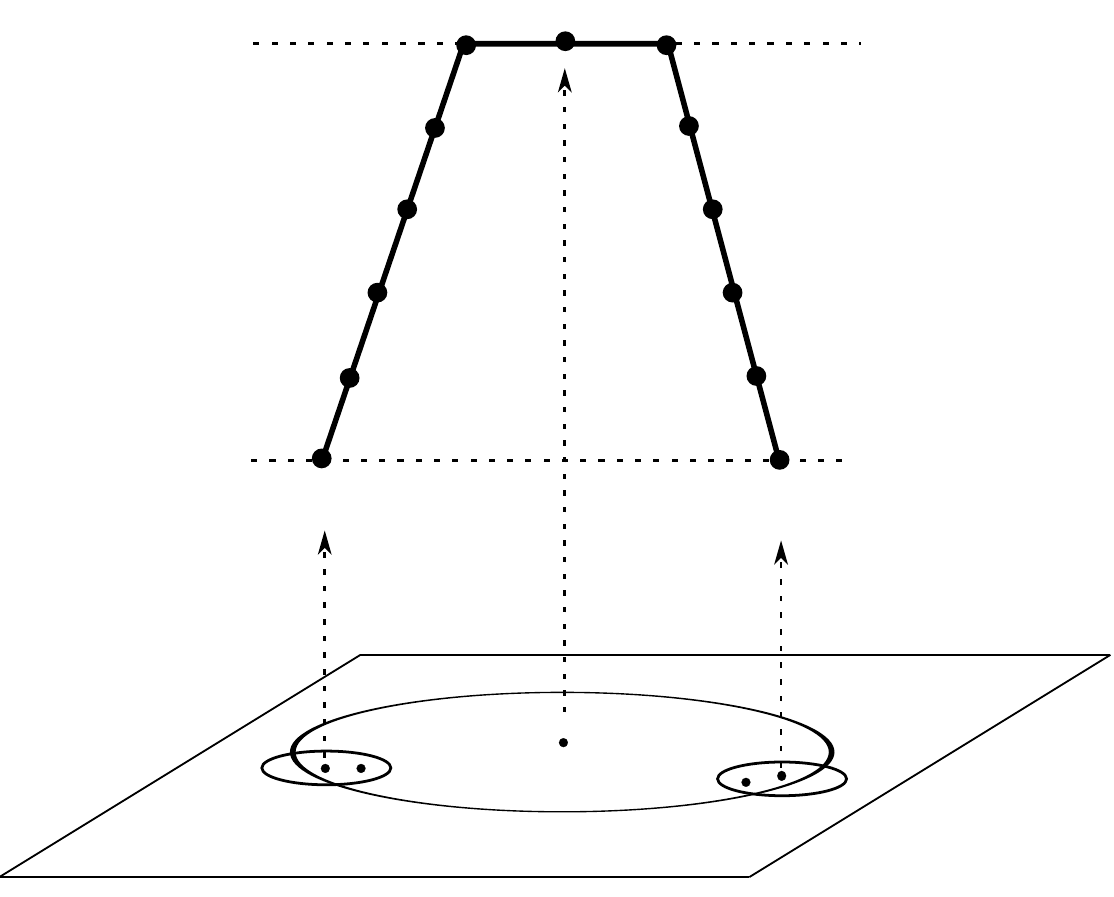}
\put(-8.3,3.9){\footnotesize $n$}
\put(-8.3,7.7){\footnotesize $m$}
\put(-5.1,1){\footnotesize $x_C$}
\put(-4.9,0.4){\footnotesize $\alpha\cdot C$}
\put(-2.3,0.4){\footnotesize $X$}
\put(-5.5,8.1){\scriptsize $c_\alpha(x,y)$}
\put(-6.2,8.1){\scriptsize $B_m$}
\put(-4,8.1){\scriptsize $B'_m$}
\put(-7,6){\footnotesize $\gamma$}
\put(-7.4,3.6){\scriptsize $B$}
\put(-3.1,3.6){\scriptsize $B'$}
\put(-6.7,1.3){\footnotesize $x$}
\put(-3.5,1.3){\footnotesize $y$}
\end{picture}
\caption{\label{fig:figurecurvecand}
The curve $\gamma$ of $Z_\rho$ is minimizing, up to a multiplicative constant, for the length $\ell_\varepsilon$.}
\end{figure}
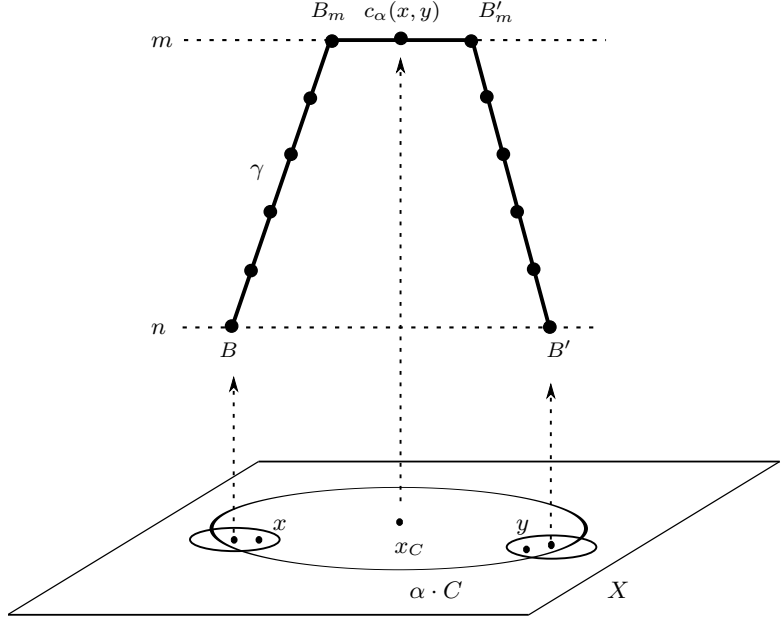

\begin{proof}[Proof of Proposition \ref{controlepsilon}]
Let $x$ and $y$ be two points in $X$, from (H3) and Lemma \ref{longarete}, there exists $n_0 $ such that if $n\geq n_0$ and if $\gamma=\left\{\left(B_i,B_{i +1}\right)\right\}_{i=1}^{N-1}$ is a curve with $B_1=B$ and $B_N= B'$, where $B$ and $B'$ are elements in $\mathcal{S}_n$ such that $x\in B$ and $y\in B'$, then
$$\ell_\varepsilon\left(\gamma\right)\gtrsim\sum\limits_{i = 1}^{N-1}\frac{\pi_\varepsilon(B_i) + \pi_\varepsilon(B_{i +1})}{2}\geq\frac{1}{2}L_{\rho_\varepsilon}\left(\gamma\right)\geq\frac{1}{2C_1^\varepsilon}\cdot\pi_\varepsilon\left(c(x,y)\right).$$
This and Lemma \ref{bornesup} imply that for any $\varepsilon\in (0,1]$, the boundary at infinity
$\partial_\varepsilon Z_d$ is homeomorphic to $X$. Moreover, for all $x,y\in X$, we have
$$d_\varepsilon(x,y)\asymp \pi(c(x,y))^\varepsilon,$$
where $c(x,y)$ is the center of $x$ and $y$ in $Z_d$. On the other hand, we know that there exists $\varepsilon_0>0$ small enough, depending only on the hyperbolicity constant of $Z_\rho$, such that for all $x, y\in X$, we have
$$d_{\varepsilon_0}(x,y)\asymp e^{-\varepsilon_0(x | y)_\rho}.$$
But then, for $\varepsilon=1$, we obtain
$$d_1(x,y)^{\varepsilon_0}\asymp\pi(c (x, y))^{\varepsilon_0}\asymp d_{\varepsilon_0}(x, y)\asymp e^{-\varepsilon_0(x|y)_\rho}.$$
That is, $\theta_\rho=d_1$ is a visual metric and in addition $\pi(c(x,y))\asymp e^{-(x|y)_\rho}$. This finish the proof of the proposition.
\end{proof}

\begin{remark}
This proposition can be interpreted as an analogue of the Gehring-Hayman theorem for Gromov-hyperbolic spaces (see Theorem 5.1 of \cite{BHK}). The assumption (H2) is equivalent to a Harnack type inequality. The proposition says that geodesics of $Z_\rho$ are minimizers, up to a multiplicative constant, for the length $\ell_\varepsilon$. Indeed, given two points $x,y\in X$, if $n\geq m=\left|c(x,y)\right|$ and $B,B'\in\mathcal{S}_n$ are such that $x\in B$ and $y\in B'$, then the curve $\gamma=\left[B, B_m\right]\ast\left(B_m, C\right)\ast\left(C,B'_m\right)\ast\left[B'_m,B'\right]$, where $C\in c(x,y)$ and $B_m,B'_m\in\mathcal{S}_m$ are the parents of $B$ and $B'$ respectively, has an $\varepsilon$-length comparable to $\pi_\varepsilon\left(c(x,y)\right)$. Therefore, this curve is minimizing up to a multiplicative constant ---for $n\geq n_0 $. The important point here is the fact that one can control the visual parameter $\varepsilon$ using the hypothesis (H3). See Figure \ref{fig:figurecurvecand}.
\end{remark}

The third step is to show that for the distance $\theta_\rho$, the $p$-dimensional Hausdorff measure is Ahlfors regular. We will use the assumption (H4) to construct a measure $\mu$ on $X$ which is comparable to the $p$-dimensional Hausdorff measure.

Let $\omega:\mathcal{S}\to (0, +\infty)$ be given by $\omega(B)=\rho(B)^p$. We can define by induction a sequence of purely atomic measures $\mu_n$, with atoms on $X_n$ the centers of the elements in $\mathcal{S}_n$, by setting $\mu_0(x_w)=1$ for the sole point $w\in\mathcal{S}_0$, and
$$\mu_{n +1}(x_B)=\omega(B)\mu_n(x_{B'}),$$ for $B\in T_n(B')$ and $B'\in\mathcal{S}_n$. That is, $\mu_n(x_B)=\pi(B)^p$ if $B\in\mathcal{S}$, and if we denote $\delta_x$ the Dirac measure at $x$, we can write
$$\mu_n=\sum\limits_{B\in\mathcal{S}_n}\pi(B)^p\delta_{x_B}.$$
Recall that we write $X(B)$ for the set of centers of the elements $B'\in D(B)$. Note that according to (H4), for all $n\geq 0$, $B\in\mathcal{S}_n$ and $l\geq n$, we have
$$\frac{1}{K_2}\mu_n(x_B)\leq\mu_l\left(X(B)\right)\leq K_2\mu_n(x_B).$$
In particular, we have $K_2^{-1}\leq \mu_n(X)\leq K_2$ for all $n\geq 0$. Moreover, according to (H1) and (H2):
\begin{enumerate}
\item[(i)] There exists a constant $c\in(0,1)$ such that $c\leq\omega(B)\leq 1$ for all $B\in\mathcal{S}$.
\item[(ii)] There exists a constant $K_0'=K_0^p\geq 1$ such that if $B,B'\in\mathcal{S}_n$ satisfy $B\sim B'$, then
\begin{equation}\label{h2mu}
\frac{\mu_n(x_B)}{\mu_n(x_{B'})}\leq K_0'.
\end{equation}
\end{enumerate}
Let $\mu$ be any weak limit of the sequence $\mu_n$. More precisely, there exists a subsequence $\mu_{n_i}$ which weakly converges to the measure $\mu$ on $X$. To simplify notation, we remove the sub-index $i$.

\begin{lemma}
Let $x$ and $y$ be two points in $X$ and let $r=d(x,y)$. Then 
\begin{equation}\label{muvspi}
\mu\left(B(x,r)\right)\asymp \pi\left(c(x,y)\right)^p.
\end{equation}
\end{lemma}
\begin{proof}
Let $m=|c(x,y)|$ so that $(\alpha-1)\kappa r_{m +1}<r\leq 2\alpha\kappa r_{m}$ holds. Let's start with the lower bound. There exists $B_1\in\mathcal{S}_{m +2}$ such that $d(x,x_{B_1})<\kappa r_{m +2}$. Indeed, the balls $B\left(x_B, \kappa r_{m +2}\right)$, with $B\in\mathcal{S}_{m+2}$, form a covering of $X$. For simplicity, write $x_1:=x_{B_1}$. According to the inclusion (\ref{inclusion1}), we have $X(B_1)\subset B\left(x,\frac{r}{2}\right)$, because $r_{m+2}=a^{-(m+2)}<r$. Therefore, for all $n\geq m+2$, we obtain
\begin{equation}\label{mukinf}
\mu_n\left(\overline{B}\left(x,\frac{r}{2}\right)\right)\geq\mu_n(X(B_1))\geq K_2^{-1}\mu_{m+2}(z_1).
\end{equation}
Take $B\in\mathcal{S}_m$ such that $d(x,x_B)\geq r+ 2\kappa r_{m}$. By property (\ref{inclusion2}), we know that $X(B)\cap B(x,r)=\emptyset$. This implies that for all $n\geq m$, if an element $B$ of $\mathcal{S}_n$ is such that $x_B\in B(x,r)$, then its $m$-generation parent $g(B)_m$ belongs to $B(x,r+2\kappa r_{m})$. Thus,
\begin{equation}\label{muksup}
\mu_n\left(B(x,r)\right)\leq K_2\mu_m\left(B(x,r+2\kappa r_{m})\right).
\end{equation}
Making $n\to +\infty$ (in the subsequence $n_i$), from (\ref{mukinf}) and (\ref{muksup}), one concludes that
\begin{align}\label{musup}
\mu\left(B(x,r)\right)&\leq\liminf \mu_n\left(B(x,r)\right)\leq K_2\mu_m\left(B(x,r+2\kappa r_{m})\right)\text{ and }\\ \label{muinf}
\mu(B(x,r))&\geq\mu\left(\overline{B}\left(x,\frac{r}{2}\right)\right)\geq\limsup \mu_n\left(\overline{B}\left(x,\frac{r}{2}\right)\right)\geq K_2^{-1}\mu_{m+2}(x_1).
\end{align}
Let $Y=X_m\cap B(x,r+2\kappa r_{m})$ and let $B_2=g(B_1)_m$; we denote $x_2$ the center of $B_2$. On the one hand, recall that from (i) $\omega\geq c$, so we have $\mu_{m +2}(z_1)\geq c^{2}\mu_m(z_2)$. Moreover, the cardinal number of $Y$ is uniformly bounded by a constant $M$, which depends only on the doubling constant of $X$, so
$$\mu_m\left(B(x,r+2\kappa r_{m})\right)=\sum\limits_{z\in Y}\mu_m(z)\leq M\max\left\{\mu_m(z):z\in Y\right\}.$$
It remains to compare $\mu_m(z)$ with $\mu_m(x_2)$ for all $z\in Y$. If $z\in Y$, then
$$d(z, x_2)\leq d(z, x)+d(x,x_1)+d(x_1,x_2)\leq r+2\kappa a^{-m}+\kappa a^{-(m+2)}+\kappa\epsilon_m$$
$$\leq \kappa\left(2\alpha+3+\tau\right)a^{-m}\leq \lambda\kappa a^{-m},$$
where the last two inequalities are true by (\ref{constantes}) and the choice of $\lambda$.
Thus, according to item (ii), there exists a constant $K_0'\geq 1$ such that $\mu_m(z)\leq K_0'\mu_m(x_2)$ for all $z\in Y$. Therefore, we obtain
\begin{equation}\label{muposta}
c^2K_2^{-1}\mu_m(x_2)\leq K_2^{-1}\mu_{m+2}(x_1)\leq\mu\left(B(x,r)\right)\leq M\cdot K_0'\cdot K_2\cdot \mu_m(x_2).
\end{equation}
Let $z_0$ be the center of an element $C\in c(x,y)$, then
$$d(x_2,z_0)\leq d(x_2,x_1)+d(x_1,x)+d(x,z_0)\leq \epsilon_m+ a^{-(m+2)}+\alpha a^{-m}\leq \lambda a^{-m},$$
so $B_2\sim C$; recall that $B_2\sim C$ means that $B_2$ and $C$ are the ends of a horizontal edge. According to (H2), we have
\begin{equation}\label{muvscentro}
\pi(B_2)^p\asymp \pi(C)^p\asymp \pi\left(c(x,y)\right)^p.
\end{equation}
Finally, (\ref{muposta}) and (\ref{muvscentro}) imply (\ref{muvspi}).
\end{proof}

\begin{proposition}[Ahlfors regularity]\label{ahregular}
The $p$-dimensional Hausdorff measure of the distance $\theta_\rho$ of Proposition \ref{controlepsilon} is regular.
\end{proposition}
\begin{proof}
We show that $(X,\theta,\mu)$ is $p$-regular. Write $B_d$ and $B_\theta$ for the balls in the metric $d$ and $\theta_\rho$ respectively. Let $x\in X$ and $0<r<1$, we take $y_0$ and $y_1$ in $X$ such that
\begin{align*}
r_0&=d(y_0,x)=\min\left\{d(w,x):\theta_\rho(w,x)\geq r\right\},\\
r_1&=d(y_1,x)=\max\left\{d(w,x):\theta_\rho(w,x)\leq r\right\}.
\end{align*}
So we have
$$B_d(x,r_0)\subset B_\theta(x,r)\subset B_d(x,r_1).$$
Since $\theta_\rho(y_i,x)^p\asymp \pi\left(c(x,y_i)\right)^p\asymp \mu\left(B_d(x,r_i)\right)$ by (\ref{muvspi}) and $\theta_\rho(y_1,x)\leq r\leq\theta_\rho(y_0,x)$, we obtain
$$r^p\lesssim \pi\left(c(x,y_0)\right)^p\asymp\mu\left(B_d(x,r_0)\right)\leq\mu\left(B_\theta(x,r)\right)\leq \mu\left(B_d(x,r_1)\right)\asymp \pi\left(c(x,y_1)\right)^p\lesssim r^p.$$
Then $\mu\left(B_\theta(x,r)\right)\asymp r^p$. This proves the proposition.
\end{proof}

\subsection{Dimension control and simplification of the hypotheses}\label{sec:controldim}

We end this section by simplifying the hypothesis of Theorem \ref{description} to facilitate its application in
concrete situations, like in the following sections. We always assume that $X$ is a compact, doubling and uniformly perfect metric space. We continue to assume that $a$ and $\lambda$ verify (\ref{constantes}). 

We need some notation. Let $\gamma=\{(B_i,B_{i +1})\}_{i=1}^{N-1}$ be a path of edges in $Z_d$, we say that $\gamma$ is a \emph{horizontal path of level $k\geq1$} if $B_i\in\mathcal{S}_k$ for all $i=1,\ldots,N$. We adopt the convention that for such a path $\gamma$ the point $z_i\in X$ denotes the center of $B_i$ for $i=1,\ldots,N$. Denote by $\Gamma_{k +1}\left(B\right)$, where $B\in\mathcal{S}_k$ and $k\geq 0$, the family of horizontal paths $\gamma=\left\{\left(B_i,B_{i+1}\right)\right\}_{i=1}^{N-1}$ of level $k+1$ such that $z_{1}\in B$, $z_{i}\in 2\cdot B$ for $i=2,\ldots,N-1$ and $z_{N}\in X\setminus 2\cdot B$. The aim is to show the following theorem.

\begin{theorem}[Dimension control]\label{simplification}
Let $p> 0$. There exists $\eta_0\in(0,1)$, which depends only on $p$, $\lambda$, $\kappa$ and the doubling constant of $X$, such that if there exists a function $\sigma:\mathcal{S}\to\R_+$ which verifies:
\begin{itemize} 
\item[(S1)] for all $B\in\mathcal{S}_k$ and $k\geq 0$, if $\gamma=\left\{B_i\right\}_{i=1}^N$ is a path in $\Gamma_{k+1}(B)$, then
\begin{equation}
\sum\limits_{i=1}^N\sigma\left(B_i\right)\geq 1,
\end{equation}
\item[(S2)] and for all $k\geq 0$ and all $B\in\mathcal{S}_k$, we have
\begin{equation}\label{sumachica}
\sum\limits_{B'\in T_k(B)}\sigma(B')^p\leq \eta_0,
\end{equation}
\end{itemize}
then there exists an Ahlfors regular distance $\theta\in\mathcal{J}\left(X, d\right)$ of dimension $p$. Therefore, the Ahlfors regular conformal dimension of $X$ is smaller than or equal to $p$.
\end{theorem}

The important point of the proposition is that we can get rid of the condition (H2), as long as the sum (\ref{sumachica}) is sufficiently small. We remark that even if $\eta_0$ depends on $p$, we can take $\eta_0$ to be uniform if $p$ varies in a bounded interval of $(0,+\infty)$. We derive Theorem \ref{simplification} from Proposition \ref{controldim} below. We start by modifying the hypothesis (H3) of Theorem \ref{description}. The purpose is to state a condition on the lengths of horizontal curves which implies (H3). 

Let $\rho:\mathcal{S}\to \R_+$ be a function, we define $\rhostar:\mathcal{S}\to\R_+$ by
$$\rhostar(B)=\min\limits_{B'\sim B\in\mathcal{S}}\rho(B),\ \ \textrm{for } B\in\mathcal{S}.$$
If $\gamma$ is a horizontal path of level $k$, we define
$$L_h(\gamma,\rho)=\sum\limits_{j=1}^{N-1}\rhostar(B_j)\wedge\rhostar(B_{j+1}).$$
The $h$ stands for horizontal. We have the following result.

\begin{proposition}\label{controldim}
Let $(X,d)$ be a compact, doubling and uniformly perfect metric space. Consider the graph $Z_d$ constructed in the previous section with $a$ and $\lambda$ satisfying (\ref{constantes}). Assume there exist $p>0$ and a function $\rho:\mathcal{S}\to (0,+\infty)$, which satisfy the hypothesis (H1), (H2), (H4) of Theorem \ref{description}, and also
\begin{enumerate}
\item[(H3')] for all $k\geq 0$ and all $B\in\mathcal{S}_k$, if $\gamma\in \Gamma_{k+1}\left(B\right)$, then $L_h(\gamma,\rho)\geq 1$.
\end{enumerate}
Then the function $\rho$ also verifies the hypothesis (H3).
\end{proposition}

We first prove Proposition \ref{controldim}. We divide the proof into several lemmas. We start with the following remark: by Lemma \ref{longarete}, we have $\ell_1(\gamma)\asymp L_\rho(\gamma)$; recall that we denote by $\ell_1$ the length $\ell_\varepsilon$ for $\varepsilon=1$. Thus, to control the length $L_\rho(\gamma)$ of curves in $Z_d$, in order to show (H3), it is enough to work with the length function $\ell_1$. For technical reasons, we modify the length function $\ell_1$ by replacing it with another bi-Lipschitz equivalent one. For $k\geq 0$, we define $\pi^*:\mathcal{S}\to (0,+\infty)$ by setting
$$\pi^*(B)=\min\limits_{B'\sim B\in\mathcal{S}}\pi(B').$$
From (\ref{czero}), one has 
\begin{equation}\label{asterisque}
\pi(B)\geq\pi^*(B)\geq \frac{1}{K_0}\pi(B)\ \ \text{for all }B\in\mathcal{S}.
\end{equation}
This and (H1), imply that if $B\in\mathcal{S}_{k +1}$ and if $B'=g(B)_k$, then
\begin{equation}
\frac{1}{K_0}\pi^*(B)\leq \frac{1}{K_0}\pi(B)\leq \frac{1}{K_0}\pi(B')\leq \pi^*(B'),\label{asterisque1}
\end{equation}
and
\begin{equation}
\pi^*(B')\leq \pi(B')\leq\frac{1}{\eta_-}\pi(B)\leq\frac{K_0}{\eta_-}\pi^*(B).\label{asterisque2}
\end{equation}
Let $e=(B,B')$ be an edge of $Z_d$, by Lemma \ref{longarete} and the inequalities (\ref{asterisque1}) and (\ref{asterisque2}), we have:
\begin{enumerate}
\item if $e$ is horizontal with $B,B'\in\mathcal{S}_k$, then
\begin{equation}\label{longhorizontale}
\ell_1(e)\asymp \frac{\pi(B)+\pi(B')}{2}\asymp \pi^*(B)\wedge\pi^*(B').
\end{equation}
\item and if $e$ is vertical with $B'\in\mathcal{S}_{k}$ and $B=g(B')_{k-1}$, then
\begin{equation}\label{longverticale}
\ell_1(e)\asymp \pi(B')\asymp \pi^*(B').
\end{equation}
\end{enumerate}
Let $K_5= K_0/ \eta_-$. We simply change the length of an edge in $Z_d$ by setting 
\begin{displaymath}
\hat{\ell}_1(e)=\begin{cases}\pi^*(B)\wedge\pi^*(B'),&\text{ if }e=(B,B')\text{ is a horizontal edge}.\\
K_5\pi^*(B'),&\text{ if } e=(B,B')\text{ is a vertical edge}.\end{cases}
\end{displaymath}
This definition is inspired by a similar one used in \cite{KL}. From (\ref{longhorizontale}) and (\ref{longverticale}), the length functions $\ell_1$ and $\hat{\ell}_1$ are bi-Lipschitz equivalent. We note in particular that the length distance induced by $\hat{\ell}_1$ is bi-Lipschitz equivalent to $d_1$ ($d_\varepsilon$ with $\varepsilon=1$).

The first step is to estimate the length $\hat{\ell}_1(\gamma)$ of certain curves in $Z_d$. The first type of curves, discussed in the following lemma, are horizontal curves which have a large enough, relative to the scale, ``diameter'', i.e. curves which verify the statement of (H3').

\begin{lemma}\label{type1}
Let $k\geq 0$ and $B\in\mathcal{S}_k$. Consider $\gamma=\{(B_i,B_{i+1})\}_{i=1}^{N-1}$ a horizontal path of level $k+1$, such that $z_i\in 3\cdot B$ for all $i$,  $z_1\in B$ and $z_N\notin 2\cdot B$. We denote $B'\in\mathcal{S}_k$ the parent of $z_1$. Then
\begin{equation}\label{preestimation}
\hat{\ell}_1(\gamma)=\sum_{i=1}^{N-1}\pi^*(B_i)\wedge\pi^*(B_{i+1})\geq \max\left\{\pi^*(B'),\pi^*(B)\right\}.
\end{equation}
\end{lemma}
\begin{proof}
First we show that for all $j=1,\ldots, N$, we have
\begin{equation}\label{eqpreestimation}
\pi^*(B_j)\geq \max\left\{\pi^*(B'),\pi^*(B)\right\}\rho^*(B_j).
\end{equation}
Let $A\sim B_j\in \mathcal{S}_{k+1}$ be such that $\pi^*(B_j)=\pi(A)$ and let $A'=g(A)_k$. Then
\begin{align*}
d(x_B,x_{A'})&\leq d(x_B,z_j)+d(z_j,x_A)+d(x_A,x_{A'}) \leq 3\kappa r_{k}+2\lambda\kappa r_{k+1}+ \kappa r_k\\
&=\kappa\left(4+\frac{2\lambda}{a}\right)r_{k}< \lambda \kappa r_{k}.
\end{align*}
The last inequality follows from the choice made in (\ref{constantes}). Since $d(x_{B'},x_B)\leq 2\kappa r_{k}\leq \lambda\kappa r_{k}$, we also have $x_B\in \lambda\cdot B'\cap \lambda\cdot A'$. This implies $A'\sim B$ and $A'\sim B'$. Therefore, $\max\left\{\pi^*(B'),\pi^*(B)\right\}\leq \pi(A')$ and
$$\pi^*(B_j)=\pi(A)=\pi(A')\cdot \rho(A)\geq \max\left\{\pi^*(B'),\pi^*(B)\right\}\min\limits_{C\sim B_j\in\mathcal{S}_{k+1}}\rho(C).$$
This shows (\ref{eqpreestimation}). By (H3'), we know that $L_h(\gamma,\rho)\geq 1$ so
\begin{align*}
\sum_{j=1}^{N-1}\pi^*(B_j)\wedge \pi^*(B_{j+1})&\geq \sum_{j=1}^{N-1}\max\left\{\pi^*(B'),\pi^*(B)\right\}\left(\rho^*(B_j)\wedge \rho^*(B_{j+1})\right)\\
&=\max\left\{\pi^*(B'),\pi^*(B)\right\}L(\gamma,\rho)\geq \max\left\{\pi^*(B'),\pi^*(B)\right\}.
\end{align*}
This ends the proof of the lemma.
\end{proof}

The second type of curves is the set of curves which possess a vertical edge, despite a small ``diameter''. The definition of $\hat{\ell}_1(e_v)$ for $e_v$ a vertical edge, can be used to estimate their length from below. More precisely: if $e_v=(x,y)$ is a vertical edge, with $B\in\mathcal{S}_{k+1}$ and $B'=g(B)_k$, and if $e_h=(B',C)$ is a horizontal edge, then
\begin{equation}\label{type2}
\hat{\ell}_1(e_h)\leq\hat{\ell}_1(e_v).
\end{equation} 
In fact, by definition and from (\ref{asterisque2}), we have $\hat{\ell}_1(e_h)\leq\pi^*(B')\leq C_1\pi^*(B)=\hat{\ell}_1(e_v).$

Let $\gamma=\{e_i=(B_i,B_{i+1})\}_{i=1}^{N-1}$ be an edge-path in $Z_d$, we say that \emph{$\gamma$ is of level at most $k$} if $\left|B_i\right|\leq k$ for all $i$. 

\begin{lemma}\label{mascorta}
Let $A_1$ and $A_2$ be two elements of $\mathcal{S}_{k +1}$ such that $4\kappa r_k<d(y_1,y_2)$, where we write $y_i:=x_{A_i}$. Let $\gamma=\{(B_j,B_{j +1})\}_{j=1}^{N-1}$ be a path of level at most $k+1$ joining $A_1$ and $A_2$. Then there exists a path of level at most $k$, $\gamma'= \{(C_i,C_{i +1})\}_{i=1}^{N'-1}$, such that:
\begin{enumerate}
\item $C_1,C_{N'}\in\mathcal{S}_k$ are the parents of $A_1$ and $A_2$ respectively, and
\item $\hat{\ell}_1(\gamma')\leq\hat{\ell}_1(\gamma)$.
\end{enumerate}
\end{lemma}
\begin{proof}

\begin{figure}
\centering
\setlength{\unitlength}{1cm}
\begin{picture}(12.5,4)
\includegraphics[scale=0.8]{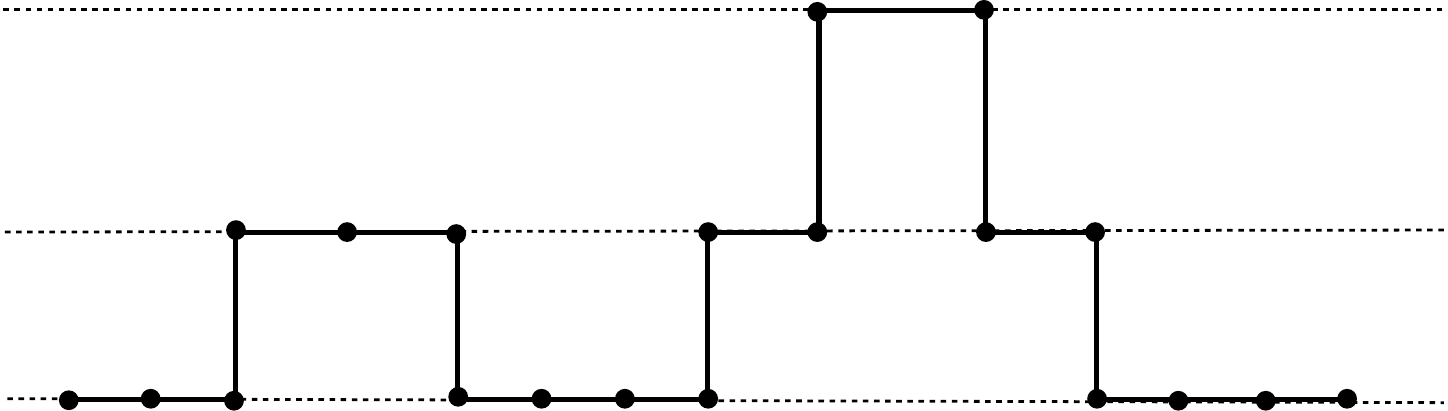}
\put(0.1,0){\footnotesize $k+1$}
\put(0.1,1.4){\footnotesize $k$}
\put(0.1,3.2){\footnotesize $k-1$}
\put(-0.8,0.3){\footnotesize $t_3$}
\put(-11.2,0.3){\footnotesize $s_1$}
\put(-9.8,1.7){\footnotesize $t_1$}
\put(-8.05,1.7){\footnotesize $s_2$}
\put(-6,1.7){\footnotesize $t_2$}
\put(-2.85,1.7){\footnotesize $s_3$}
\end{picture}
\caption{Proof of Lemma \ref{mascorta}: decomposing the path in sub-paths $\gamma_i$.
\label{fig:decomposition}}
\end{figure}
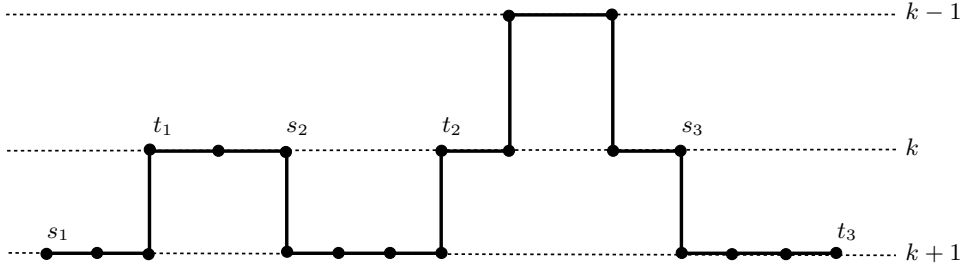

Let $\gamma$ be such a path of level at most $k+1$ with $B_1,B_N\in\mathcal{S}_{k +1}$. We denote $x_j$ the center of the ball $B_j$. We can decompose $\gamma$ in sub-paths of level at most $k$, or level equal to $k+1$. Let $s_1=1$, define inductively positive integers $s_i$ and $t_i$ as follows:
$$t_i=\min\left\{j>s_i:|B_j|\leq k\text{ or } j=N\right\},$$
$$s_{i+1}=\min\left\{j\geq t_i:|B_{j+1}|= k+1\right\}.$$
We stop when $t_i=N$ for some $i:=M$. Note that $|B_{s_1}|=|B_{t_M}|=k+1$, and for the others $|B_{s_i}|=|B_{t_i}|=k$ (see Figure \ref{fig:decomposition}). Since we are trying to bound from below the length of $\gamma$, we can assume without loss of generality that $\gamma$ is a path without self-intersections; thus $B_{s_i}\neq B_{t_i}$ for all $i$.

For each $i\in\{1,\ldots, M\}$, set $\gamma_i=\{(B_j,B_{j +1})\}_{j=s_i} ^ {t_i-1}$; we will construct $\gamma'_i$ of level at most $k$ such that $\hat{\ell}_1\left(\gamma'_i\right)\leq\hat{\ell}_1\left(\gamma_i\right)$. We let the cases $i=1$ and $i=M$ to the end.

Fix $i\in\{2,\ldots, M-1\}$ and we write $C=B_{s_i}$. We divide the construction into two cases.

\begin{flushleft}
\underline{First case:} $x_j\in B\left(x_C, 2\kappa r_k\right)$ for all $j\in\{s_i+1,\ldots,t_i-1\}$.
\end{flushleft}
In this case, $B_{s_i}$ and $B_{t_i}$ in $\mathcal{S}_k$ are the parents of $B_{s_i+1}$ and $B_{t_i-1}$ respectively. Since $d\left(x_{s_i+1},x_{t_i-1}\right)\leq 4\kappa r_{k}$, by item (iv) of Lemma \ref{propgraphe}, we know that $e =\left(B_{s_i},B_{t_i}\right)$ is an edge of $Z_d$. So we set $\gamma'_i=e$. Since $e'=\left(B_{t_i-1},B_{t_i}\right)$ is a vertical edge, from (\ref{type2}), we obtain
$$\hat{\ell}_1\left(\gamma'_i\right)=\hat{\ell}_1(e)\leq \hat{\ell}_1(e')\leq \hat{\ell}_1\left(\gamma_i\right).$$

\begin{flushleft}
\underline{Second case:} There exists $j_1\in\{s_i+1,\ldots,t_i-1\}$ such that $x_{j_1}\notin B\left(x_C,2\kappa r_{k}\right)$.
\end{flushleft}
We can assume that $j_1$ is the first index with this property. The path $\{(B_j, B_{j +1})\}_{j=s_i+1}^{t_i-2}$ is of level equal to $k+1$. We decompose this path again to use the estimate (\ref{preestimation}). We denote $j_0=s_i+1$ and $C_0=C$. Suppose $j_l$ and $C_l$ defined, we denote $z_l$ the center of $C_l$, and if $j_l<t_i-1$, we define
$$j_{l+1}=\min\left\{j_l< j\leq t_i-1:x_{j}\notin B\left(z_l,2\kappa r_{k}\right)\text{ or } j=t_i-1\right\},$$ 
and let $C_{l+1}\in\mathcal{S}_k$ be the parent of $B_{j_{l+1}}$; we also denote $z_{l+1}$ the center of $C_{l+1}$. In particular, we have $x_{j_{l+1}}\in B\left(z_{l+1},\kappa r_{k}\right)$. Thus, we obtain a sequence $\{j_0,\ldots,j_{L_i}\}\subset \{s_i+1,\ldots,t_i-1\}$ with $j_0=s_i+1$ and $j_{L_i}=t_i-1$. Write $\sigma_l:=\left\{(B_j,B_{j+1})\right\}_{j=j_l}^{j_{l+1}-1}$. 

Let us show that $\sigma_l$ and $z_l$ satisfy the hypotheses of Lemma \ref{type1} for each $l\in\{0,\ldots,L_i-2\}$. We know by construction that $x_{j_{l+1}}\notin B\left(z_l, 2\kappa r_{k}\right)$ and that $x_j\in B\left(z_l, 2\kappa r_{k }\right)$ for all $j_l\leq j<j_{l+1}$. Moreover, since
\begin{multline}
d\left(z_l,x_{j_{l+1}}\right)\leq d\left(z_l,x_{j_{l+1}-1}\right)+d\left(x_{j_{l+1}-1},x_{j_{l+1}}\right)\\ \leq 2\kappa a^{-k}+2\lambda \kappa a^{-(k+1)}=\kappa \left(2+\frac{2\lambda}{a}\right)a^{-k}\leq 3\kappa r_{k},
\end{multline}
---the last inequality follows from the choice made in (\ref{constantes})--- we have
$$\left\{x_j\right\}_{j=j_l}^{j_{l+1}}\subset B\left(z_l,3\kappa r_{k}\right).$$
So, from (\ref{preestimation}), we know that
\begin{equation}\label{jlestimation}
\pi^*(C_l)\leq\hat{\ell}_1\left(\sigma_l\right)\ \ \text{for $l\in\{0,\ldots,L_i-2\}$}.
\end{equation}
By item (v) of Lemma \ref{propgraphe}, $e_l= \left(C_l,C_{l+1}\right)$ is an edge of $Z_d$. In fact, the edge $\left(C_l,B_{j_ {l+1}}\right)\in G_d$, and $C_{l+1}$ is the parent of $B_{j_{l+1}}$. Since for $l=L_i-1$, we have $x_{t_i-1}=x_{j_{L_i}}\in B\left(z_{L_i-1}, 2\kappa r_{k}\right)$, similarly the existence of the edge $e_{t_i}=\left(C_{L_i-1},B_{t_i}\right)$ holds. Moreover, since $\left(B_{t_i-1},B_{t_i}\right)$ is a vertical edge, from (\ref{type2}), we have
\begin{equation}\label{tiestimation}
\hat{\ell}_1(e_{t_i})\leq \hat{\ell}_1\left((B_{t_i-1},B_{t_i})\right).
\end{equation}
Let $\gamma'_i=e_0\ast\cdots\ast e_{L_i-2}\ast e_{t_i}$, then $\gamma'_i$ joins $B_{s_i}$ and $B_{t_i}$. Moreover, from (\ref{jlestimation}) and (\ref{tiestimation}), we have
\begin{align*}
\hat{\ell}_1(\gamma'_i)&=\sum_{l=0}^{L_i-2}\hat{\ell}_1(e_l)+\hat{\ell}_1(e_{t_i})\leq \sum_{l=0}^{L_i-2}\pi^*(C_l)+ \hat{\ell}_1\left((B_{t_i-1},B_{t_i})\right)\\
& \leq \sum_{l=0}^{L_i-2}\hat{\ell}_1(\sigma_l)+ \hat{\ell}_1\left((B_{t_i-1},B_{t_i})\right)
\leq \hat{\ell}_1(\gamma_i).
\end{align*}

Consider the case $i=1$, we do a similar construction to the one above. Let $z$ be the center of $D_1\in\mathcal{S}_k$, the parent of $B_1$, so in particular $x_{1}\in B\left (z, \kappa r_k\right)$. Similarly, divide the construction into two cases. Assume first that $x_j\in B\left(z, 2\kappa r_{k}\right)$ for all $j\in\{1,\ldots,t_1-1\}$. Since $d(x_1, x_N)> 4\kappa r_{k}$, we know that $\gamma_1$ is a proper sub-path of $\gamma$. Thus $\left(B_ {t_1-1},B_{t_1}\right)$ is a vertical edge. An argument similar to that given above shows that $e=(D_1,B_{t_1})$ is an edge of $Z_d$, and that if we set $\gamma'_1= e$, we obtain $\hat{\ell}_1(\gamma'_1)\leq\hat{\ell}_1(\gamma_1)$. If instead, there exists $j$ such that $x_j\notin B\left(z, 2\kappa r_{k}\right)$, as in the second case above let $\gamma'_1=e_0\ast\cdots\ast e_{L_1-2}\ast e_{t_1}$.

Also for $i=M$ we do the same construction. Set $z=x_{s_M}$, so in particular $x_ {s_M+1}\in B\left(z, \kappa r_k\right)$. If $x_j \in B\left(z, 2\kappa r_{k}\right)$ for all $j=s_M+1,\ldots, N$, using the fact that $\left(B_{s_M},B_{s_M+1}\right)$ is a vertical edge, we see as above that it suffices to take $\gamma'_M=\left(B_{s_M}, D_2\right)$, where $D_2\in\mathcal{S}_k$ is the parent of $B_N$. Otherwise, with $j_0=s_M+1$ and $C_0=B_{s_M}$, we do as in the second case above and we obtain $\gamma'_M=e_0\ast\cdots\ast e_{L_M-3}\ast\left(C_{L_M-2},D_2\right)$. If $L_M\leq 2$ we take $(C_0,D_2)$. We must show the existence of the edge $\left(C_{L_M-2},D_2\right)$, which is done similarly as in the other cases.

Finally, we note that if $M=1$, i.e. $\gamma=\gamma_1$ is a path of level $k+1$, we have $s_1=1$ and $t_1=M$. We define similarly $j_0=1$, and $z_0$ the center of $C_0$ the parent parent of $B_1$. We also define by induction
$$j_{l+1}:=\min\left\{j_l<j\leq N:x_j\notin B\left(z_l, 2\kappa r_{k}\right)\text{ or }j=N\right\},$$
and $z_{l+1}$ the center of $C_{l+1}$ the parent of $B_{j_{l +1}}$. We obtain a sequence $\{j_0,\ldots,j_{L}\}\subset\{1,\ldots,N\}$ with $j_0=1$ and $j_{L}= N$. Write $\sigma_l=\left\{(B_j,B_{j+1})\right\}_{j=j_l}^{j_{l+1}-1}$. We show in the same way that both edges $e_l=\left(C_l,C_{l+1}\right)$ and $e= \left(C_{L-2},C_L\right)$ are in $Z_d$. The same arguments as above show that if $\gamma':=e_0\ast\cdots\ast e_{L-3}\ast e$, then $\hat{\ell}_1\left(\gamma'\right)\leq\hat{\ell_1}(\gamma)$.

In conclusion, in both cases, for $i\neq1,M$, we obtain a path $\gamma'_i$ of level at most $k$ joining $B_{s_i}$ and $B_{t_i}$, and length less than or equal to $\hat{\ell}_1(\gamma_i)$. For $i=1$, we obtain such a path joining $D_1\in\mathcal{S}_k$, the parent of $B_1$, to $B_{t_1}$. And for $i=M$, we obtain such a path joining $B_{s_M}$ to $D_2\in\mathcal{S}_k$, the parent of $B_N$. Finally, if we denote $\zeta_i=\left\{(B_j,B_{j+1})\right\}_{j=t_i}^{s_{i+1}-1}$ for $i=1,\ldots,M-1$, it suffices to take
$$\gamma'=\gamma'_1\ast\zeta_1\ast\gamma'_2\cdots\ast\zeta_{M-1}\ast\gamma'_M.$$
This completes the proof of the lemma.
\end{proof}

We take $\alpha=8$ in the statement of (H3), and to simplify the notation, we write $c(x,y)$ instead of $c_\alpha(x,y)$.
\begin{lemma}\label{borneinf}
There exists a uniform constant $K_6\geq 1$ with the following property: for all $x,y\in X$, there exists $k_0$ depending on $x$ and $y$, such that for $k\geq k_0$, if $B,B'\in\mathcal{S}_k$ are such that $x\in B$ and $y\in B'$, then any edge-path $\gamma$ joining $B$ and $B'$ verifies
\begin{equation}\label{eqborneinf}
\hat{\ell}_1(\gamma)\geq \frac{1}{K_6}\cdot \pi\left(c(x,y)\right).
\end{equation}
\end{lemma}
\begin{proof}
Let $m=|c(x,y)|$ be the level of the center of $x$ and $y$. We suppose $k\geq m+1$. By definition of $m$, we know that $d(x,y)\geq 7\kappa r_{m+1}$ and
\begin{equation}\label{distxy}
d(x_B,x_{B'})\geq d(x,y)-2\kappa r_k\geq 5\kappa r_{m+1}.
\end{equation}
Let $\gamma$ be an edge-path joining $B$ and $B'$. The idea is to inductively use Lemma \ref{mascorta} to find a path of level at most $m+1$, and of length smaller than or equal to that of $\gamma$. We divide the proof into two cases.

\begin{flushleft}
\underline{First case:} the path $\gamma$ is of level at most $k$.
\end{flushleft}

From (\ref{distxy}), we can apply Lemma \ref{mascorta} at least once. Set $\gamma_k=\gamma$, and suppose constructed the paths $\gamma_i$ for $i\in \{l,l+1,\ldots,k\}$, which verify the following properties:
\begin{itemize}
\item $\gamma_i$ is of level at most $i$ and joins the elements $B_i,B'_i\in\mathcal{S}_i$,
\item $B_i,B'_i$ are the parents of $B_{i+1},B'_{i+1}$ respectively, and
\item $\hat{\ell}_1(\gamma_i)\leq \hat{\ell}_1(\gamma_{i+1})$.
\end{itemize}
We denote $x_i$ and $y_i$ the centers of the elements $B_i$ and $B_i'$ respectively. Then, we recall that $\tau=\frac{a}{a-1}$,
\begin{equation*}
d(x_l,y_l)\geq d(x_k,y_k)-2\sum_{i=l}^{k-1}\kappa a^{-i}\geq 5\kappa a^{-(m+1)}-2\tau\kappa a^{-l}.
\end{equation*}
Using (\ref{constantes}), we have $d(x_l,y_l)\geq 4\kappa a^{-(l-1)}$ if $l\geq m+2$. But this allows us to apply, provided that $l\geq m+2$, at least one more time Lemma \ref{mascorta} to obtain a path $\gamma_{l-1}$. In conclusion, we know that there exists a path $\gamma_{m+1}$ of level at most $m+1$ joining $B_{m+1}$ and $B'_{m+1}$, the parents in $\mathcal{S}_{m+1}$ of $B$ and $B'$ respectively, with the property that $\hat{\ell}_1(\gamma_{m+1})\leq\hat{\ell}_1(\gamma)$. Furthermore, if $B_m\in\mathcal{S}_m$ is the parent of $B_{m+1}$, then
\begin{equation}\label{cercadelcentro}
d(x_m,x)\leq \sum_{i=m}^{k-1}\kappa a^{-i}+\kappa a^{-k}\leq \tau \kappa a^{-m},
\end{equation}
where we write $x_m$ for the center of $B_m$. Thus, if $A\in c(x,y)$, the fact that $d(x_A,x)\leq 8\kappa a^{-m}$ and (\ref{cercadelcentro}) gives $d(x_A,x_m)\leq \lambda\kappa a^{-m}$. That is, $e=(A,B_m)$ is an edge of $Z_d$ and therefore $\pi(B_m)\geq K_0^{-1}\pi(A)$. Finally, since $\hat{\ell}_1(\gamma_{m+1})\geq K_0^{-2}\pi^*(B_{m+1})\geq K_0^{-3 }K_5^{-1}\pi(B_m)$, we obtain
\begin{equation}\label{longgamma}
\hat{\ell}_1(\gamma)\geq \hat{\ell}_1(\gamma_m)\geq \frac{1}{K_0^{4}K_5}\pi(c(x,y)).
\end{equation}

\begin{flushleft}
\underline{Second case:} $\gamma$ is a path of level at least $k+1$.
\end{flushleft}
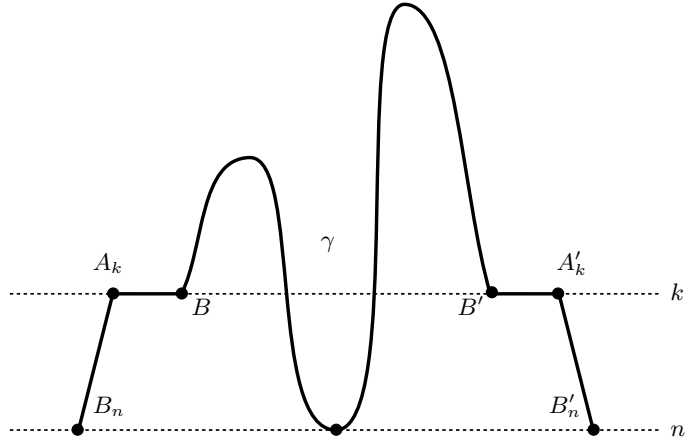
\begin{figure}
\centering
\setlength{\unitlength}{1cm}
\begin{picture}(9,6.5)
\includegraphics[scale=0.8]{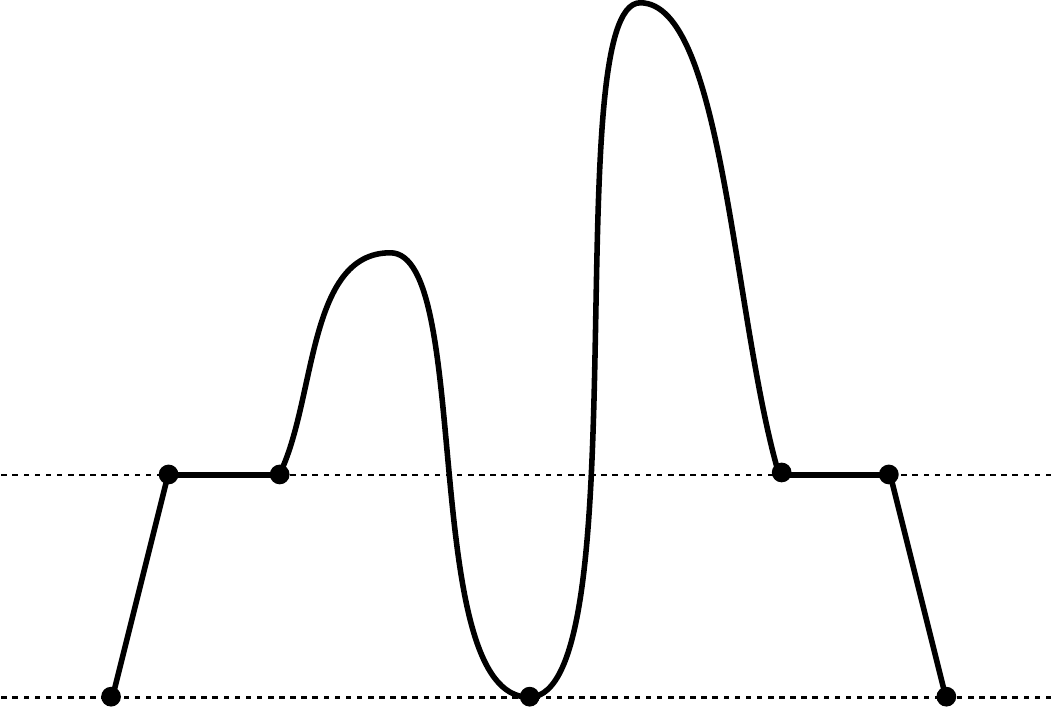}
\put(0.1,0){\footnotesize $n$}
\put(0.1,1.8){\footnotesize $k$}
\put(-1.5,0.3){\footnotesize $B'_n$}
\put(-7.5,0.3){\footnotesize $B_n$}
\put(-7.5,2.2){\footnotesize $A_k$}
\put(-1.4,2.2){\footnotesize $A'_k$}
\put(-2.7,1.6){\footnotesize $B'$}
\put(-6.2,1.6){\footnotesize $B$}
\put(-4.5,2.5){\footnotesize $\gamma$}
\end{picture}
\caption{\label{fig:figurecurvanivel}
Proof of Lemma \ref{borneinf}, second case.}
\end{figure}
Let $k_0$ be large enough such that 
\begin{equation}\label{kzero}
2K_5\sum_{i=k_0}^{\infty}\left(\eta_+\right)^i\leq \frac{1}{2K_0^{4}K_5}\left(\eta_-\right)^m,
\end{equation}
and suppose $k\geq k_0$. Let $n>k$ be the maximal level of a vertex of $\gamma$, and let $B_n,B'_n\in\mathcal{S}_n$ be such that $x\in B_n$ and $y\in B_n'$. Set $A_k=g(B_n)_k$ and $A'_k=g(B'_n)_k$. We write $g_x=[A_k,B_n]$ and $g_y=[A'_k,B'_n]$ for the corresponding geodesic segments. As usual, $x_k$, $x_n$ and $x'_k$ denote the centers of $B$, $B_n$ and $A_k$ respectively. By the triangle inequality and from (\ref{constantes}), we have
$$d(x_k',x_k)\leq d(x'_k,x_n)+d(x_n,x_k)\leq (\tau +2)\kappa a^{-k}\leq \lambda \kappa a^{-k}.$$
So $e_x=(A_k,B)$ is an edge of $Z_d$. Analogously, we see that $e_y=(B',A'_k)$ is an edge of $Z_d$. Then $\gamma_n=g_x\ast e_x\ast\gamma\ast e_y\ast g_y$ is a path of level at most $n$ joining $B_n$ and $B'_n$. From (\ref{longgamma}), we know that 
\begin{equation}\label{longgamman}
\hat{\ell}_1(\gamma_n)\geq \frac{1}{K_0^{4}K_5}\pi(c(x,y))\geq \frac{1}{K_0^{4}K_5}\left(\eta_-\right)^m.
\end{equation}
On the other hand, we have 
$$\hat{\ell}_1(g_x\ast e_x)\leq \left(\eta_+\right)^k+K_5\sum_{i=k+1}^{n}\left(\eta_+\right)^i\leq K_5\sum_{i=k}^{\infty}\left(\eta_+\right)^i.$$
The same computation holds for $e_y\ast g_y$; therefore, we obtain
\begin{align*}
\hat{\ell}_1(\gamma)&=\hat{\ell}_1(\gamma_n)-\hat{\ell}_1(g_x\ast e_x)-\hat{\ell}_1(e_y\ast g_y)\\
&\geq \hat{\ell}_1(\gamma_n)-2K_5\sum_{i=k}^{\infty}\left(\eta_+\right)^i\geq \frac{1}{2K_0^{4}K_5}\pi(c(x,y)).
\end{align*}
The last inequality holds by definition of $k_0$ and (\ref{longgamman}). This completes the proof of the lemma.
\end{proof}

We now give the proof of Theorem \ref{simplification}. We first show two lemmas, the first proof is inspired by the construction of doubling measures of Vol'berg and Koniagyn \cite{VK} (see also \cite{W} and \cite{He}).

\begin{lemma}\label{modificasion}
Suppose we have a function $\pi_0:\mathcal{S}_k\to (0,+\infty)$ which verifies 
\begin{equation}\label{hpi0}
\forall B\sim B'\in\mathcal{S}_k,\ \ \frac{1}{K}\leq \frac{\pi_0(B)}{\pi_0(B')}\leq K,
\end{equation}
where $K\geq 1$ is a constant. Suppose also that we have a function $\pi_1:\mathcal{S}_{k+1}\to (0,+\infty)$ which verifies the following property:
\begin{equation}\label{hpi1}
\forall B\in\mathcal{S}_{k+1},\exists A\in\mathcal{S}_{k}\text{ with }d(x_B,x_A)\leq 2\kappa r_k\text{ and } 1\leq \frac{\pi_0(A)}{\pi_1(B)}\leq K.
\end{equation}
Then there exists a function $\hat{\pi}_1:\mathcal{S}_{k+1}\to \R_+$ such that 
\begin{enumerate}
\item for all $B\sim B'\in\mathcal{S}_{k+1}$, $$\frac{1}{K}\leq \frac{\hat{\pi}_1(B)}{\hat{\pi}_1(B')}\leq K.$$
\item for all $B'\in\mathcal{S}_{k+1}$, we have $\hat{\pi}_1(B')=\pi_1(B')$ or there exists $B\sim B'\in\mathcal{S}_{k+1}$ such that $\hat{\pi}_1(B')=\frac{\pi_1(B)}{K}$. More precisely, in the second case, we have
$$\hat{\pi}_1\left(B'\right)=\frac{1}{K}\max\left\{\pi_1(B):B\sim B'\right\}.$$
\end{enumerate}
\end{lemma}
\begin{proof}
For each pair of neighbors $B,B'\in\mathcal{S}_{k+1}$, we check if the inequalities 
$$\frac{1}{K}\leq\frac{\pi_1(B)}{\pi_1(B')}\leq K,$$ hold or not. Only one of these two inequalities can be false; therefore, we put an oriented edge going from $B$ to $B'$ if $\pi_1(B)>K\pi_1(B')$. The fundamental property is the following: there is no oriented path of edges of length at least two. In fact, suppose that $B\sim B'$ and $B'\sim B''$ are such that $\pi_1(B)>K\pi_1(B')$ and $\pi_1(B')>K\pi_1(B'')$. Then we obtain $\pi_1(B)>K^2\pi_1(B'')$. Since $d\left(x_B,x_{B''}\right)\leq 4\lambda\kappa a^{-(k+1)}\leq 4\kappa a^{-k}$, if we write $A,A''\in\mathcal{S}_k$ such that 
$$d(x_B,x_A)\leq 2\kappa a^{-k}\text{ and }d(x_{B''},x_{A''})\leq 2\kappa a^{-k},$$ and also such that
$$1\leq \frac{\pi_0(A)}{\pi_1(B)}\leq K\text{ and } 1\leq \frac{\pi_0(A'')}{\pi_1(B'')}\leq K,$$
we obtain, from item (iv) of Lemma \ref{propgraphe}, that $A\sim A''$. But since 
$$\pi_0(A)\geq\pi_1(B)>K^2\pi_1(B'')\geq K\pi_0(A''),$$ we get $\pi_0(A)>K\pi_0(A'')$, which is impossible. As a consequence, for all $B\in\mathcal{S}_{k+1}$, the directed edges which have $B$ as an extremity, or all enter or all leave the vertex $B$. We modify $\pi_1$ only in that subset of vertices $B'$ for which all directed edges enter. See Figure \ref{fig:directed}.

\begin{figure}
\centering
\setlength{\unitlength}{1cm}
\begin{picture}(5,5)
\includegraphics{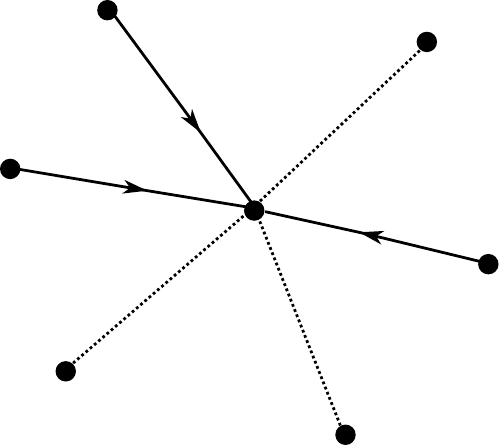}
\put(-0.5,4){\footnotesize $B''$}
\put(-2.5,2.7){\footnotesize $B'$}
\put(-5,2.4){\footnotesize $B$}
\end{picture}
\caption{Typical situation for a vertex $B'\in\mathcal{S}_{k+1}$ with entering edges.
\label{fig:directed}}
\end{figure}

Let $B'\in\mathcal{S}_{k+1}$ be such that there exists at least one entering directed edge. To define $\hat{\pi}_1(B')$ we proceed in the following way. Let $B_i\sim B',i=1,\ldots,l$, be all the neighbors of $B'$, and let $B\in\{B_1,\ldots,B_l\}$ be such that
$$\pi_1(B)\geq \pi_1(B_i),\ \ i=1,\ldots, l.$$
We replace $\pi_1(B')$ by $\frac{\pi_1(B)}{K}$. In other words, we replace it by $\hat{\pi}_1(B')=\alpha\pi_1(B')$, where 
$$\alpha=\frac{\pi_1(B)}{K\pi_1(B')}>1.$$
Thus, for all $i\in\{1,\ldots,l\}$, we have
$$\frac{\pi_1(B_i)}{\hat{\pi}_1(B')}=K\cdot\frac{\pi_1(B_i)}{\pi_1(B)}\leq K,$$
by definition of $B$.
To see the other inequality, let $A_i\in\mathcal{S}_k$ be such that $d(x_{A_i},x_{B_i})\leq 2\kappa a^{-k}$, and such that 
$$1\leq \frac{\pi_0(A_i)}{\pi_1(B_i)}\leq K.$$
We denote $A$ the element corresponding to $B$. Then
$$\frac{\pi_1(B_i)}{\hat{\pi}_1(B')}=K\cdot\frac{\pi_1(B_i)}{\pi_1(B)}\geq\frac{\pi_0(A_i)}{\pi_0(A)}\geq\frac{1}{K}.$$
Finally, we obtain $\hat{\pi}_1(B')$ which verifies
$$\frac{1}{K}\leq\frac{\pi_1(B)}{\hat{\pi}_1(B')}\leq K,$$ for all $B\sim B'$, and such that there exists $B\sim B'$ with $\hat{\pi}_1(B')=\frac{\pi_1(B)}{K}$. That completes the proof of the lemma.
\end{proof}

\begin{lemma}\label{pasageconv}
Let $G=\left(V,E\right)$ be a graph with valence bounded by a constant $K$ and let $p>0$. Let $\Gamma$ be a family of edge-paths of $G$ and let $p>0$. Suppose that $\tau:V\to\R_+$ is a function verifying
$$\sum_{i=1}^{N-1}\tau\left(z_i\right)\geq 1,\ \text{for all path $\gamma=\left\{\left(z_i,z_{i+1}\right)\right\}_{i=1}^{N-1}\in \Gamma$}.$$ 
Then there exists $\tilde{\tau}:V\to\R_+$ such that
\begin{equation}\label{conv}
\sum\limits_{i=1}^{N-1}\tilde{\tau}^*\left(z_i\right)\wedge \tilde{\tau}^*\left(z_{i+1}\right)\geq 1\ \text{for all path $\gamma=\left\{\left(z_i,z_{i+1}\right)\right\}_{i=1}^{N-1}\in\Gamma$},
\end{equation}
where $\tilde{\tau}^*(x)=\min\left\{\tilde{\tau}(y):y\sim x\right\}$, and such that
\begin{equation}
\sum\limits_{z\in V}\tilde{\tau}\left(z\right)^p\leq 2^pK^2\cdot \sum\limits_{z\in V}\tau(z)^p.
\end{equation}
\end{lemma}
\begin{proof}
For $x\in V$, let $V_2(x)=\left\{y\in V:\exists\ z\in V\text{ s.t. }y\sim z\sim x\right\}$ be the ``combinatorial'' ball of radius $2$ in the graph $G$. We define $\hat{\tau}:V\to\R_+$ by setting $$\hat{\tau}(x)=\max\left\{\tau(y):y\in V_2(x)\right\}.$$ If $\gamma=\left\{\left(z_i,z_{i+1}\right)\right\}_{i=1}^{N-1}$ is a path of  $\Gamma$,  for $i\in\left\{1,\ldots,N-1\right\}$, we write $A_i$ the vertices of $G$ which are neighbors of $z_i$ or neighbors of $z_{i+1}$. Then
\begin{align*}
\sum\limits_{i=1}^{N-1}\hat{\tau}^*\left(z_i\right)\wedge \hat{\tau}^*\left(z_{i+1}\right)=\sum\limits_{i=1}^{N-1}\min\left\{\hat{\tau}(z):z\in A_i\right\}.
\end{align*}
If $y\in V_2(x)$, then $\hat{\tau}(y)\geq \tau(x)$. But this implies that 
\begin{equation}
\min\left\{\hat{\tau}(z):z\in A_i\right\}\geq\min\left\{\hat{\tau}(z):z\in V_2\left(z_i\right)\right\}\geq\max\left\{\tau\left(z_i\right),\tau\left(z_{i+1}\right)\right\},
\end{equation}
since $A_i$ is contained in $V_2\left(z_{i}\right)$ and in $V_2\left(z_{i+1}\right)$. Therefore,
\begin{equation}
\sum\limits_{i=1}^{N-1}\hat{\tau}^*\left(z_i\right)\wedge \hat{\tau}^*\left(z_{i+1}\right)\geq \sum\limits_{i=1}^{N-1}\frac{\tau\left(z_i\right)+\tau\left(z_{i+1}\right)}{2}\geq \frac{1}{2}.
\end{equation}
On the other hand, the cardinal number of $V_2(x)$ is bounded from above by $K^2$ for all $x\in V$. So
\begin{align*}
\sum\limits_{x\in V}\hat{\tau}(x)^p\leq\sum\limits_{x\in S}\sum\limits_{z\in V_2(x)}\tau(z)^p\leq K^2\sum\limits_{x\in V}\tau(x)^p.
\end{align*}
It is enough to take $\tilde{\tau}=2\cdot\hat{\tau}$. This finish the proof of the lemma.
\end{proof}

\begin{proof}[Proof of Theorem \ref{simplification}]
Take $\eta_0\in(0,1)$ which will be fixed later, and define
$$\eta_-=\left(\eta_0\cdot M_1^{-1}\right)^{1/p}\in (0,1),$$ 
where $M_1$ is a constant, depending only on $a$, $\lambda$, $\kappa$ and the constant $K_D$, which bounds from above the cardinal number of $T_k(B)$ for all $k\geq 0$ and $B\in\mathcal{S}_k$. We define the function $\tau=\left(\sigma^p+\eta_-^p\right)^{1/p}\geq\eta_-$, which also verifies item (S1). From inequality (\ref{sumachica}), for all $B\in\mathcal{S}_k$ and $k\geq 0$, we have
\begin{equation}
\sum_{B'\in T_k(B)}\tau(B')^p\leq\sum_{B'\in T_k(B)}\sigma(B')^p+\eta_-^pM_1\leq 2\cdot\eta_0.
\end{equation}
For $B\in\mathcal{S}_k$, let $V_{2,k}(B)=\left\{B'\in\mathcal{S}_k:\exists\ B''\in\mathcal{S}_k\text{ s.t. }B\sim B''\sim B'\right\}$ be the ``combinatorial'' ball of radius $2$ in the graph $\mathcal{S}_k$. For $k\geq 0$ and $B\in\mathcal{S}_k$, we define 
$$\tilde{\tau}(B)=2\cdot \max\left\{\tau(B'):B'\in V_{2,k}(B)\right\}.$$
By Lemma \ref{pasageconv}, we obtain a function $\tilde{\tau}$ satisfying condition (H3'), bounded from below by $\eta_-$ and such that
\begin{equation}
\sum_{B'\in T_k(B)}\tilde{\tau}(B')^p\leq 2^{p+1}\cdot M_2^2\cdot \eta_0,
\end{equation}
for all $B\in\mathcal{S}_k$ and $k\geq 0$. Here, the constant $M_2$, that only depends on $\lambda$, $\kappa$ and the doubling constant $K_D$, bounds from above the cardinal number of horizontal $2$-neighbors of any vertex $B\in\mathcal{S}_k$; i.e. elements in $\mathcal{S}_k$ and at combinatorial distance at most $2$ from $B$. 
\newline
Let $K=\eta_-^{-1}$, we construct a function $\hat{\rho}:\mathcal{S}\to\R_+$ verifying 
\begin{enumerate}
\item $\hat{\rho}\geq \tilde{\tau}$, and
\item (H2) with the constant $K$.
\end{enumerate}
Moreover, we will see that by construction, $\hat{\rho}$ also verifies
\begin{enumerate}
\item[(3)] $\hat{\rho}(B)\leq \max\left\{\tilde{\tau}(B'):B'\sim B\right\}$.
\end{enumerate}
We will construct $\hat{\rho}$ by defining it inductively on each $\mathcal{S}_k$. We set $\hat{\rho}(w)=1$, and since $\eta_-\leq\tilde{\tau}\leq1$ we can set $\hat{\rho}_1=\tilde{\tau}|_{\mathcal{S}_1}$.  Suppose constructed $\hat{\rho}_i:\mathcal{S}_i\to\R_+$ verifying items 1 and 2 for $i=1,2,\ldots,j$, let's construct $\hat{\rho}_{j+1}:\mathcal{S}_{j+1}\to\R_+$ using Lemma \ref{modificasion}. With the same notation as in the lemma, we denote for $A\in\mathcal{S}_j$,
$$\pi_0(A)=\prod_{i=1}^{j}\hat{\rho}_i\left(g(A)_i\right),$$
and for $B\in\mathcal{S}_{j+1}$,
$$\pi_1(B)=\tilde{\tau}(B)\pi_0\left(g(B)_j\right).$$
Since $K=\eta_-^{-1}$, and for all $B\in\mathcal{S}_{j+1}$, we have $d\left(x_B,x_{g(B)_j}\right)\leq \kappa a^{-j}$, we see that the hypothesis of Lemma \ref{modificasion} are verified. Let $\hat{\pi}_1:\mathcal{S}_{j+1}\to\R_+$ be the application given by the lemma. 

Let $B\in\mathcal{S}_{j+1}$, from the item (2) of Lemma \ref{modificasion}, we have two possibilities for $\hat{\pi}_1(B)$: it is equal to $\pi_1(B)$, or there exists some $B'\in\mathcal{S}_{j+1}$ such that $B'\sim B$ and $\hat{\pi}_1(B)=\frac{\pi_1(B')}{K}$. Equivalently, we can write $\hat{\pi}_1(B)=\hat{\rho}_{j+1}(B)\pi_0(g(B)_j)$, where $\hat{\rho}_{j+1}(B)$ is equal to $\tilde{\tau}(B)$, or it is equal to $\alpha\tilde{\tau}(B)$, with
$$\alpha=\frac{1}{K}\frac{\tilde{\tau}(B')\pi_0(g(B')_j)}{\tilde{\tau}(B)\pi_0(g(B)_j)}>1.$$ 
We remark that 
$$\alpha=\frac{\tilde{\tau}(B')\pi_0(g(B')_j)}{K\tilde{\tau}(B)\pi_0(g(B)_j)}\leq\frac{\tilde{\tau}(B')}{\tilde{\tau}(B)}.$$
Thus, we obtain
$$\hat{\rho}_{j+1}(B)=\alpha\tilde{\tau}(B)\leq \tilde{\tau}(B').$$
In any case, the function $\hat{\rho}_{j+1}$ verifies
$$\tilde{\tau}(B)\leq \hat{\rho}_{j+1}(B)\leq \max\limits_{B'\sim B}\tilde{\tau}(B'),$$
for all $B\in\mathcal{S}_{j+1}$. This shows the existence of the function $\hat{\rho}_{j+1}:\mathcal{S}\to\R_+$ which verifies items (1), (2) and (3) above. Finally, define $\hat{\rho}:\mathcal{S}\to(0,+\infty)$ by setting $\hat{\rho}|_{\mathcal{S}_k}:=\hat{\rho}_k$.

We now estimate, using item (3) above, the sum of $\hat{\rho}^p$ over $T_k(B)$. Since for all $k\geq1$ and all $B\in\mathcal{S}_k$, the cardinal number of the set $\{C:C\sim B\}$ is bounded from above by the constant $M_2$, we obtain
\begin{equation}\label{sumarhohat}
\sum\limits_{B'\in T_k(B)}\hat{\rho}(B')^p\leq\sum\limits_{B'\in T_k(B)}\sum\limits_{B''\sim B'}\tilde{\tau}(B'')^p\leq M_2\sum\limits_{C\sim B}\sum\limits_{B'\in T_k(C)}\tilde{\tau}(B')^p\leq 2^{p+1}M_2^4\cdot\eta_0=M_3\eta_0.
\end{equation}

We fix $\eta_0=\left(2M_3\right)^{-1}$, which only depends on $\lambda$, $\kappa$ and the doubling constant $K_D$. Therefore, the sum (\ref{sumarhohat}) is smaller than $1/2$.

We still have to modify $\hat{\rho}$ taking into account (H4). Note that for each level $k\geq 1$, it makes sense to ask about conditions (H1), (H2), and (H3'), since they are concerned with properties of the function $\rho$ up to this level. To start, we can simply normalize $\hat{\rho}_1$ so that the sum is equal to $1$. Since we divide by a quantity smaller than $1$, and the same for all $B\in\mathcal{S}_1$, we preserve also the conditions (H1), (H2) and (H3').

Let now $k>1$. We should remark that if $B\in\mathcal{S}_{k-1}$, then by item (vii) of Lemma \ref{propgraphe}, we know that all neighbors of an element $B'$ in $\mathcal{S}_{k}$ are descendants of $B$ if $x_B$ belongs to the ball $B\left(x_{B'},\kappa r_{k}\right)$. For each $B\in \mathcal{S}_{k-1}$, we chose one descendant $C_B\in \mathcal{S}_k$ with the above property. We denote $T_{k-1}^*(B)=T_{k-1}(B)\backslash \{C_B\}$, and we call $C_B$ the center of $T_{k-1}(B)$. For $B\in \mathcal{S}_{k-1}$, let $\omega_B\in[1,+\infty)$ be such that
$$\left(\omega_B\hat{\rho}_{k}(C_B)\right)^p+\sum\limits_{B'\in T_{k-1}^*(B)}\hat{\rho}_{k}(B')^p=1.$$ 
The fact that the sum (\ref{sumarhohat}) is strictly smaller than $1$ justifies the existence of the number $\omega(B)$. We define $\rho_{k}:\mathcal{S}_{k}\to \R_+$ by setting
$$\rho_{k}(B')=\begin{cases}\omega_{B}\hat{\rho}_{k}(C_B) & \text{ if } B'=C_B\text{ for some }B\in\mathcal{S}_{k-1}.\\
\hat{\rho}_{k}(B') & \text{ otherwise.}\end{cases}$$
Since $\omega_{B}\geq 1$, conditions (H1) and (H3') are verified. For (H1), it's enough to take $\eta_+=1-\eta_-$, because $\#T_k(B)\geq 2$. By the choice of $\omega_{B}$, we also have condition (H4) with constant $K_2=1$. 

Let us show that the condition (H2) is also verified. Recall that for $B\in \mathcal{S}_{k-1}$, all neighbors of $C_B$ in $\mathcal{S}_{k}$ belong to $T_{k-1}(B)$. Let $A,A'\in\mathcal{S}_k$ be such that $A\sim A'$, and let $0\leq n\leq k-1$ be the biggest positive integer such that $g(A)_n=g(A')_n$. Since $A\sim A'$, we have $g(A)_i\sim g(A')_i$ for all $i\in\{n+1,\ldots,k\}$. Therefore, for all $i\in\{n+2,\ldots,k\}$, neither $g(A)_i$ nor $g(A')_i$ can be centers. Otherwise, since they are neighbors, 
they would have the same parent, which is in contradiction with the definition of $n$. This implies that $\rho_i\left(g(A)_i\right)=\hat{\rho}_i\left(g(A)_i\right)$ and $\rho_i\left(g(A')_i\right)=\hat{\rho}_i\left(g(A')_i\right)$ for all $i\in\{n+2,\ldots,k\}$. For $i=n+1$, only one of them can be a center. If neither is a center we have
\begin{equation}\label{rapport}
\frac{\prod\limits_{i=n+1}^k\rho_i\left(g(A)_i\right)}{\prod\limits_{i=n+1}^k\hat{\rho}_i\left(g(A)_i\right)}=\frac{\prod\limits_{i=n+1}^k\rho_i\left(g(A')_i\right)}{\prod\limits_{i=n+1}^k\hat{\rho}_i\left(g(A')_i\right)}=1.
\end{equation}
If for example $g(A)_{n+1}$ is a center, the quotient (\ref{rapport}) is equal to $\omega_{g(A)_{n+1}}$. 
Since for all center $C_B$, we have $\eta_-\leq\omega_{B}\rho_{i}(C_B)\leq 1$, we see that in any case the quotient (\ref{rapport}) is between $1$ and $K^2$. Therefore we obtain (H2) with $K_0=K^2$. This completes the proof of the proposition.
\end{proof}

\section{Ahlfors regular conformal dimension and combinatorial modulus}\label{dimqas}
\subsection{The critical exponent\label{secexpcrit}}

Let $G=\left(V,E\right)$ be a graph and let $\Gamma$ be a family of subsets of $V$. Consider a function $\rho:V\to \R_+$ and for $\gamma\in\Gamma$, define its $\rho$-length as
\begin{equation}\label{longcomb}
\ell_\rho\left(\gamma\right)=\sum\limits_{v\in\gamma}\rho\left(v\right).
\end{equation}
For $p> 0$, we denote the $p$-volume of $\rho$ by 
\begin{equation}\label{pmasse}
\mathrm{Vol}_p(\rho)=\sum\limits_{v\in V}\rho\left(v\right)^p.
\end{equation}
Thus the $p$-combinatorial modulus of $\Gamma$ is by definition
\begin{equation}\label{eq:modcomb}
\Mod_p\left(\Gamma,G\right)=\inf\limits_{\rho}\mathrm{Vol}_p(\rho),
\end{equation}
where the infimum is taken over all functions $\rho:V\to\R_+$ which are $\Gamma$-admissible, i.e. $\ell_\rho\left(\gamma\right)\geq 1$ for all $\gamma\in\Gamma$. We remark that if $p\in (0,1)$, then $\Mod_p\left(\Gamma,G\right)\geq 1$ unless $\Gamma$ is empty.

\begin{remark}
This definition is a discretization of the classical notion of \emph{conformal modulus} from complex analysis, see \cite{Ah}. See also \cite{H1} for a detailed exposition on the combinatorial modulus.
\end{remark}

We recall that we suppose $X$ doubling of constant $K_D\geq 1$, and uniformly perfect of constant $K_P\geq 1$. In particular, the conformal gauge $\mathcal{J}_{AR}(X,d)\neq\emptyset$. We fix $\kappa>1$ and $b>1$. For each $k\geq1$, let $\mathcal{U}_k$ be a finite covering of $X$ satisfying eq.\,(\ref{almostball1}) and (\ref{almostball2}) with $b$ in the place of $a$. We write $\mathcal{U}:=\bigcup_k\mathcal{U}_k$. For each $k\geq 1$, we define the graph $G_k$ as the nerve of $\mathcal{U}_k$, i.e. the vertices of $G_k$ are the elements of $\mathcal{U}_k$ and we put an edge between $B$ and $B'$ if $\lambda \cdot B\cap \lambda\cdot B'\neq\emptyset$, where $\lambda$ is a constant (recall (\ref{constantes})). We use the same notation as in the previous sections. 

\begin{definition}[Combinatorial modulus]\label{modcomb}
Let $p>0$ and $L>1$, we define
\begin{equation}\label{module}
M_{p,k}\left(L\right):=\sup\limits_{B\in\mathcal{U}}\ \Mod_p\left(\Gamma_{k,L}(B),G_{|B|+k}\right),
\end{equation}
where, for $k\geq 1$ and $B\in\mathcal{U}$, we denote by $\Gamma_{k,L}(B)$ the family of paths $\gamma=\left\{B_i\right\}_{i=1}^N$ of $G_{|B|+k}$ such that $z_1$, the center of $B_1$, belongs to $B$ and $z_N$, the center of $B_N$, belongs to $X\setminus L\cdot B$. See Figure \ref{figuremodcomb}.
\end{definition}
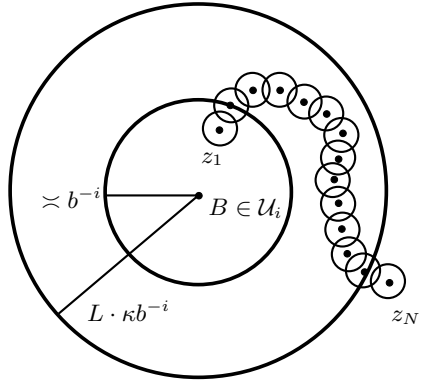
\begin{figure}
\centering
\setlength{\unitlength}{1cm}
\begin{picture}(6,6)
\includegraphics{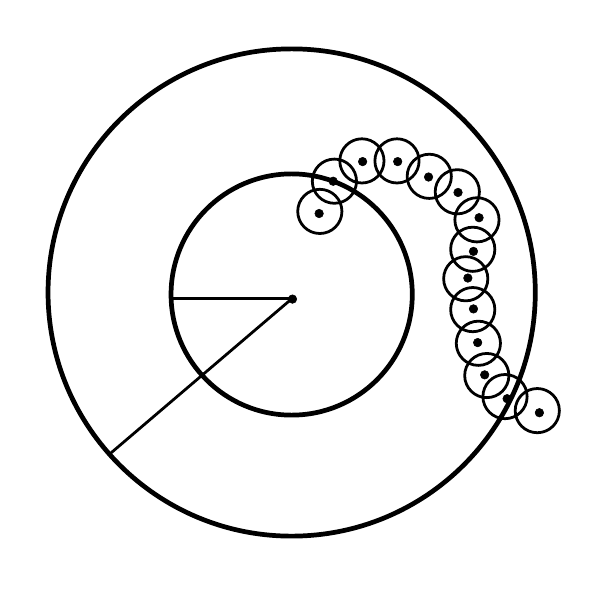}
\put(-0.5,1.3){\footnotesize $z_N$}
\put(-3,3.4){\footnotesize $z_1$}
\put(-2.9,2.7){\footnotesize $B\in \mathcal{U}_i$}
\put(-4.5,1.3){\footnotesize $L\cdot \kappa b^{-i}$}
\put(-5.1,2.8){\footnotesize $\asymp b^{-i}$}
\end{picture}
\caption{\label{figuremodcomb} Definition of the combinatorial modulus $M_{p,k}(L)$ (Definition \ref{module}). In the figure, $B$ is an element of $\mathcal{U}$ and $z_1,z_N$ are the centers of the extremities of a path $\gamma$ in $\Gamma_{k,L}(B)$. The scale of the covering is $|B|+k$ so $k$ represents the scale relative to that of $B$. The number $L$ represents the relative diameter of the paths $\gamma$. The modulus $M_{p,k}(L)$ takes into account all moduli of the ``annuli'' associated to the elements of $\mathcal{U}$.}
\end{figure}

In this first part, $L$ is considered as a fixed parameter. We remark that $M_{p,k}\left(L\right)<+\infty$ for all $k\geq 1$, since the number of elements in $\mathcal{U}_{|B|+k}$ that intersect $L\cdot B$ is bounded above by some constant, which only depends on $k$, and therefore not on $B\in \mathcal{U}$. 

We study the asymptotic behavior of the sequence $\left\{M_{p,k}\left(L\right)\right\}_k$ when $k$ tends to infinity, and its dependence on $p$. We define 
\begin{equation}\label{modulelimite}
M_p\left(L\right)=\liminf\limits_{k\to+\infty}M_{p,k}\left(L\right).
\end{equation}
For fixed $k$, the function $p\mapsto M_{p,k}\left(L\right)$ is non-increasing, since an optimal function for the combinatorial modulus, which always exists, is less than or equal to $1$. This important fact implies that the set of $p\in(0,\infty)$ such that $M_p\left(L\right)=0$ is an interval.

\begin{definition}[The critical exponent]\label{defexpcrit}
We define the \emph{critical exponent} of the combinatorial modulus by setting
\begin{equation}\label{expcritcomb}
Q_N\left(L\right)=\inf\left\{p\in(0,+\infty):M_p\left(L\right)=0\right\}.
\end{equation}
\end{definition}
\begin{remark}[Remark 1.]
Later we will consider another critical exponent closely related to the topology of $X$, so it is important to note that $Q_N$ is defined in purely combinatorial terms, i.e. we only use the combinatorial modulus on the nerves $G_k$ of the sequence of coverings $\mathcal{U}_k$.
\end{remark}
\begin{remark}[Remark 2.]
If $p\in (0,1)$, then $M_{p,k}(L)\geq 1$ unless $\Gamma_{k,L}(B)$ is empty for all $B\in\mathcal{U}$. Conversely, if the curve families $\Gamma_{k,L}(B)$ are empty, for all $k$ sufficiently large we also have $M_p(L)=0$ for all $p>0$. Therefore, $Q_N(L)\notin (0,1)$, and $Q_N(L)=0$ if and only if $X$ is uniformly disconnected (for a definition, see Chapter 15 of  \cite{DS}).
\end{remark}

\subsection{Proof of Theorem \ref{maintheo}}\label{proofmain}

We can prove now the first inequality between $Q_N\left(L\right)$ and the Ahlfors regular conformal dimension of $X$ (compare with \cite{HP2} Corollaire 3.3). We will prove that if $p>\dim_{AR}X$, then $M_p\left(L\right)=0$. Therefore, $Q_N\left(L\right)\leq\dim_{AR}X$. In particular, since $X$ is doubling and uniformly perfect, we have $Q_N\left(L\right)<+\infty$.
\begin{proof}[Proof of $Q_N\left(L\right)\leq\dim_{AR}X$]\label{expcritdimconfar}
Suppose that $\dim_{AR}X<q<p$, and let $\theta $ be an Ahlfors $q$-regular distance in the gauge of $X$. We denote $\mu$ the $q$-dimensional Hausdorff measure of $\left(X,\theta \right)$ and $\eta:\R_+\to\R_+$ the distortion function of $id:\left(X,d\right)\to\left(X,\theta \right)$. Fix some element $B\in\mathcal{U}$ and let $k\geq 1$, we set $i=|B|$. Define $\rho:\mathcal{U}_{|B|+k}\to\R_+$ by setting
\begin{equation}\label{defderho}
\rho\left(B'\right)=\begin{cases}\left(\frac{\mu\left(B'\right)}{\mu\left(\left(L+1\right)\cdot B\right)}\right)^{1/q}&\text{if }B'\cap \overline{L\cdot B}\neq\emptyset.\\
0&\text{otherwise.}
\end{cases}
\end{equation}
Then
\begin{align*}
\sum\limits_{B'\in\mathcal{U}_{|B|+k}}\rho\left(B'\right)^p\leq\max\limits_{B'\cap \overline{L\cdot B}\neq\emptyset}\rho\left(B'\right)^{p-q}\cdot\sum\limits_{B'\cap \overline{L\cdot B}\neq\emptyset}\rho\left(B'\right)^{q}.
\end{align*}
We write $\diam_\theta$ for the diameter, and $B_\theta(s)$ for a ball of radius $s$, both in the distance $\theta$. Since $X$ is uniformly perfect,
$$\diam\left(\left(L+1\right)\cdot B\right) \geq \left(L+1\right)K_P^{-1}\kappa b^{-i}.$$ From the diameter distortion formula for quasisymmetric maps (see eq.\,(\ref{distordiam})), for all element $B'$ in $\mathcal{U}_{i+k}$ such that $B'\cap \overline{L\cdot B}\neq\emptyset$, we have
\begin{align*}
\rho\left(B'\right)=&\left(\frac{\mu\left(B'\right)}{\mu\left(\left(L+1\right)\cdot B\right)}\right)^{1/q}\asymp\frac{\diam_\theta B'}{\diam_\theta \left(L+1\right)\cdot B}\\
\lesssim&\ \eta\left(2\cdot\frac{\diam B'}{\diam \left(L+1\right)\cdot B}\right)\leq\eta\left(\frac{4K_P}{\left(L+1\right)\cdot b^{k}}\right):=\eta_k.
\end{align*}
There exists a constant $K\geq 1$, which depends only on $\eta$ and $\kappa$, such that for any $B'\in \mathcal{U}_{i+k}$, there is a ball $B_\theta(s)$ for the distance $\theta$ such that 
$$B_\theta(s)\subset B\left(x_{B'},\frac{1}{\kappa} b^{-(i+k)}\right)\subset B'\subset B_\theta(Ks).$$ 
Since the balls $\left\{B\left(x_{B'},\kappa^{-1} b^{-(i+k)}\right):B'\in\mathcal{U}_{i+k}\right\}$ are pairwise disjoint, 
the same holds for the balls $B_\theta(s)$. Also, since the union of the elements $B'$ such that $B'\cap \overline{L\cdot B}\neq\emptyset$, is contained in $\left(L+1\right)\cdot B$, we obtain
\begin{align*}
\sum\limits_{B'\cap \overline{L\cdot B}\neq\emptyset}\rho\left(B'\right)^{q}=&\frac{1}{\mu\left(\left(L+1\right)\cdot B\right)}\cdot\sum\limits_{B'\cap \overline{L\cdot B}\neq\emptyset}\mu\left(B'\right)\leq\frac{1}{\mu\left(\left(L+1\right)\cdot B\right)}\cdot\sum\limits_{B'\cap \overline{L\cdot B}\neq\emptyset}\mu\left(B_\theta(Ks)\right)\\
\lesssim&\frac{1}{\mu\left(\left(L+1\right)\cdot B\right)}\cdot\sum\limits_{B'\cap \overline{L\cdot B}\neq\emptyset}\mu\left(B_\theta(s)\right)\leq 1.
\end{align*}
We now look at the admissibility condition. Let $\gamma=\left\{B_j\right\}_{j=1}^N\in\Gamma_{k,L}(B)$, we can suppose that $B_j\cap \overline{L\cdot B}\neq\emptyset$ for all $j$. We denote the center of $B_j$ by $z_j$. Since for each $j\in \left\{1,\ldots,N-1\right\}$ we have $\theta \left(z_j,z_{j+1}\right)\leq \diam_\theta \left(\lambda\cdot B_j\right)+\diam_\theta \left(\lambda\cdot B_{j+1}\right)$, we obtain
\begin{align*}
\sum\limits_{j=1}^N\rho\left(B_j\right)&=\sum\limits_{j=1}^N\left(\frac{\mu\left(B_j\right)}{\mu\left(\left(L+1\right)\cdot B\right)}\right)^{1/q}\\
&\asymp\sum\limits_{j=1}^N\frac{\diam_\theta  B_j}{\diam_\theta  \left(L+1\right)\cdot B}\gtrsim\frac{\theta \left(z_1,z_N\right)}{2\cdot\diam_\theta  \left(L+1\right)\cdot B}\geq c,
\end{align*}
where $c>0$ is a constant that depends only on $\eta$, $\lambda$, $\kappa$, $K_P$ and $L$.
Therefore, we finally obtain
$$M_{p,k}\lesssim \eta_k^{\ p-q},$$
which tends to zero when $k$ tends to infinity. This completes the proof of the inequality.
\end{proof}

If $L'\geq L\geq 1$, then $M_{p,k}\left(L'\right)\leq M_{p,k}\left(L\right)$ for all $k\geq 1$; therefore, $Q_N\left(L'\right)\leq Q_N\left(L\right)$. We start by showing in the following lemma that, in fact, $Q_N(L)$ does not depend on $L>1$.

\begin{lemma}[Independence on $L$]\label{indepdel}
Let $1<L\leq L'$ and $p>0$. There exists an integer $l\geq 0$ and a constant $K_7\geq 1$, which depend only on $L$, $L'$ and $\kappa$, such that for all $k\geq 1$, we have
$$M_{p,l+k}\left(L\right)\leq K_7\cdot M_{p,k}\left(L'\right).$$
In particular, $Q_N(L)=Q_N(L')$ for all $L$ and $L'$. 
\end{lemma}
\begin{proof}
Let $1<L\leq L'$ and $B\in \mathcal{U}_i$ for some $i\geq 1$. We take $l\geq 0$ the smallest integer such that $b^{-l}<\frac{L-1}{2L'}$, and let 
$$\mathcal{U}_{i+l}(B):=\left\{A\in \mathcal{U}_{i+l}: A\cap B\neq \emptyset\right\}.$$ 
Let $\gamma=\left\{B_j\right\}_{j=1}^N$ be a path of $\Gamma_{l+k,L}\left(B\right)$, and denote $z_j$ the center of $B_j$. If $A$ is an element of $\mathcal{U}_{i+l}(B)$ such that $z_1$ belongs to $A$, then $\gamma$ is a path in $\Gamma_{l+k,L'}(A)$. In fact, by the choice of $l$ the point $z_N$ does not belongs to $L'\cdot A$, since $d\left(z_1,z_N\right)\geq (L-1)\kappa b^{-i}$.

For each $A\in \mathcal{U}_{i+l}(B)$, let $\rho_A:\mathcal{U}_{i+l+k}\to\R_+$ be an optimal function for $\Gamma_{l+k,L'}(A)$. We define $\rho:\mathcal{U}_{i+l+k}\to \R_+$ by setting
$$\rho(B')=\max\left\{\rho_A(B'):A\in\mathcal{U}_{i+l}(B)\right\}.$$
Therefore, $\rho$ is $\Gamma_{l+k,L}(B)$-admissible. Remark that there exists a constant $K_7$, which depends only on $l$, $\kappa$ and the doubling constant of $X$, that bounds from above the number of elements in $\mathcal{U}_{i+l}(B)$. So we obtain 
$$\mathrm{Vol}_p\left(\rho\right)\leq \sum\limits_{A\in\mathcal{U}_{i+l}(B)}\Mod_p\left(\Gamma_{l+k,L'}(A),G_{i+l+k}\right)\leq K_7\cdot M_{p,k}\left(L'\right).$$
Therefore, $M_{p,l+k}\left(L\right)\leq K_7\cdot M_{p,k}\left(L'\right)$.
\end{proof}
We fix $L=2$, and we consider $M_{p,k}:=M_{p,k}\left(2\right)$, $M_p:=M_p\left(2\right)$ and $Q_N:=Q_N(2)$. We can prove now the main result of this Section.

\begin{proof}[Proof of the inequality $\dim_{AR}X\leq Q_N$]
Let $p>0$ such that $M_p=0$, applying Theorem \ref{controldim} we will show that $\dim_{AR}X\leq p$. Let $n_0\geq 1$ be large enough so that $a:=b^{n_0}$ verifies (\ref{constantes}), and that $M_{p,n_0}\leq \eta$, where $\eta\in (0,1)$ is a number that will be fixed later.

We take $\mathcal{S}_k=\mathcal{U}_{k\cdot n_0}$. For simplicity, we write $G_k$ in the place of $G_{k\cdot n_0}$, and $\Gamma(B)$ for the family of paths in $G_{k+1}$ which ``join'' $B$ and $X\setminus 2\cdot B$. Therefore, we have $\Mod_p\left(\Gamma(B),G_{k+1}\right)\leq \eta$ for all $B\in\mathcal{S}$. We fix the genealogy $\mathcal{V}$ as that of eq.\,(\ref{gencanonique}), i.e. we set 
$$V_k(B)=\left\{y\in X: d(y,x_B)=\mathrm{dist}(y,X_k)\right\}.$$
Using the fact that the combinatorial modulus is small, we construct a function $\rho:\mathcal{S}\to (0,1)$ which verifies the hypotheses (S1) and (S2) of Proposition \ref{simplification}. In fact, for all $B\in\mathcal{S}$, there exists $\sigma_B:\mathcal{S}_{k+1}\to\R_+$ such that: 
\begin{enumerate}
\item if we denote by $V_B=\left\{B'\in\mathcal{S}_{k+1}:B'\cap 3\cdot B\neq\emptyset\right\}$, then $\sigma_B(B')=0$ if $B'\notin V_B$.
\item for any path $\gamma=\left\{B_i\right\}_{i=1}^N$ of level $k+1$ such that $z_1\in B$ and $z_N\in X\setminus 2\cdot B$ ---we write as usual $z_i$ the center of $B_i$--- we have
$$\sum\limits_{i=1}^N\sigma_B\left(B_i\right)\geq 1,$$
\item and $\sum\limits_{B'\in\mathcal{S}_{k+1}}\sigma_B(B')^p\leq \eta$.
\end{enumerate}
To define $\rho$, we start by setting $\sigma_{k+1}:\mathcal{S}_{k+1}\to \R_+$ to be
$$\sigma_{k+1}(B')=\max\left\{\sigma_A(B'):A\in\mathcal{S}_k\right\}.$$
Since $\sigma_{k+1}\geq \sigma_B$, item 2 above still holds if we replace $\sigma_B$ by $\sigma_{k+1}$. Using the item (1) and the fact that $T_k(B)\subset V_B$ for all $B\in\mathcal{S}$, we obtain
\begin{align*}
\sum\limits_{B'\in T_k(B)}\sigma_{k+1}(B')^p=&\sum\limits_{B'\in T_k(B)}\max\left\{\sigma_{A}\left(B'\right)^p:A\in\mathcal{S}_k\right\}\leq 
\sum\limits_{B'\in T_k(B)}\sum\limits_{A:B'\in V_A}\sigma_{A}\left(B'\right)^p\\
\leq&\sum\limits_{A:V_B\cap V_A\neq\emptyset}\sum\limits_{B'\in V_A}\sigma_{A}\left(B'\right)^p\leq K_8\cdot\eta,
\end{align*}
where $K_8$ is a constant, which depends only on $\kappa$ and the doubling constant of $X$, such that $\left|\left\{A\in\mathcal{S}_k:V_A\cap V_B\neq \emptyset\right\}\right|\leq K_8$ for all $k\geq 1$ and all $B\in\mathcal{S}_k$.
Therefore, to apply Proposition \ref{simplification} it is enough to choose $\eta\leq K_8^{-1}\eta_0$. This ends the proof of Theorem \ref{maintheo}.
\end{proof}
\begin{remark}[Remark 1.]
One consequence of the proof of Theorem \ref{maintheo} is the following: if $p>Q_N$, i.e. $M_p=\liminf_k M_{p,k}=0$, we have shown that $\dim_{AR}X\leq p$. Therefore, from the proof of the first inequality, we have $\lim_k M_{q,k}=0$ for all $q>p$. In other words, we can replace the lower limit by the limit in the definition of $Q_N$.
\end{remark}
\begin{remark}[Remark 2.]
From the remark that follows Definition \ref{defexpcrit}, we have that $\dim_{AR}X\notin (0,1)$, and is equal to zero if and only if $X$ is uniformly disconnected. In that case, the AR conformal dimension is not attained. Compare with \cite{Ko}.
\end{remark}

\subsection{Comparison with the moduli on the tangent spaces\label{sectangents}}

The purpose of this section is to show that when $X$ is $p$-regular, $M_p$ is bounded from above by the $p$-analytic modulus of curve families in the weak tangent spaces of $X$.

We start by recalling some definitions, for a detailed exposition we refer to \cite{MT}. A sequence of nonempty closed subsets $\{F_n\}_n$ of a metric space $(Z,d)$ converges in the sense of Hausdorff to a closed set $F\subset Z$ if
\begin{equation}\lim\limits_{n\to\infty}\sup\limits_{z\in F_n\cap B(x,R)} \mathrm{dist}(z,F) = 0 \text{ and }
\lim\limits_{n\to\infty}\sup\limits_{z\in F\cap B(x,R)}\mathrm{dist}(z,F_n)=0,
\end{equation}
for all $x\in Z$ and $R>0$. 
\begin{definition}[Convergence of metric measure spaces]
A pointed sequence of complete metric measure spaces $\{(Z_n, d_n, \mu_n, p_n)\}$ converges to a pointed complete metric measure space $(Z,d,\mu,p)$, if there exists a pointed metric space $(\mathcal{Z},D,q)$, and isometric embeddings $\iota_n:Z_n\to\mathcal{Z}$ and $\iota:Z\to\mathcal{Z}$, with $\iota_n(p_n)=\iota(p)=q$ for all $n\geq 0$, such that $\{\iota_n(Z_n)\}$ converges in the sense of Hausdorff to $\iota(Z)$, and the sequence of measures $\{(\iota_n)_*\mu_n\}$ weakly converges to $\iota_*\mu$. If we ignore measures, we obtain the Gromov-Hausdorff convergence of metric spaces.
\end{definition}

If $X$ is a doubling space, for any sequence $\{r_n\}$ of scales and any sequence of points $\{x_n\}$ of $X$, the family $\left\{\left(X,x_n,r_n^{-1}d\right)\right\}$ is relatively compact in the Gromov-Hausdorff topology. The limit points are called weak tangent spaces of $X$, and tangent spaces when $\left\{x_n\right\}$ is constant and $\left\{r_n\right\}$ tends to zero. 

If $\left(X,d,\mu\right)$ is Ahlfors regular of dimension $p>0$ and $\left(X_\infty,d_\infty, x_\infty\right)$ is a weak tangent space of  $X$, with sequence of scales $\left\{r_n\right\}$, then $X_\infty$ is also regular of dimension $p$, where the $p$-dimensional Hausdorff measure is comparable to a weak limit of $\{r_n^{-p}\mu\}$, which we denote $\mu_\infty$. We remark that if $p\in (0,1)$, then $X$ is uniformly disconnected (see also Theorem 5.1.9 of \cite{MT}). We also recall that if $\Gamma$ is a curve family of $X$, then the $p$-analytic modulus of $\Gamma$ is by definition 
$$\Mod_p\left(\Gamma\right)= \inf\limits_{\rho}\int_X\rho^{p}\ d\mu,$$
where the infimum is taken over all Borel measurable functions $\rho:X\to\R_+\cup\{+\infty\}$ which are $\Gamma$-admissible (see \cite{He}). The analytic moduli in the weak tangent spaces of $X$ are always defined using this measure $\mu_\infty$.

From now on we suppose that $X$ is $p$-regular. Let $(X_\infty,d_\infty,x_\infty)$ be a weak tangent space of $X$. We consider the family $\Gamma(x_\infty)$ of curves which join $B(x_\infty,1)$ and $X\setminus B(x_\infty,2)$. 
\begin{definition}[Moduli on the tangent spaces]
We define 
$$M_p^T:=\sup\left\{\Mod_p\left(\Gamma(x_\infty)\right):\left(X_\infty,d_\infty,x_\infty\right)\text{ is a weak tangent space of }X\right\}.$$
\end{definition}

The following proposition shows that the combinatorial modulus is dominated by the analytical moduli on the weak tangent spaces of $X$.
\begin{proposition}\label{semicont}
There exists a constant $K_9$, which depends only on $\kappa$ and the doubling constant of $X$, such that $M_p\leq K_9\cdot M_p^T$. 
\end{proposition}
\begin{proof}
If $p\in (0,1)$, the inequality trivially holds, because $X$ is uniformly disconnected, and therefore, both $M_p$ and $M_p^T$ are null. So we can suppose that $p\geq 1$. We does not provide full details of the proof, because it consists of small modifications of arguments that appear in Section 3 of \cite{KL} and Appendix B of \cite{H1}.

Let $C\geq 1$, $\epsilon >0$ and suppose that $M_p> K\cdot M_p^T$. This means that there exists $k_0\geq 1$ such that for all $k\geq k_0$, there exists $i=i_k\geq 1 $ and $B=B_k\in\mathcal{U}_{i}$ such that 
$$\Mod_p\left(\Gamma_{k}\left(B\right),G_{i+k}\right)\geq K\cdot M_p^T+\delta, \text{ where }\delta>0.$$
Let $r_k=b^{-i_k}$, and consider $\left(X_\infty,d_\infty,x_\infty\right)$ a limit point of the sequence $\left\{\left(X,r_k^{-1}d,x_k\right)\right\}_k$. We fix a compact metric space $Z$, and isometric embeddings $\iota: \overline{B}\left(x_\infty,3\right)\to Z$ and $\iota_k:\left(\overline{B}\left(x_k,3\right),r_k^{-1}d\right)\to Z$ for each $k$, where we denote $x_k$ the center of $B$. We identify the curve family $\Gamma\left(x_\infty\right)$ and the family of paths $\Gamma_{k}(B)$ with its images by the embedding $\iota$ and $\iota_k$. Analogously, we identify the measure $\mu_\infty$ with its image by $\iota$.

Consider $\rho:Z\to \R_+$ a continuous $\Gamma\left(x_\infty\right)$-admissible function such that $\mathrm{Vol}_p(\rho)\leq \Mod_p\left(\Gamma\left(x_\infty\right)\right)+\epsilon\leq M_p^T+\epsilon$. We can suppose that $\rho\geq m>0$. Define $\rho_k:\mathcal{U}_{i+k}\to\R_+$ by ---we also identify $\mathcal{U}_{i+k}$ with its image by $\iota_k$---
\begin{equation}
\rho_k(A)=\frac{3}{2}\inf\left\{\rho(y):y\in \lambda\cdot A\right\}\diam_Z\left(\lambda\cdot A\right).
\end{equation}
For big enough $k$, the function $\rho_k$ is $\Gamma_{k}(B)$-admissible (see \cite{H1} Proposition B.2 and \cite{KL} Proposition 3.2.4). Since the balls $\left\{B\left(x_{A},\kappa^{-1}b^{-(i+k}\right):A\in\mathcal{U}_{i+k}\right\}$ are pairwise disjoint, there exists a constant $M$, which depends only on $\kappa$ and the Ahlfors regularity constant of $Z$, such that $\mathrm{Vol}_p\left(\rho_k\right)\leq M\cdot\mathrm{Vol}_p(\rho)\leq M\cdot \left(M_p^T+\epsilon\right)$. Therefore, for all $\epsilon>0$, we have
$$K\cdot M_p^T+\delta\leq M\cdot \left(M_p^T+\epsilon\right),$$
which is impossible if $K>M$. This finish the proof.
\end{proof}

\subsection{Positiveness of moduli at the critical exponent}
In this section, we show that the sequence $\left\{M_{p,k}\right\}_k$ admits a strictly positive lower bound when $p=Q_N$. Indeed, this is a consequence of the fact that the sequence $\left\{M_{p,k}\right\}_k$ satisfies a weak sub-multiplicative inequality on $k$. The proof is an adaptation of arguments from \cite{BK} Section 3, Proposition 3.12. The difference here is that we don't suppose $X$ to be approximately self-similar.

We fixe $L\geq 2$, and we also denote $M_{p,k}=M_{p,k}\left(L\right)$. For $i,k\geq 1$ and $B\in\mathcal{U}_i$, we denote by $\Gamma_{k}'(B)$ the family of paths in $G_{i+k}$ which join $L_1\cdot B$ and $L_2\cdot B$, where 
$$L_1=1+\frac{1}{b} \text{ and } L_2=L-\frac{1}{b}.$$
Define $M'_{p,k}$ in the same way as $M_{p,k}$, replacing $\Gamma_{k}(B)$ by $\Gamma_{k}'(B)$ in the definition \ref{module}. We have the following lemma:
\begin{lemma}\label{lme:smultiplicative}
There exists a constant $K_{10}\geq 1$, which depends only on $p$, $L$, $\kappa$ and the doubling constant $K_D$, such that
\begin{equation}\label{smultiplicative}
M_{p,k+l}\leq K_{10}\cdot M'_{p,k}\cdot M_{p,l}.
\end{equation}
for all $k$ and $l$.
\end{lemma}
\begin{proof}
For each $i,k\geq 1$ and $B\in\mathcal{U}_{i}$, we denote $\rho^B_{k}:\mathcal{U}_{i+k}\to \R_+$ an optimal function, i.e. which verifies: $\rho^B_{k}$ is $\Gamma_{k}(B)$-admissible and
\begin{equation*}
\sum_{B'\in \mathcal{U}_{i+k}}\rho^B_{k}\left(B'\right)^p=\Mod_p\left(\Gamma_{k}(B),G_{i+k}\right).
\end{equation*}
Analogously, define $\sigma_{k}^B:\mathcal{U}_{i+k}\to\R_+$ an optimal function for $M'_{p,k}$. Optimality implies that $\rho^B_{k}\left(A\right)=\sigma^B_{k}\left(A\right)=0$ for any element $A$ of $\mathcal{U}_{i+k}$ which does not intersects $\overline{L\cdot B}$. For $B\in\mathcal{U}_i$, we set $\mathcal{U}_{i+k}\left(B\right)$ the elements $A\in\mathcal{U}_{i+k}$ such that $A\cap \overline{L\cdot B}\neq\emptyset$, and $\chi_{k}^B:\mathcal{U}_{i+k}\to\{0,1\}$ the characteristic function of $\mathcal{U}_{i+k}\left(B\right)$.

We fix now $i\geq 1$ and $B\in\mathcal{U}_i$. We must bound from above the $p$-combinatorial modulus of the path family $\Gamma_{k+l}(B)$ of $G_{i+k+l}$.
We define $\rho:\mathcal{U}_{i+k+l}\to\R_+$ by
\begin{equation}\label{rhofinal}
\rho\left(C\right)=\max\left\{\sigma^B_{k}\left(A\right)\cdot\rho^A_{l}\left(C\right):A\in\mathcal{U}_{i+k}\right\}\cdot\chi_{k+l}^B\left(C\right).
\end{equation}
Therefore, the $p$-volume is bounded from above by: 
\begin{align*}
\sum\limits_{C\in\mathcal{U}_{i+k+l}}\rho\left(C\right)^p=& \sum\limits_{C\in\mathcal{U}_{i+k+l}}\max\left\{\sigma^B_{k}\left(A\right)^p\cdot\rho^A_{l}\left(C\right)^p:A\in\mathcal{U}_{i+k}\right\}\chi^B_{k+l}\left(C\right)\\
\leq& \sum\limits_{C\in\mathcal{U}_{i+k+l}}\sum\limits_{A\in \mathcal{U}_{i+k}}\sigma^B_{k}\left(A\right)^p\cdot\rho^A_{l}\left(C\right)^p\chi_{k+l}^B\left(C\right)\chi_{l}^A\left(C\right)\\
=&\sum\limits_{A\in \mathcal{U}_{i+k}}\sigma^B_{k}\left(A\right)^p\cdot\left(\sum\limits_{C\in\mathcal{U}_{i+k+l}}\rho^A_{l}\left(C\right)^p\chi_{k+l}^B\left(C\right)\chi_{l}^A\left(C\right)\right)\\
 \leq& \sum\limits_{A\in \mathcal{U}_{i+k}}\sigma^B_{k}\left(A\right)^p\cdot\Mod_p\left(\Gamma_{l}(A),G_{(i+k)+l}\right)\\
\leq&\ \Mod_p\left(\Gamma'_{k}(B),G_{i+k}\right)\cdot\max\limits_{A\in \mathcal{U}_{i+k}}\Mod_p\left(\Gamma_{l}(A),G_{(i+k)+l}\right)\\
\leq&\ M'_{p,k}\cdot M_{p,l}.
\end{align*}
We look now for the admissibility condition. Let $\gamma=\left\{C_j\right\}_{j=1}^N\in\Gamma_{k+l}(B)$, we write $w_j$ for the center of $B_j$. For $A\in\mathcal{U}_{i+k}$ such that $A\cap \gamma\neq\emptyset$, the path $\gamma$ also belongs to $\Gamma_{l}(A)$, because $\diam\gamma\geq (L-1)b^{-i}$ and $\diam\left(L\cdot B\right)\leq 2L\cdot b^{-(i+k)}$. We fix $A\in\mathcal{U}_{i+k}$ such that $A\cap\gamma\neq\emptyset$, and let $j_1<j_2\in\left\{1,\ldots,N\right\}$ be such that $w_{j_1}\in A$ and $w_{j_2}\in X\setminus L\cdot A$. Using the admissibility of $\rho^A_{l}$, we obtain
\begin{equation}\label{admiss1}
\sum\limits_{j=j_1}^{j_2}\rho\left(C_{j}\right)\geq\sum\limits_{j=j_1}^{j_2}\sigma^B_{k}\left(A\right)\rho^A_{l}\left(C_{j}\right)\geq \sigma^B_{k}\left(A\right).
\end{equation}
Let $\mathcal{U}_{i+k}\left(\gamma\right)$ be the set of $A\in\mathcal{U}_{i+k}$ such that $A\cap\gamma\neq\emptyset$. Then there exists a path $\gamma'\in \Gamma'_{k}(B)$ which is contained in $\mathcal{U}_{i+k}\left(\gamma\right)$. This implies 
\begin{equation}\label{admiss2}
1\leq\sum\limits_{A\cap\gamma\neq\emptyset}\sigma^{B}_{k}\left(A\right)\leq\sum\limits_{A\cap\gamma\neq\emptyset}\sum\limits_{j_1}^{j_2}\rho\left(C_j\right)\leq K\cdot\sum\limits_{1}^N\rho\left(C_j\right),
\end{equation}
where $K$ is a constant that bounds from above the cardinal number of $\mathcal{U}_{i+k}\left(B\right)$, and which only depends on $\kappa$, $L$ and the doubling constant $K_D$. Therefore, if we multiply $\rho$ by $K$, we obtain a $\Gamma_{k+l}(B)$-admissible function with $p$-volume bounded from above by $K^p\cdot M'_{p,k}\cdot M_{p,l}$. This completes the proof of the proposition.
\end{proof}

An important consequence of this sub-multiplicative inequality is the following: 
\begin{lemma}
Let $\epsilon=\left(3K_{10}\right)^{-1}$, where $K_{10}$ is the constant of Lemma \ref{lme:smultiplicative}. Then $M'_{Q_N,k}\geq\epsilon$ for all  $k\geq 1$.
\end{lemma}
\begin{proof}
Let $1\leq k\leq n$ be any integers. Then, from Lemma \ref{lme:smultiplicative}, we have 
$$M_{p,n}\leq \left(K_{10}\cdot M'_{p,k}\right)^{m}\cdot\max\limits_{0<r<k}M'_{p,r},$$
where $m=\left[n/k\right]$. In particular, for fixed $k$, the limit of $M_{p,n}$ when $n$ tends to infinity is bounded from above by the limit of $c_k\theta_k^{[n/k]}$, where $\theta_k=K_{10}\cdot M'_{p,k}$ and $c_k$ is the maximum of $M'_{p,r}$ with $0<r<k$. Therefore, $M_p=0$ if there exists $k\geq 1$ such that $M'_{p,k}< K_{10}^{-1}$. That is, we have the following interval inclusions:
$$I':=\left\{p:M'_p=0\right\}\subset J:=\left\{p:\exists\ k,\ M'_{p,k}<K_{10}^{-1}\right\}\subset I:=\left\{p:M_p=0\right\}.$$
Suppose there exists $k\geq 1$ such that $M'_{p,k}<\epsilon$. Let $K$ be a constant which bounds from above the number of elements $A\in \mathcal{U}_{i+k}$ such that $A\cap L\cdot B\neq\emptyset$ for all $i\geq 1$ and $B\in\mathcal{U}_i$. Since $M_{p,k}'<\epsilon$, we have $\Mod_p\left(\Gamma_{k}'(B),G_{i+k}\right)<\epsilon$ for all $i\geq 1$ and $B\in\mathcal{U}_i$. 

Let $i\geq 1$ and $B\in\mathcal{U}_i$, we consider $\rho:\mathcal{U}_{i+k}\to\R_+$ an optimal function for $\Gamma_{k}'(B)$. By optimality $\rho(A)=0$ for any $A$ in $\mathcal{U}_{i+k}$ which does not intersects $L\cdot B$. We define $\sigma:\mathcal{U}_{i+k}\to\R_+$ by setting
$$\sigma(A):=\max\left\{\rho(A),\left(\epsilon K^{-1}\right)^{1/p}\right\}.$$
Then $\sigma$ is a $\Gamma_{k}'(B)$-admissible function, bounded from below by $\left(\epsilon K^{-1}\right)^{1/p}$ and with $p$-volume bounded from above by 
$$\sum\limits_{A\in\in\mathcal{U}_{i+k}}\sigma(A)^p\leq \sum\limits_{A\cap L\cdot B\neq\emptyset}\rho(A)^p+\left(\epsilon K^{-1}\right)\cdot \#\left\{A\in \mathcal{U}_{i+k}:A\cap L\cdot B\neq\emptyset\right\}\leq 2\epsilon.$$
This implies that for $q\leq p$, we have
\begin{align*}
\Mod_q\left(\Gamma_{k}'(B),G_{i+k}\right)\leq &\sum\limits_{A\in\mathcal{U}_{i+k}}\sigma(A)^q 
\leq \max\left\{\sigma(A)^{q-p}\right\}\cdot \sum\limits_{A\in\mathcal{U}_{i+k}}\sigma(A)^p\\
\leq &\left(\frac{K}{\epsilon}\right)^{(p-q)/p}\cdot 2\epsilon=\left(\frac{K}{\epsilon}\right)^{1-\frac{q}{p}}\cdot 2\epsilon.
\end{align*}
In particular, $M_{q,k}'<3\epsilon=K_{10}^{-1}$ if $q<p$ is close enough to $p$. Suppose now by contradiction that $M'_{Q_N,k}<\epsilon$. Then there exists $q<Q_N$, close enough to $Q_N$, such that $M'_{q,k}<K_{10}^{-1}$. Therefore, $M_q=0$ which is a contradiction. This finish the proof. 
\end{proof}
A slight modification of the proof of Lemma \ref{indepdel}, shows that there exists an integer $l\geq 1$ and a constant $K_{11}$ such that $M'_{Q_N,k+l}\leq K_{11}\cdot M_{Q_N,k}$. This allows us to prove the following corollary.
\begin{corollary}[Positiveness of the modulus at the critical exponent]\label{modpos}
The sequence of moduli $\left\{M_{Q_N,k}(L)\right\}_{k}$ admits a positive lower bound, which depends only on $L$ and the doubling constant of $X$.
\end{corollary}

This lower bound ---therefore, the sub-multiplicative inequality--- with the following facts: (a) combinational modulus is bounded by the analytical moduli on tangent spaces of $X$, and (b) the critical exponent is equal the AR conformal dimension, give a more conceptual proof of the following theorem of Keith and Laakso (see \cite{KL}):

\begin{corollary}[Keith-Laakso]\label{kl}
Let $X$ be compact and $Q$-regular, $Q>1$, such that $\dim_{AR}X=Q$. Then there exists a weak tangent space $X_\infty$ of $X$, which admits a curve family $\Gamma\subset X_\infty$ of definite diameter and of positive $Q$-analytical modulus.
\end{corollary}
\begin{proof}
Indeed, from Proposition \ref{semicont}, we know that $M_{Q_N}\leq K_9\cdot M_{Q_N}^T$. Since from Corollary \ref{modpos}, we have that $M_{Q_N}>0$, we conclude that there exists a weak tangent space $(X_\infty,x_\infty)$ of $X$ such that the family of curves joining $B(x_\infty,1)$ and $X_\infty\setminus B(x_\infty,2)$ is of positive $Q_N$-modulus.
\end{proof}

\subsection{Some variants\label{secdifermods}}
When the tangent spaces of $X$ are not locally homeomorphic to $X$, the nerves $G_k$, associated to the coverings of $X$, differ from $X$ by approaching its tangent spaces when $k$ becomes large, i.e. small scales. For example, it is usually possible to find curves in $G_k$ that do not exist in $X$.

In this Section, we introduce a second combinatorial modulus $M_{p,k}^X$ defined using curves of $X$. We give topological and metric conditions on $X$ so these two moduli, $M_{p,k}^X$ and $M_{p,k}$, have the same asymptotic behavior when $k$ tends to infinity. This new modulus will allow us to compute the AR conformal dimension of $X$ using ``genuine'' curves of $X$.

Let $\Gamma$ be a curve family in $X$, and let $\mathcal{U}$ be a covering of $X$. For $\gamma\in \Gamma$, we denote by $\mathcal{U}\left(\gamma\right)=\left\{B\in \mathcal{U}:B\cap\gamma\neq\emptyset\right\}$. For each $B\in\mathcal{U}$, we denote $\Gamma(B)$ the family of curves in $X$ which intersect both $\overline{\ell\cdot B}$ and $X\setminus L\cdot B$, where $L\geq 2$ and $\ell:=1+b^{-1}$. Therefore, for all $k\geq 1$ and $B\in\mathcal{U}$, we define the following family of subsets of $G_{|B|+k}$: 
$$\Delta_{k}(B)=\left\{\mathcal{U}_{|B|+k}\left(\gamma\right):\gamma\in \Gamma(B)\right\}.$$
Finally, we set $\Mod_p\left(\Gamma(B),\mathcal{U}_{|B|+k}\right):=\Mod_p\left(\Delta_{k}(B),G_{|B|+k}\right)$. 
\begin{definition}[Combinatorial modulus of curves in the space]\label{modcombx}
We define 
\begin{equation}
M^X_{p,k}\left(L\right)=\sup\limits_{B\in\mathcal{U}}\Mod_p\left(\Gamma(B),\mathcal{U}_{|B|+k}\right).
\end{equation}
The symbol $X$, in the notation, indicates that the modulus is computed using curves of $X$.
\end{definition}

To simplify the notation, we write $M^X_{p,k}$ instead of $M^X_{p,k}\left(L\right)$. The sequence $\{M^X_{p,k}\}$ have the same properties as $\{M_{p,k}\}$: for fixed $k$, the function $p\mapsto M^X_{p,k}$ is non-increasing, and the set of $p\in(0,\infty)$ such that $M^X_p:=\liminf_k M^X_{p,k}=0$ is an interval. We define in the same way the critical exponent $Q_X$.

The first remark, is that the sequence $\{M^X_{p,k}\}_k$ verifies a ``stronger'' sub-multiplicative inequality on $k$: there exists a constant $K\geq 1$, which depends only on $p$, $\kappa$, $L$ and the doubling constant $K_D$, such that
\begin{equation}\label{vsmultiplicative}
M^X_{p,k+l}\leq K\cdot M^X_{p,k}\cdot M^X_{p,l}.
\end{equation}
for all $k$ and $l$. The proof is analogous to that of (\ref{smultiplicative}), noting that here the family of curves $\Gamma(B)$ does not change when scale does; therefore, it is not necessary to consider the modulus $M'_{p,k}$ like before. If we set $\epsilon_L=K^{-1}$, we have (compare with \cite{BK} Section 3)
\begin{equation}\label{modptit}
\lim\limits_{k\to+\infty}M^X_{p,k}=0\Leftrightarrow  \exists\ k\geq 1\text{ tel que }M^X_{p,k}<\epsilon_L,
\end{equation}
and therefore, $M^X_{p,k}\geq \epsilon_L$ for all $k\geq 1$ if $p\in\left(0,Q_X\right]$. We remark that $Q_X\leq Q_N$ always holds, because for any curve $\gamma\in \Gamma(B)$, the subset $\mathcal{U}_{|B|+k}\left(\gamma\right)$ contains a path which belongs to $\Gamma'_{k}(B)$; and therefore, $M^X_{p,k}\leq M'_{p,k}$. In general, it is a strict inequality (see the remark following Theorem \ref{egalite}). 

Recall the following definition. Suppose $X$ is connected, the for $x,y\in X$, define
\begin{equation*}
\delta(x,y):=\inf\{\diam\ J: J\textrm{ is connected with }x,y\in J\}.
\end{equation*}
For $r>0$ let
\begin{equation*}
h(r):=\sup\{\delta(x,y): d(x,y)\leq r\}.
\end{equation*}
We say that $X$ is \emph{locally connected} if $h(r)\to 0$ when $r\to 0$. The function $h$ is called the \emph{modulus of local connectivity}. We say that $X$ is \emph{linearly connected} ---LC for short--- if there exists a constant $K_\ell\geq 1$ such that $h(r)\leq K_\ell r$ for all $0<r\leq \diam X$. Up to changing the constant $K_\ell$, this is equivalent to the following: for any $x,y\in X$, there exists a curve $\gamma$ in $X$ joining them with $\diam\gamma\leq K_\ell d(x,y)$. We can also give the following interpretation: the distance $\delta$ is bi-Lipschitz equivalent to $d$: $d\leq \delta\leq K_\ell d$. For the distance $\delta$, every ball is path-connected.

We also recall that $V_r(A)$ denotes the $r$-neighborhood of $A$. The goal of this section is to prove the following result:

\begin{theorem}\label{egalite}
Let $X$ be a doubling, uniformly perfect, compact metric space. Suppose that $X$ also verifies the following two hypotheses: 
\begin{enumerate}
\item (Uniform linear connectivity of components) There exists a constant $K_\ell\geq 1$ such that any connected component of $X$ is $K_\ell$-linearly connected. 
\item (Uniform separation of components) There exists a constant $K_s\geq 1$ such that: for all $0<r\leq\diam X$, there exists a covering $\mathcal{W}_r$ of $X$, by open and closed sets, such that for all $W\in \mathcal{W}_r$, we have $\mathrm{dist}\left(W,X\setminus W\right)\geq r/K_s$ and there exists a connected component $Y$ of $X$ with $Y\subset W\subset V_r(Y)$.
\end{enumerate}
Then $Q_X=Q_N$. In particular, when $X$ linearly connected the critical exponent $Q_X$ is equal to the AR conformal dimension of $X$.
\end{theorem}
We make some remarks before proving the theorem.
\begin{remark}[Remark 1.]
In general, $Q_X<Q_N$. It's not hard to construct a Cantor set $X$ in the plane $\R^2$ such that $Q_X=0$ and $Q_N=2$. See also Figure \ref{figsepcomp}.
\end{remark}
\begin{remark}[Remark 2.]
The hypothesis of the item (2) is inspired in the analogous notion of uniform disconnectness of David and Semmes \cite{DS}. By compactness, we can always suppose that the covering $\mathcal{W}_r$ is finite.

We can state this condition in the following way. Given $\epsilon>0$, we can define an equivalence relation $\sim_\epsilon$ in $X$, for which two points $x$ and $y$ of $X$ are $\epsilon$-equivalent if they can be connected by an $\epsilon$-chain, i.e. there exists a sequence $\left\{z_i\right\}_{i=1}^N\subset X$ with $z_1=x$, $z_N=y$ and $d(z_i,z_{i+1})\leq \epsilon$ for all $i=1,\ldots,N-1$. 

Each $\epsilon$-class $W$, is open and closed with $\mathrm{dist}\left(W,X\setminus W\right)> \epsilon$. Moreover, if $\epsilon_1\leq \epsilon_2$, and if we denote $W_{\epsilon_i}(x)$, $i=1,2$, the $\epsilon_i$-class which contains $x$, then $W_{\epsilon_1}(x)\subset W_{\epsilon_2}(x)$. Also $\bigcap_{\epsilon>0}W_\epsilon(x)=Y$, where $Y$ is the connected component of $X$ containing $x$. In particular, for all $0<r\leq\diam X$, there exists $\epsilon_r$ such that if $\epsilon\leq \epsilon_r$, then $W_\epsilon(x)\subset V_r(Y)$. 

Condition 2 above is equivalent to the following: for all $\epsilon\in (0,\diam X)$ and all $\epsilon$-class $W$, 
there exists a connected component $Y$ of $X$ such that $W\subset V_{K_s\epsilon}(Y)$. 

In fact, suppose that $X$ verifies condition (2), then we take $r= K_s \epsilon$. If $\mathcal{W}_r$ is a finite covering of $X$ associated to $r$, like in the statement of condition 2, then each element $W$ of $\mathcal{W}_r$ is a union of $\epsilon$-classes, and each one of these classes is in the $r$-neighborhood of the connected component $Y$ of $X$ corresponding to $W$. 

Conversely, since the $\epsilon$-classes form an open covering of $X$ and are pairwise disjoint, for each $\epsilon>0$, there are only finitely many such classes. We denote them by $W_i(\epsilon)$ for $i=1,\ldots, N_\epsilon$. If a component $Y$ of $X$ intersects an $\epsilon$-class $W_i(\epsilon)$, it must be contained in that class. Consider $Y_i$ a connected component of $X$ such that $W_i(\epsilon)$ is in the $K_s\epsilon$-neighborhood of $Y_i$. For each $Y_i$, we consider the open and closed set $U_i$ consisting of the $\epsilon$-classes contained in the $K_s\epsilon$-neighborhood of $Y_i$. Thus, we obtain a covering of $X$, by open and closed subsets $\{U_i\}$, at distance at least $\epsilon$ of their complements, and such that $Y_i\subset U_i\subset V_ {K_s \epsilon}(Y_i)$ for each $i$. We remark that the $U_i$ are not necessarily disjoint.

We end with another formulation of condition 2. For each $\epsilon>0$, each $i\in\{1,\ldots,N_\epsilon\}$ and each component $Y$ of $X$, we set
$$d_Y(\epsilon,i):=\inf\left\{r>0:W_i(\epsilon)\subset V_r(Y)\right\}.$$
For each class $W_i(\epsilon)$, denote
$$r_i(\epsilon):=\inf\left\{d_Y(\epsilon,i):Y\text{ connected component of } X\right\},$$
and finally, define $h:(0,\diam X]\to \R_+$ by setting
\begin{equation}\label{defh}h(\epsilon)=\max\left\{r_i(\epsilon):i=1,\ldots,N_\epsilon\right\}.\end{equation}
The hypothesis says that there exists a uniform constant $K_s$ such that $h(\epsilon)\leq K_s\cdot\epsilon$ for all $0<\epsilon\leq \diam X$ (see also Figure \ref{figsepcomp}). For example, the Cantor set of segments $X:=\mathcal{C}\times[0,1]$, where $\mathcal{C}$ is the standard middle Cantor set, verifies the hypothesis of Theorem \ref{egalite}.

\begin{figure}
\centering
\setlength{\unitlength}{1cm}
\begin{picture}(5,5)
\includegraphics{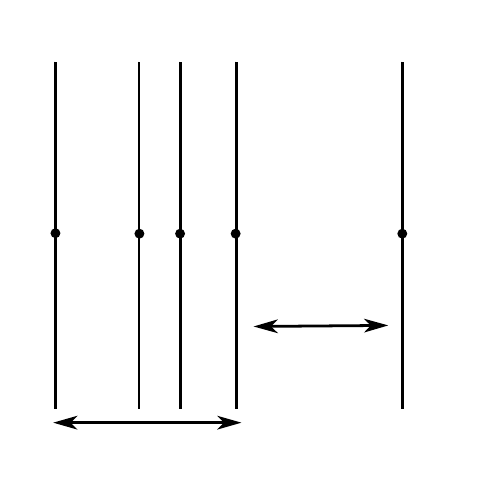}
\put(-2.5,1.3){\tiny $\epsilon=\frac{1}{(n-1)n}$}
\put(-3.8,0.3){\tiny $r=\frac{1}{n}$}
\put(-4.7,2.5){\tiny $0$}
\end{picture}
\caption{\label{figsepcomp} Let $X=\overline{\{1/n:n\geq 1\}}\times [0,1]$. The condition of uniform separation of connected components is not verified. Indeed, if $\epsilon=1/((n-1)n)$, then with $\epsilon$-chains we can connect $0$ to all the $r$-neighborhood of $\{0\}\times [0,1]$, where $r=1/n$. But $r/\epsilon\to\infty$ when $n\to\infty$. This behavior is forbidden by condition 2 of the theorem. For this simple example $Q_X<Q_N$.}
\end{figure}
\end{remark}

\begin{proof}[Proof of Theorem \ref{egalite}]
We must show the inequality $Q_N\leq Q_X$. For $p> 0$, we show that there exist constants $M$ and $k_0$, which depend only on $\lambda$ and the geometry of $X$, such that $M_{p,k}\leq M^{p+1}\cdot M^X_{p,k}$ for all $k\geq k_0$.

First, remark that condition (2) implies that paths of $G_{m}$ are at distance comparable to $b^{-m}$ of genuine curves of $X$. If $B_1\sim B_2$ are two elements of $\mathcal{U}_{m}$, then their centers, which we denote by $z$ and $w$ respectively, verify $d(z,w)< 2\lambda\kappa b^{-m}=\epsilon_m$ where $\epsilon_m:=2\lambda\kappa b^{-m}$. Therefore, if $\gamma=\left\{B_j\right\}_{j=1}^N$ is a path in $G_{m}$, the centers of $B_j$, $z_j$, $j=1,\ldots N$, belong to the same $\epsilon_m$-class $W$ of $X$. Since $h(\epsilon_m)\leq K_s\cdot \epsilon_m$, there exists a connected component $Y$ of $X$ such that $W\subset V_{K_s\epsilon_m}(Y)$. This implies that $\gamma$ is contained in $W\subset V_{\epsilon_m}(Y)$.

So there exists $y_j\in Y$ such that $d\left(y_j,z_j\right)<K_s\cdot \epsilon_m$ for all $j\in\left\{1,\ldots,N\right\}$. In particular, we have $d\left(y_j,y_{j+1}\right)<3K_s\epsilon_m$, and since $Y$ is $K_\ell$-linearly connected, there exists a curve $\gamma_j$ contained in $Y$, joining $y_j$ to $y_{j+1}$, and with diameter bounded from above by $3K_\ell K_s\epsilon_m$. Set $K:=3K_\ell K_s$. Let $\zeta_\gamma=\gamma_1* \cdots*\gamma_{N-1}$ be the concatenation of the curves $\gamma_j$. We write $\zeta_\gamma(1)=y_1$ and $\zeta_\gamma(2)=y_N$.

Let $k\geq 1$, $B\in\mathcal{U}$, and let $\hat{\rho}_B:\mathcal{U}_{|B|+k}\to \R_+$ be an $\Gamma(B)$-admissible function such that
\begin{equation}\label{sumachica2}
\sum\limits_{A\in \mathcal{U}_{|B|+k}}\hat{\rho}_B\left(A\right)^p=\Mod_p\left(\Gamma(B),\mathcal{U}_{|B|+k}\right).
\end{equation}
Take a path $\gamma=\left\{B_j\right\}_{j=1}^N$ of $G_{|B|+k}$ which verifies: $z_1$ belongs to $B$, $z_j$ belongs to $\left(L+1\right)\cdot B$ for $j=2,\dots,N-1$ and $z_N$ does not belong to $\left(L+1\right)\cdot B$. Let $\zeta_\gamma=\gamma_1* \cdots*\gamma_{N-1}$ be the curve constructed before. Write for simplicity $i=|B|$. Since $d(z_1,\zeta_\gamma(1))\leq K\cdot\epsilon_{i+k}$, we have 
$$d(\zeta_\gamma(1),x)\leq d(z_1,x)+K\cdot\epsilon_{i+k}\leq\kappa b^{-i}\left(1+\frac{2\lambda K}{b^k}\right).$$
Therefore, $\zeta_\gamma\cap\ \ell\cdot B\neq\emptyset$ if $k\geq k_0$, where $k_0$ is the smallest integer such that $k_0\geq \log_{b}(2\lambda K)+1$. Also, since 
\begin{equation*}
d\left(\zeta_\gamma(2),x\right)\geq d\left(z_N,x\right)-d\left(\zeta_\gamma(2),z_N\right)\geq \kappa \left(L+1-\frac{2\lambda K}{b^k}\right)b^{-i}> L\kappa b^{-i},
\end{equation*}
we have $\zeta_\gamma\cap X\setminus L\cdot B\neq\emptyset$, and so $\zeta_\gamma\in \Gamma(B)$. For each point $w$ of $\gamma_j$, we have 
\begin{equation*}
d\left(z_j,w\right)\leq \diam\gamma_j+K\epsilon_{i+k}\leq 2K\epsilon_{i+k}\leq\Lambda \kappa b^{-(i+k)},
\end{equation*} 
where $\Lambda\geq4\lambda K$ is a uniform constant which only depends on $\lambda$, $\kappa$, $K_s$ and $K_\ell$. We can suppose $\Lambda$ large enough so that any element $A$ of $\mathcal{U}_{i+k}$, which intersects $\gamma_j$, is contained in $\Lambda \cdot B_{j}$. The same holds for $j+1$.
Define $\rho_B:\mathcal{U}_{i+k}\to \R_+$ by
\begin{equation}
\rho_B\left(A\right)=\max\left\{\hat{\rho}_B\left(C\right):C\subset \Lambda \cdot A\right\}.
\end{equation}
Since the number of elements $C$ of $\mathcal{U}_{i+k}$ which are contained in $\Lambda\cdot A$, is bounded from above by a constant $M$, which depends only on $\Lambda$ and the doubling constant $K_D$, we have
\begin{equation}
\rho_B\left(B_{j}\right)\geq \frac{1}{M}\sum\limits_{C\cap\gamma_j\neq\emptyset}\hat{\rho}_B\left(C\right).
\end{equation}
Therefore,
\begin{equation}
\sum\limits_{j=1}^N\rho_B\left(B_{j}\right)\geq \frac{1}{M}\sum\limits_{C\cap \zeta\neq\emptyset}\hat{\rho}_B\left(C\right)\geq \frac{1}{M}.
\end{equation}
On the other hand, take $M$ large enough so that the number of elements $A$ in $\mathcal{U}_{i+k}$ such that $\Lambda\cdot  A$ contains $C$, is also bounded from above by $M$ for each $C$ in $\mathcal{U}_{i+k}$. $M$ still depends only on $\Lambda$ and the doubling constant. Then the $p$-volume is bounded by
\begin{align*}
\sum\limits_{A\in \mathcal{U}_{i+k}}\rho_B\left(A\right)^p&=\sum\limits_{A\in \mathcal{U}_{i+k}}\max\left\{\hat{\rho}_B\left(C\right)^p:C\subset\Lambda\cdot A\right\}\\
&\leq \sum\limits_{A\in \mathcal{U}_{i+k}}\sum\limits_{C\subset\Lambda\cdot A}\hat{\rho}_B\left(C\right)^p
\leq M\cdot\sum\limits_{C\in \mathcal{U}_{i+k}}\hat{\rho}_B(C)^p\\
&= M\cdot\Mod_p\left(\Gamma(B),\mathcal{U}_{|B|+k}\right).
\end{align*}
That is to say, if we multiply $\rho_B$ by $M$, we obtain a $\Gamma_{k}(B)$-admissible function which has $p$-volume bounded from above by $K\cdot \Mod_p\left(\Gamma(B),\mathcal{U}_{|B|+k}\right)$, where $K:=M^{p+1}$. Therefore, $\Mod_p\left(\Gamma_{k}(B),G_{|B|+k}\right)\leq K\cdot\Mod_p\left(\Gamma(B),\mathcal{U}_{|B|+k}\right)$. This completes the proof of the theorem.
\end{proof}

\subsection{The approximately self-similar case}
We finish by applying this result to the case when the space is approximately self-similar. In this case, we can simplify the definition of $Q_X$ using a family of curves of definite diameter. The following definition appears in \cite{BK}: we say that $X$ is approximately self-similar if there exist constants $c_0>0$ and $L_0\geq 1$ such that for any $0<r\leq\diam X$ and any $x\in X$, there exists an open set $U\subset X$, with $\diam U\geq c_0$, and a $L_0$-bi-Lipschitz map $\phi:\left(B(x,r),\frac{d}{r}\right)\to (U,d)$. This definition implies that $X$ is doubling and uniformly perfect, and that if $X$ is connected and locally connected, then $X$ is LC (see \cite{CaPi} Chapter 2).

Two important classes of approximately self-similar spaces are the boundaries of hyperbolic groups and the Julia sets of hyperbolic rational maps. Other examples indlude the  Sierpi\'nski carpet and gasket, the Menger curve and other classical fractals, which appear as attractors of some Iterated Function Systems.

The following definition appears in \cite{BK} and \cite{HP2}. From now on we suppose $X$ approximately self-similar. For $\delta>0$, denote 
$$\Gamma_\delta=\left\{\gamma\subset X:\diam\gamma\geq \delta\right\},$$
and let $N_{p,k}\left(\delta\right):=\Mod_p\left(\Gamma_\delta,\mathcal{U}_k\right)$. In \cite{BK} Section 3, several important properties of $N_{p,k}(\delta)$ for approximately self-similar sets are proved. In fact, the sequence $\left\{N_{p,k}\right\}_k$ verifies a sub-multiplicative inequality, and there exists $\epsilon_\delta>0$, which depends only on $\delta$ and the doubling constant of $X$, such that
\begin{equation}\label{modptit2}
\lim\limits_{k\to+\infty}N_{p,k}\left(\delta\right)=0\Leftrightarrow  \exists\ k\geq 1\text{ such that }N_{p,k}\left(\delta\right)<\epsilon_\delta.
\end{equation}
Therefore, we can define the \emph{large scale critical exponent} of $X$ by
\begin{equation}\label{expcritique}
Q_D\left(\delta\right)=\inf\left\{p>0:N_{p,k}\left(\delta\right)\to 0,\text{ quand }k\to+\infty\right\}.
\end{equation}
From \cite{HP2} Corollary 3.3, we have $Q_D\left(\delta\right)\leq \dim_{AR}X$ for all $\delta>0$. 
 
\begin{proposition}
Let $X$ be approximately self-similar. There exists $\delta_0>0$, which depends only on the constant $L_0$, such that if $0<\delta\leq\delta_0$, then $Q_X\leq Q_D\left(\delta\right)$.
\end{proposition}
\begin{proof}
We use in the proof various ingredients taken from \cite{BK} Section 3. Take $0<\delta\leq\frac{1}{6L_0}$, and let $p>0$ so that $N_{p,k}\left(\delta\right)\to 0$ when $k\to+\infty$. Let $k\geq 1$, and let $\rho:\mathcal{U}_k\to \R_+$ be a $\Gamma_\delta$-admissible optimal function, i.e. so that $\mathrm{Vol}_p(\rho)=\Mod_p\left(\Gamma_\delta,\mathcal{U}_k\right)=N_{p,k}\left(\delta\right)$.

Let $L\geq 2$. For $B\in\mathcal{U}$, consider $\phi:\left(L+1\right)\cdot B\to U$ the map given by the definition of self-similarity of $X$. We denote $i=|B|$ and
$$V_B=\left\{A\in \mathcal{U}_{i+k}:A\subset \left(L+1\right)\cdot B\right\}.$$
The set $A':=\phi\left(A\right)$ is defined for any $A$ in $V_B$. Since the map $\phi$ is a $L_0$-bi-Lipschitz homeomorphism, from $\left(L+1\right)\cdot B$ with the rescaled distance $(L+1)^{-1}\kappa^{-1} b^id$, into $U$, for any element $A$ of $V_B$, we have:
$$B\left(\phi(x_A),\frac{1}{\left(L+1\right)\kappa^2L_0}b^{-k}\right)\subset \phi\left(B\left(x_A,\kappa^{-1}b^{-(i+k)}\right)\right)\subset A'\subset B\left(\phi(x_A),\frac{L_0}{L+1}b^{-k}\right).$$
Set $\kappa'=\left(L+1\right)\kappa^2L_0$, since the balls $\left\{B\left(\phi(x_A),\kappa'^{-1}\right)\right\}_{A\in V_B}$ are pairwise disjoint, there exists a constant $K\geq 1$, which only depends on $\kappa'$ and the doubling constant of $X$, such that:
\begin{equation}\label{k1}
\forall C\in \mathcal{U}_k:\ \ \#\left\{A \in V_B:A'\cap C\neq\emptyset\right\}\leq K.
\end{equation}
Define $\sigma:\mathcal{U}_{i+k}\to\R_+$ by
$$\sigma\left(A\right)=\begin{cases}
\max\left\{\rho\left(C\right):C\cap A'\neq\emptyset\right\} & \text{ if } A\in V_B.\\
0 & \text{ otherwise.}
\end{cases}$$
Then
\begin{align*}
\sum\limits_{A\in \mathcal{U}_{i+k}}\sigma\left(A\right)^p=&\sum\limits_{A\in V_B}\sigma\left(A\right)^p\leq \sum\limits_{A\in V_B}\sum\limits_{C\cap A'\neq\emptyset}\rho\left(C\right)^p\\
\leq& K\sum\limits_{C\in \mathcal{U}_k}\rho\left(C\right)^p= K\cdot \Mod_p\left(\Gamma_\delta,\mathcal{U}_k\right).
\end{align*}

We recall that $\ell=1+b^{-1}$ comes from Definition \ref{modcombx}. Let $\gamma\subset X$ be a curve that $\gamma\cap \ell\cdot B\neq\emptyset$ and $\gamma\cap X\setminus L\cdot B\neq\emptyset$. We can suppose $\gamma$ to be contained in $\overline{L\cdot B}$. Since the diameter of $\gamma$ is bounded from below by $\left(L-\ell\right)\kappa b^{-i}$, the diameter of $\phi\left(\gamma\right)$ is bounded from below by $\frac{L-\ell}{\left(L+1\right)L_0}\geq\frac{1}{6L_0}\geq\delta$. Thus, $\phi\left(\gamma\right)$ is a curve in $\Gamma_\delta$ and
$$\sum\limits_{A\cap\phi\left(\gamma\right)\neq\emptyset}\rho(A)\geq 1.$$
An element $A$ belongs to $V_B$ if $A\cap \overline{L\cdot B}\neq\emptyset$. Then, for any element $C$ of $\mathcal{U}_k$ which intersects $\phi\left(\gamma\right)$, there exists an element $A$ of $V_B$ such that $C\cap A'\cap\phi\left(\gamma\right)\neq\emptyset$. Thus, for any element $C$ of $\mathcal{U}_k\left(\phi\left(\gamma\right)\right)$, there exists an element $A_C$ of $\mathcal{U}_{i+k}\left(\gamma\right)$ such that $\rho\left(C\right)\leq \sigma\left(A_C\right)$.

We can suppose $K$ big enough so that
$$\forall A\in \mathcal{U}_{i+k}\left(\gamma\right),\ \ \#\left\{C\in \mathcal{U}_k\left(\phi\left(\gamma\right)\right):A_C=A\right\}\leq K,$$
since this quantity only depends on the doubling constant of $X$.
Then 
$$1\leq \sum\limits_{C\in \mathcal{U}_k\left(\phi\left(\gamma\right)\right)}\rho\left(C\right)\leq \sum\limits_{C\in \mathcal{U}_k\left(\phi\left(\gamma\right)\right)}\sigma\left(A_C\right)\leq K\sum\limits_{A\in \mathcal{U}_{i+k}\left(\gamma\right)}\sigma\left(A\right).$$
This shows that $M^X_{p,k}\left(L\right)\leq K^{1+p}N_{p,k}\left(\delta\right)$, and therefore, $M^X_{p,k}\left(L\right)\to 0$ when $k\to \infty$. This completes the proof of proposition.
\end{proof}

So to estimate the Ahlfors regular conformal dimension of an approximately self-similar space, we just need to look at the modulus of curves of definite diameter.

\begin{corollary}\label{kk}
Let $X$ be an approximately self-similar space. If $X$ verifies items 1 and 2 of Theorem \ref{egalite}, then $\dim_{AR}X=Q_D\left(\delta\right)$ for all $0<\delta\leq \delta_0$. This is the case, in particular, when $X$ connected and locally connected.
\end{corollary}
\begin{proof}
Since $Q_X\leq Q_D\leq \dim_{AR}X$, it suffices to show that $Q_X=\dim_{AR}X$, but this is true from Theorem \ref{egalite}.
\end{proof}

We say that the diameter of the connected components of $X$ tends to zero, if for all $\delta>0$, there are only finitely many connected components of $X$ which have diameter greater than or equal to $\delta$. By convention, we define the AR conformal dimension of a point set to be zero. It's not clear in general whether the AR conformal dimension behaves well under countable unions, see for example Figure \ref{figsepcomp} (see also \cite{MT} for a discussion on this problem). The next corollary is a positive result in this direction.

\begin{corollary}\label{supcomposantesqas}
Let $X$ be approximately self-similar which verifies items 1 and 2 of Theorem \ref{egalite}. Suppose that the diameter of the connected components of $X$ tends to zero. Then
$$\dim_{AR}X=\sup\left\{\dim_{AR}Y:Y\text{ connected component of }X\right\}.$$
\end{corollary}
\begin{proof}
We remark that $\dim_{AR}X\geq \dim_{AR}Y$ for any connected component $Y$ of $X$. If $Y$ is a point, the inequality trivially holds. Otherwise, $Y$ is doubling and uniformly perfect: therefore, its AR conformal dimension is equal to the Assouad conformal dimension (see \cite{MT}). For the Assouad conformal dimension the inequality is clear.

We show the other inequality. Set
$$q:=\sup\left\{\dim_{AR}Y:Y\text{ connected component of }X\right\},$$
and let $p>q$. We can suppose that $q\geq 1$, otherwise, any connected component is a singleton, and since $X$ verifies the uniform separation of components, it is uniformly disconnected. In that case $q=\dim_{AR}X=0$.

We know that there exists $\delta>0$ such that $\dim_{AR}X=Q_D(\delta)$. Consider the set $\mathcal{Y}$ of connected components of $X$ which have diameter bigger than or equal to $\delta$. By hypothesis, $\mathcal{Y}$ is finite, and we write
$$N_\delta:=\left|\mathcal{Y}\right|.$$
If $Y$ is a component of $X$, we denote by $\Gamma_\delta(Y)$ the curves of $\Gamma_\delta$ which are contained in $Y$. Therefore,
$$\Gamma_\delta=\bigcup\limits_{Y\in\mathcal{Y}}\Gamma_\delta(Y),$$
and consequently, for all $k\geq 1$, we have
$$\Mod_p\left(\Gamma_\delta,\mathcal{U}_k\right)\leq \sum\limits_{Y\in\mathcal{Y}}\Mod_p\left(\Gamma_\delta(Y),\mathcal{U}_k\right)\leq N_\delta\cdot\max\limits_{Y\in\mathcal{Y}}\left\{\Mod_p\left(\Gamma_\delta(Y),\mathcal{U}_k\right)\right\}.$$
If we denote by $\mathcal{U}_k(Y)$ the set of elements $B$ of $\mathcal{U}_k$ which intersect $Y$, we have the following equality
$$\Mod_p\left(\Gamma_\delta(Y),\mathcal{U}_k\right)=\Mod_p\left(\Gamma_{\delta}(Y),\mathcal{U}_k(Y)\right).$$
Fix now $Y\in\mathcal{Y}$. For each element $B$ of $\mathcal{U}_k(Y)$, consider a point $x'\in B\cap Y$ and let $B'=B\left(x',2\kappa b^{-k}\right)$; if the point $x_B$ already belongs to $Y$, we take $x'=x_B$.

Let $\mathcal{W}_k$ be the covering of $Y$ by these balls. From Proposition B.2 of \cite{H1}, the sequence of moduli $\Mod_p\left(\Gamma_\delta(Y),\mathcal{W}_k\right)$ tends to zero when $k$ tends to infinity. Note that
$$\Mod\left(\Gamma_\delta(Y),\mathcal{U}_k(Y)\right)\lesssim \Mod_p\left(\Gamma_\delta(Y),\mathcal{W}_k\right),$$
where the comparison constant depends only on the doubling constant of $X$. Since $\mathcal{Y}$ is finite, we obtain $\Mod_p\left(\Gamma_\delta,\mathcal{U}_k\right)\to 0$ when $k\to +\infty$. Therefore, $\dim_{AR}X\leq p$. This ends the proof of the corollary.
\end{proof}
\begin{remark}
The assumption of finiteness of connected components of definite diameter is necessary, as it's shown by the example of a Cantor set of segments $X:=\mathcal{C}\times [0,1]$, whose AR conformal dimension is equal $1+\dim_H\mathcal{C}=1+\log_3(2)> 1$, although the AR conformal dimension of the connected components of $X$ is equal to $1$, and that of $\mathcal{C}$ is equal to $0$.
\end{remark}


\begin{thebibliography}{21}

\bibitem[Ah73]{Ah} L. V. Ahlfors, \emph{Conformal invariants : topics in geometric function theory}, McGraw-Hill Series in Higher Mathematics. McGraw-Hill Book Co., New York-Düsseldorf-Johannesburg, 1973.

\bibitem[BT01]{BT} C. Bishop, J. Tyson, \emph{Conformal dimension of the antenna set}, Proc. Amer. Math. Soc. 129:3631 3636, 2001.

\bibitem[B06]{B} M. Bonk, \emph{Quasiconformal geometry of fractals}, Proceedings Internat. Congress Math. (Madrid, 2006), Europ. Math. Soc., Zürich, 2006, 1349-1373. 

\bibitem[BHK01]{BHK} M. Bonk, J. Heinonen, P. Koskela, \emph{Uniformizing Gromov hyperbolic spaces.} Astérisque 270 (2001).

\bibitem[BonK02a]{BK3} M. Bonk, B. Kleiner, \emph{Quasisymmetric parametrizations of two-dimensional metric spheres}, Invent. Math. 150 (2002), 127-183. 

\bibitem[BonK02b]{BK2} M. Bonk, B. Kleiner, \emph{Rigidity for quasi-Möbius group actions}, J. Diff. Geom. 61  (2002), 81-106.

\bibitem[BonK05]{BK1} M. Bonk, B. Kleiner, \emph{Conformal dimension and Gromov hyperbolic groups with 2-sphere boundary}, Geom. Topol. 9 219-246, 2005.

\bibitem[Bou95]{Bou} M. Bourdon, \emph{Au bord de certains polyèdres hyperboliques}, Ann. Inst. Fourier (Grenoble) 45 (1995), no. 1, 119-141.

\bibitem[BouK09]{BK} M. Bourdon, B. Kleiner, \emph{Combinatorial modulus, the Combinatorial Loewner Property, and Coxeter groups}, preprint May 2009.

\bibitem[BouK12]{BouK} M. Bourdon, B. Kleiner, \emph{Some applications of $\ell_p$-cohomology to boundaries of Gromov hyperbolic spaces}, 2012.

\bibitem[BP03]{BP} M. Bourdon, H. Pajot, \emph{Cohomologie $\ell_p$ et espaces de Besov}, Journal für die Reine und Angewandte Mathematik Volume 558 (2003), 85-108.

\bibitem[BH99]{BH} M. Bridson, A. Haefliger, \emph{Metric spaces of non-positive curvature}, vol. 319 of Grundlehren der Mathematischen Wissenschaften, Springer-Verlag, Berlin, 1999.

\bibitem[Can94]{Can} J. Cannon, \emph{The combinatorial Riemann mapping theorem}, Acta Math. 173 (1994), no. 2, 155-234.

\bibitem[Ca11]{CaPi} M. Carrasco Piaggio, \emph{Jauge conforme des espaces m\'etriques compacts}, PhD Thesis, Aix-Marseille Université, 2011. \url{http://tel.archives-ouvertes.fr/tel-00645284}

\bibitem[Ca12]{CaPi2} M. Carrasco Piaggio, \emph{Spaces of conformal dimension one}, in preparation. 

\bibitem[CDP90]{CDP} M. Coornaert, T. Delzant, A. Papadopoulos, \emph{Géométrie et théorie des groupes.} No. 1441 Lecture Notes in Mathematics. Springer-Verlag, Berlin, 1990.

\bibitem[Coo93]{Coo} M. Coornaert, \emph{Mesures de Patterson-Sullivan sur le bord d'un espace hyperbolique au sens de Gromov}, Pacific J. Math. 159(1993), 241-270.

\bibitem[Chr90]{Chr} M. Christ, \emph{A $T(b)$ theorem with remarks on analytic capacity and the Cauchy integral}, Colloq. Math. 60/61:2 (1990), 601-628.

\bibitem[DS97]{DS} G. David, S. Semmes, \emph{Fractured Fractals and Broken Dreams}, Oxford Lecture Series in Mathematics and its Applications, 7, 1997.

\bibitem[Elek97]{Elek} G. Elek, \emph{The $\ell_p$-cohomology and the conformal dimension of hyperbolic cones}, Geom. Ded. 68 (1997), 263-279.

\bibitem[GH90]{GH} E. Ghys, P. de la Harpe, \emph{Sur les Groupes Hyperboliques d'après Mikhael Gromov}. Progress in Mathematics. Birkhäuser, Boston-Basel-Berlin, 1990.

\bibitem[Hei01]{He} J. Heinonen, \emph{Lectures on analysis on metric spaces}, Universitext. Springer-Verlag, New York, 2001.

\bibitem[Haïss08]{H2} P. Haïssinsky, \emph{Géométrie quasiconforme, analyse au bord des espaces métriques hyperboliques et rigidités, d'après Mostow, Pansu, Bourdon, Pajot, Bonk, Kleiner...} Séminaire Bourbaki 60, 2007/08, n 993

\bibitem[Haïss09]{H1} P. Haïssinsky, \emph{Empilements de cercles et modules combinatoires}, Ann. Inst. Fourier 59 (2009), no. 6, 2175-2222.

\bibitem[HP09]{HP1} P. Haïssinsky, K. Pilgrim, \emph{Coarse Expanding Conformal Dynamics.} Astérisque 325 2009.

\bibitem[HP08]{HP2} P. Haïssinsky, K. Pilgrim, \emph{Thurston obstructions and Ahlfors regular conformal dimension.} J. Math. Pures Appl. (9) 90, 3 (2008), 229-241.

\bibitem[KL04]{KL} S. Keith, T. Laakso, \emph{Conformal Assouad dimension and modulus.} Geom. Funct. Anal. 14, 6 (2004), 1278-1321.

\bibitem[Kle06]{K} B. Kleiner, \emph{The asymptotic geometry of negatively curved spaces: uniformization, geometrization, and rigidity}, International Congress of Mathematicians. Vol. II, pp. 743-768, Eur. Math. Soc., Zürich, 2006.

\bibitem[Kov06]{Ko} L. Kovalev, \emph{Conformal dimension does not assume values between zero and one.} Duke Math. J. 134, 1 (2006), 1-13.

\bibitem[LP04]{LP} G. Lupo-Krebs, H. Pajot, \emph{Dimensions conformes, espaces Gromov-hyperboliques et ensembles autosimilaires}, Séminaire de Théorie spectrale et géométrie, tome 22 (2003-2004), 153-182.

\bibitem[MT10]{MT} J. Mackay, J. Tyson, \emph{Conformal dimension: theory and application}, University Lecture Series, AMS 2010.

\bibitem[Pan89]{Pa} P. Pansu, \emph{Dimension conforme et sphère à l'infini des variétés à courbure négative.} Ann. Acad. Sci. Fenn. Ser. A I Math. 14 (1989), 177-212.

\bibitem[S01]{Sem} S. Semmes, \emph{Metric Spaces and Mappings Seen at Many Scales}, Appendix B 401-519 de M. Gromov, \emph{Metric structures for Riemannian and Non-Riemmannian spaces} Birkhauser 2001.

\bibitem[TV]{TV} P. Tukia, J. Vaïsälä, \emph{Quasisymmetric embeddings of metric spaces.} Ann. Acad. Sci. Fenn. Ser. A I Math. 5 (1980), 97-114.

\bibitem[Tys98]{Tys} J. Tyson, \emph{Quasiconformality and quasisymmetry in metric measure spaces}, Ann. Acad. Sci. Fenn. Math. 23(1998), 525-548.

\bibitem[VK88]{VK} A. Vol'berg, S. Konyagin, \emph{On measures with the doubling condition.} Math.
USSR-Izv. 30, 1988, 629-638.

\bibitem[Wu98]{W} J. Wu, \emph{Hausdorff dimension and doubling measures on metric spaces.} Proc. Amer. Math. Soc. 126 (1998), 1453-1459.


\end{thebibliography}
\end{document}